\documentclass[a4paper,11pt]{article}

\usepackage[T1]{fontenc}
\usepackage{lmodern}

\usepackage{dsfont}

\usepackage[latin1]{inputenc}
\usepackage{amsmath}
\usepackage{amsthm}
\usepackage{amssymb}
\usepackage{mathrsfs}
\usepackage{graphicx}
\usepackage[all]{xy}
\usepackage{hyperref}
\usepackage{xcolor}

\usepackage{makeidx}

\usepackage{stmaryrd}
\usepackage{caption}

\usepackage{abstract} 

\newtheorem{thm}{Theorem}[section]
\newtheorem{cor}[thm]{Corollary}
\newtheorem{claim}[thm]{Claim}
\newtheorem{fact}[thm]{Fact}

\newtheorem{lemma}[thm]{Lemma}
\newtheorem{prop}[thm]{Proposition}

\theoremstyle{definition}
\newtheorem{definition}[thm]{Definition}
\newtheorem{ex}[thm]{Example}

\newtheorem{question}[thm]{Question}
\newtheorem{problem}[thm]{Problem}

\newcommand{\Conv}{\mathop{\scalebox{1.5}{\raisebox{-0.2ex}{$\ast$}}}}

\newcommand{\lk}{\mathrm{lk}}

\def\rquotient#1#2{%
	\makeatletter
	\raise.3ex\hbox{$#1$}/\lower.3ex\hbox{$#2$}%
	\makeatother
}	

\newcounter{ccomments}

\newcounter{ecomments}

\newcounter{acomments}

\makeatletter
\newcommand{\subjclass}[2][2010]{%
	\let\@oldtitle\@title%
	\gdef\@title{\@oldtitle\footnotetext{#1 \emph{Mathematics subject classification.} #2}}%
}
\newcommand{\keywords}[1]{%
	\let\@@oldtitle\@title%
	\gdef\@title{\@@oldtitle\footnotetext{\emph{Key words and phrases.} #1.}}%
}
\makeatother

\newcommand{\Addresses}
{\bigskip
		\small

\noindent  \textsc{Brandeis University, Waltham, MA, USA} \par\noindent \nopagebreak
		\textit{E-mail address}: \texttt{carolynabbott@brandeis.edu}

\medskip
        
\noindent  \textsc{University of Montpellier,  
Institut Math\'ematiques Alexander Grothendieck.\\ 
Place Eug\`ene Bataillon, 
34090 Montpellier,  France} \newline \nopagebreak
		\textit{E-mail address}: \texttt{anthony.genevois@umontpellier.fr}
		
\medskip

\noindent
\textsc{Memorial University of Newfoundland, St. John's, NL, Canada} \newline \nopagebreak
		\textit{E-mail address}: \texttt{emartinezped@mun.ca}

}

\makeindex

\title{Homotopy types of complexes of hyperplanes in quasi-median graphs and applications to right-angled Artin groups}
\date{\today}
\author{Carolyn Abbott, Anthony Genevois, and Eduardo Mart\'inez-Pedroza}

\subjclass{Primary 20F65. Secondary 05C25, 57Q05.}
\keywords{Quasi-median graphs, median graphs, right-angled Artin groups, graph products of groups, homotopy of simplicial complexes}

\begin{document}

\maketitle

\begin{abstract}
In this article, we prove that, given two finite connected graphs $\Gamma_1$ and $\Gamma_2$, if the two right-angled Artin groups $A(\Gamma_1)$ and $A(\Gamma_2)$ are quasi-isometric, then the infinite pointed sums $\bigvee_\mathbb{N} \Gamma_1^{\bowtie}$ and $\bigvee_\mathbb{N} \Gamma_2^{\bowtie}$ are homotopy equivalent, where $\Gamma_i^{\bowtie}$ denotes the simplicial complex whose vertex-set is $\Gamma_i$ and whose simplices are given by joins. These invariants are extracted from a study, of independent interest, of the homotopy types of several complexes of hyperplanes in quasi-median graphs (such as one-skeleta of CAT(0) cube complexes). For instance, given a quasi-median graph $X$, the \emph{crossing complex} $\mathrm{Cross}^\triangle(X)$ is the simplicial complex whose vertices are the hyperplanes (or $\theta$-classes) of $X$ and whose simplices are collections of pairwise transverse hyperplanes. When $X$ has no cut-vertex, we show that $\mathrm{Cross}^\triangle(X)$ is homotopy equivalent to the pointed sum of the links of all the vertices in the prism-completion $X^\square$ of $X$.  
\end{abstract}

\tableofcontents

\newpage
\section{Introduction}

\noindent
Given a graph $\Gamma$, the \emph{right-angled Artin group} $A(\Gamma)$ is given by the presentation
$$\langle u \text{ a vertex of } \Gamma \mid [u,v]=1 \text{ whenever $u$ and $v$ are adjacent in $\Gamma$} \rangle.$$
Right-angled Artin groups interpolate between free groups (when ``nothing commutes'', i.e.\ $\Gamma$ has no edges) and free abelian groups (when ``everything commutes'', i.e.\ $\Gamma$ is a complete graph). {Despite the early appearance of similar interpolations for other algebraic structures, such as monoids and algebras (see for instance \cite{MR239978, MR575780, Bergman}), as well as a few mentions of the groups in some articles \cite{MR300879, MR567067, MR629331}, the study of right-angled Artin groups  started in earnest with the work of Baudisch \cite{MR463300, MR634562} and Droms \cite{MR2633165, MR880971, MR891135, MR910401}.} Since then, many articles have been dedicated to right-angled Artin groups from various perspectives. In addition to being an instructive source of examples, right-angled Artin groups also led to interesting applications. Two major contributions in this direction include the Morse theory developed in \cite{MR1465330} and the theory of special cube complexes introduced in \cite{MR2377497}. 

\medskip \noindent
Despite the fact that right-angled Artin groups have been extensively studied, a few very natural questions remain open, including the embedding problem, the classification up to commensurability, and, from the perspective of geometric group theory, the following question to which our article is dedicated:
\begin{description}
    \item[(Quasi-isometry classification)] Given two graphs $\Gamma_1$ and $\Gamma_2$, when are $A(\Gamma_1)$ and $A(\Gamma_2)$ quasi-isometric?
\end{description}
Two early major contributions to the subject include \cite{MR2376814}, which proves that right-angled Artin groups given by finite trees of diameters $\geq 3$ are all quasi-isometric; and \cite{MR2421136}, which shows that two right-angled Artin groups given by \emph{atomic} graphs are quasi-isometric if and only if they are isomorphic. Thus, right-angled Artin groups may exhibit both some flexibility and some rigidity up to quasi-isometry. So far, no global picture seems to be accessible. It is worth noticing that \cite{MR2376814} has been generalised in \cite{MR2727658}; and that important works of Huang \cite{HuangI, HuangII} also generalise \cite{MR2421136}, determining exactly when two right-angled Artin groups are quasi-isometric for a wide range of graphs. Many quasi-isometric invariants have also been computed for right-angled Artin groups, including, for instance, divergence \cite{MR2874959, MR3073670}, cut-points in asymptotic cones \cite{MR2874959}, splittings over cyclic subgroups \cite{MR4186478}, and Morse boundaries \cite{MR4563334}. 

\medskip \noindent
In this article, we contribute to the quasi-isometry classification of right-angled Artin groups by constructing new quasi-isometric invariants. More precisely, the main result of this article is:

\begin{thm}\label{thm:IntroQI}
Let $\Gamma_1,\Gamma_2$ be two finite connected graphs. If the right-angled Artin groups $A(\Gamma_1)$ and $A(\Gamma_2)$ are quasi-isometric, then the infinite pointed sums $\bigvee_{\mathbb{N}} \Gamma_1^{\bowtie}$ and $\bigvee_\mathbb{N} \Gamma_2^{\bowtie}$ are homotopy equivalent, where $\Gamma_i^{\bowtie}$ denotes the simplicial complex whose vertex-set is $\Gamma_i$ and whose simplices are given by  subsets of vertices contained in  joins.
\end{thm}

\noindent
 Recall that a graph $\Gamma$ is a \emph{join} if its vertex set decomposes as the disjoint union of two non-empty subsets $A \sqcup B$ such that every vertex of $A$ is adjacent to every vertex of $B$.  

\medskip \noindent
As a baby example illustrating the theorem, notice that the right-angled Artin groups $A(C_4)$ and $A(C_5)$, associated to cycles of lengths $4$ and $5$, are not quasi-isometric. Indeed, $C_4^{\bowtie}$ is contractible, while $C_5^{\bowtie}$ is homotopy equivalent to a circle. As more interesting examples, let $\Gamma_1,\Gamma_2,\Gamma_3,\Gamma_4$ be the graphs given below. By noticing that $\Gamma_1^{\bowtie}$ is homotopy equivalent to a point, $\Gamma_2^{\bowtie}$ to $\mathbb{S}^1$, $\Gamma_3^{\bowtie}$ to $\mathbb{S}^2 \vee \mathbb{S}^1$, and $\Gamma_4^{\bowtie}$ to $\mathbb{S}^2$, it follows from Theorem~\ref{thm:IntroQI} that the right-angled Artin groups $A(\Gamma_1), A(\Gamma_2), A(\Gamma_3), A(\Gamma_4)$ are pairwise not quasi-isometric\footnote{Part of this conclusion can also be deduced from \cite{HuangI, HuangII}, because we tried to present examples that are as simple as possible. However, we can slightly perturb  our examples  so that \cite{HuangI, HuangII} no longer applies while Theorem~\ref{thm:IntroQI} still does, for instance by gluing an edge to each graph along a vertex.}.


\begin{center}
\includegraphics[width=0.5\linewidth]{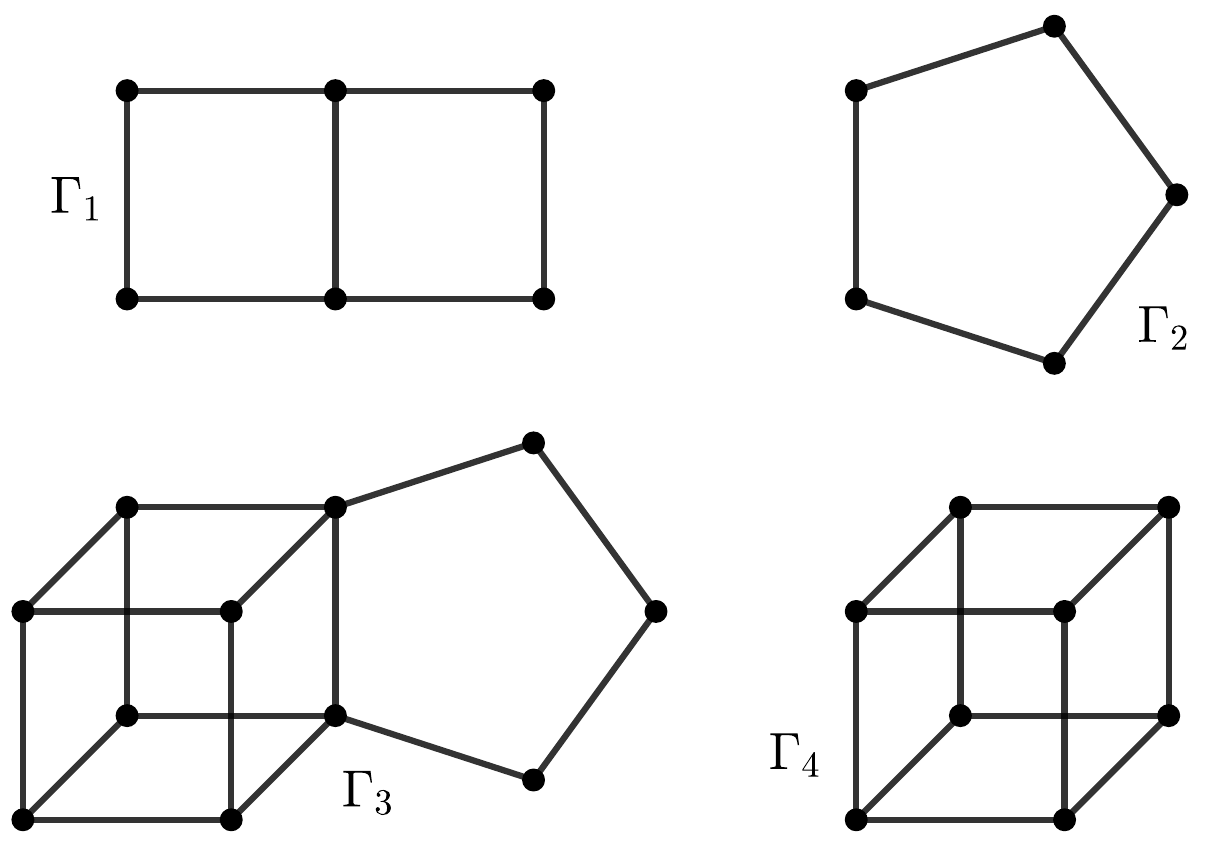}
\end{center}

\noindent
The quasi-isometry invariant given in Theorem~\ref{thm:IntroQI} expands the classification of right-angled Artin groups up to quasi-isometry.  In particular, it can be used to distinguish between infinitely many right-angled Artin groups that cannot be distinguished using \cite{HuangI,HuangII}.

\begin{cor}\label{cor:InjMap}
    There is an injective map from the set of isomorphism classes of finitely presented, one-ended groups to the set of quasi-isometry classes of right-angled Artin groups.  Moreover, for each one-ended finitely presented group, this map produces a right-angled Artin group that does not fit into the framework of  \cite{HuangI, HuangII}. 
\end{cor}

\noindent
See Section~\ref{section:ManyBowtie}, and especially Examples~\ref{ex:BowtieOne} and~\ref{ex:BowtieTwo}, for more details.

\paragraph{Strategy of the proof.} In order to prove Theorem~\ref{thm:IntroQI}, we start by proving that a specific \emph{coset intersection complex} \cite{CIC} is preserved by quasi-isometries. More precisely, given a graph $\Gamma$, define $J(\Gamma)$ as the simplicial complex  
\begin{itemize}
	\item whose vertices are the cosets $g \langle \Lambda \rangle$ where $g \in A(\Gamma)$ and where $\Lambda \subset \Gamma$ is a join;
	\item and whose simplices are given by cosets $g_1 \langle \Lambda_1 \rangle, \ldots, g_n \langle \Lambda_n \rangle$ such that the intersection $g_1 \langle \Lambda_1 \rangle g_1^{-1} \cap \cdots \cap g_n\langle \Lambda_n \rangle g_n^{-1}$ is infinite.
\end{itemize}
One can deduce from the techniques introduced in \cite{HuangI} that, given two finite connected graphs $\Gamma_1$ and $\Gamma_2$, every quasi-isometry $A(\Gamma_1) \to A(\Gamma_2)$ naturally induces an isomorphism $J(\Gamma_1) \to J(\Gamma_2)$. However, determining whether two such complexes are isomorphic may be quite tricky in practice. Our main observation is that it is much simpler to compare them up to homotopy equivalence. More precisely, we show the complex $J(\Gamma)$ is homotopy equivalent to the infinite pointed sum $\bigvee_\mathbb{N} \Gamma^{\bowtie}$, as defined in Theorem~\ref{thm:IntroQI}. 

\medskip \noindent
The key observation in order to identify the homotopy type of our complex is that $J(\Gamma)$ has a very convenient geometric interpretation. Roughly, it is homotopy equivalent to a simplicial complex that can be described using the \emph{hyperplanes} of the \emph{quasi-median graph}
$$\mathrm{QM}(\Gamma):= \mathrm{Cayl} \left( A(\Gamma), \bigcup\limits_{u \in \Gamma} \langle u \rangle \backslash \{1\} \right),$$
as introduced in \cite{QM}. This provides a good geometric framework in which various tools are available in order to study the homotopy type of the simplicial complex we are interested in. As we now discuss, our results in this direction are of independent interest.

\paragraph{Complexes of hyperplanes in quasi-median graphs.} In the same way that \emph{median graphs} (also known as one-skeleta of \emph{CAT(0) cube complexes} in geometric group theory) can be defined as retracts of hypercubes, \emph{quasi-median graphs} can be defined as retracts of \emph{Hamming graphs} (i.e.\ products of complete graphs). See Section~\ref{section:QM} for some background on the subject from the perspective of metric graph theory and geometric group theory. 

\medskip \noindent
It is worth mentioning that, similarly to median graphs, quasi-median graphs can be endowed with a higher-dimensional cellular structure that turns them into CAT(0) complexes. More specifically, given a quasi-median graph $X$, a \emph{prism} refers to a product of \emph{cliques} (i.e.\ maximal complete subgraphs). The \emph{prism-completion} $X^\square$ of $X$, which is obtained by filling all the prisms with products of simplices, can be naturally endowed with a CAT(0) metric \cite{QM}. 

\medskip \noindent
A fundamental tool in the study of quasi-median graphs is the notion of a \emph{hyperplane} (or, in the language of metric graph theory, of a \emph{$\theta$-class} with respect to the \emph{Djokovi\'c-Winkler relation} $\theta$). In a quasi-median graph, a hyperplane is an equivalence class of edges with respect to the reflexive-transitive closure of the relation that identifies two edges whenever they are opposite in a $4$-cycle or belong to a common $3$-cycle; see Figure~\ref{figure3Intro}. The key idea to keep in mind is that the geometry of a quasi-median graph reduces to the combinatorics of its hyperplanes; see, for instance, Theorem~\ref{thm:BigQM} below. 

\begin{figure}[ht!]
\begin{center}
\includegraphics[width=0.5\linewidth]{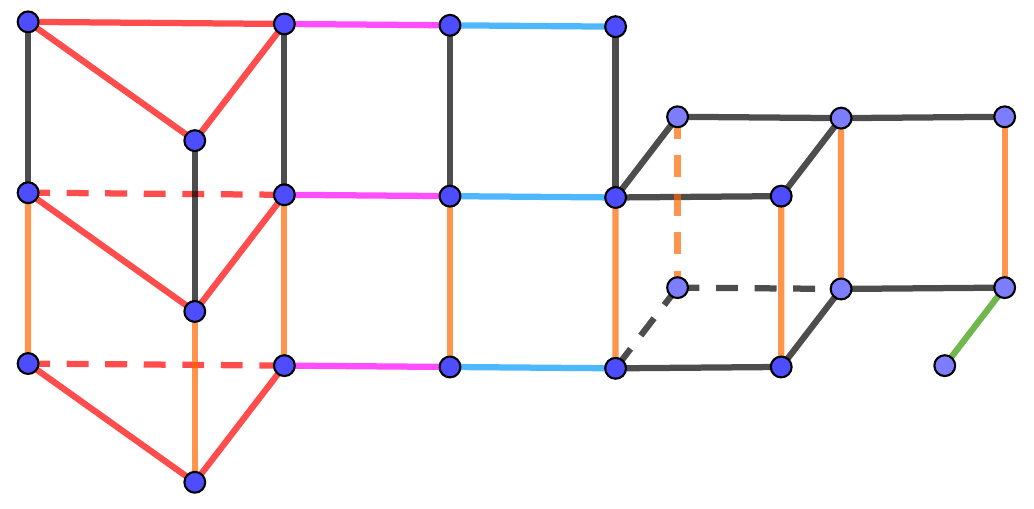}
\caption{A quasi-median graph and some of its hyperplanes. The orange hyperplane is transverse to the red, blue, and purple hyperplanes. The purple hyperplane is in contact with the red, blue, and orange hyperplanes. The orange and green hyperplanes are in contact. The red and blue hyperplanes are not in contact.}
\label{figure3Intro}
\end{center}
\end{figure}

\noindent
Given a quasi-median graph $X$, we refer vaguely to a \emph{complex of hyperplanes} as a simplicial complex whose vertices are the hyperplanes of $X$ and whose simplices correspond to collections of hyperplanes that interact in some specific way. For instance,  two hyperplanes $J$ and $H$ are \emph{in contact} (resp.\ \emph{transverse}) whenever there exist two distinct edges in $J$ and $H$ that intersect (resp.\ that intersect and span a $4$-cycle). The \emph{contact complex} $\mathrm{Cont}^\triangle(X)$ (resp.\ the \emph{crossing complex} $\mathrm{Cross}^\triangle(X)$) is the simplicial complex whose vertices are the hyperplanes of $X$ and whose simplices are given by hyperplanes pairwise in contact (resp.\ transverse). See Section~\ref{section:ComplexesHyp} for a more detailed account on contact and crossing complexes. Of course, other natural constructions are possible, including the \emph{contiguity complex} we introduce in Section~\ref{section:RelativeContact}. {In order to be able to deal with all these complexes at once, we introduce the following notion of \emph{contact complexes relative to collections of gated subgraphs}. We refer the reader to Section~\ref{section:QM} for a precise definition of \emph{gated subgraphs}; for now, they can be thought of as subgraphs satisfying a strong convexity condition.}

\begin{figure}[ht!]
\begin{center}
\includegraphics[width=0.7\linewidth]{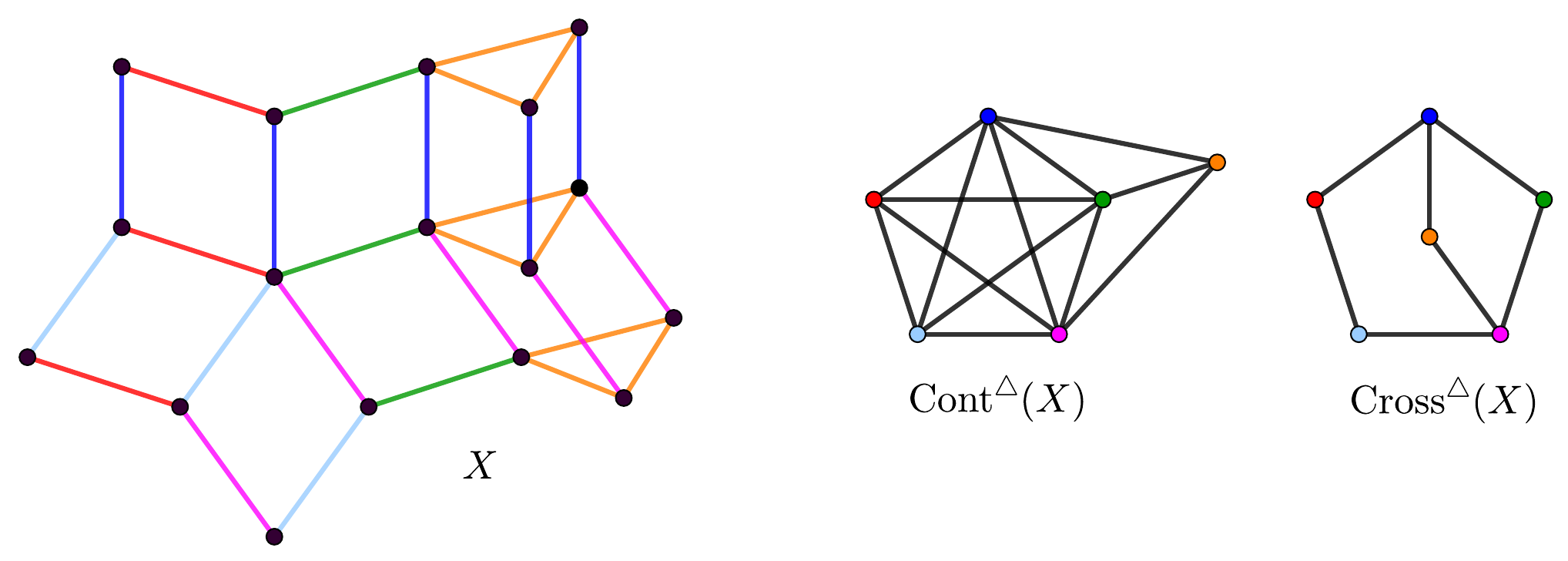}
\caption{A quasi-median graph, and its contact and crossing complexes.}
\label{ContCrossIntro}
\end{center}
\end{figure}

\begin{definition}
Let $X$ be a quasi-median graph and $\mathbb{G}$ a collection of gated subgraphs. The \emph{$\mathbb{G}$-contact complex} $\mathrm{Cont}^\triangle(X,\mathbb{G})$ is the graph whose vertices are the hyperplanes of $X$ and whose simplices are given by hyperplanes pairwise in contact that all cross a common subgraph from $\mathbb{G}$. 
\end{definition}

\noindent
The contact (resp.\ crossing) complex of a quasi-median graph $X$ then coincides with the $\mathbb{G}$-contact complex of $X$ for $\mathbb{G} = \{X\}$ (resp.\ $\mathbb{G}= \{ \text{prisms} \}$). See Figure~\ref{RelativeContIntro} for a simple example of a relative contact complex.

\begin{figure}[ht!]
\begin{center}
\includegraphics[width=0.6\linewidth]{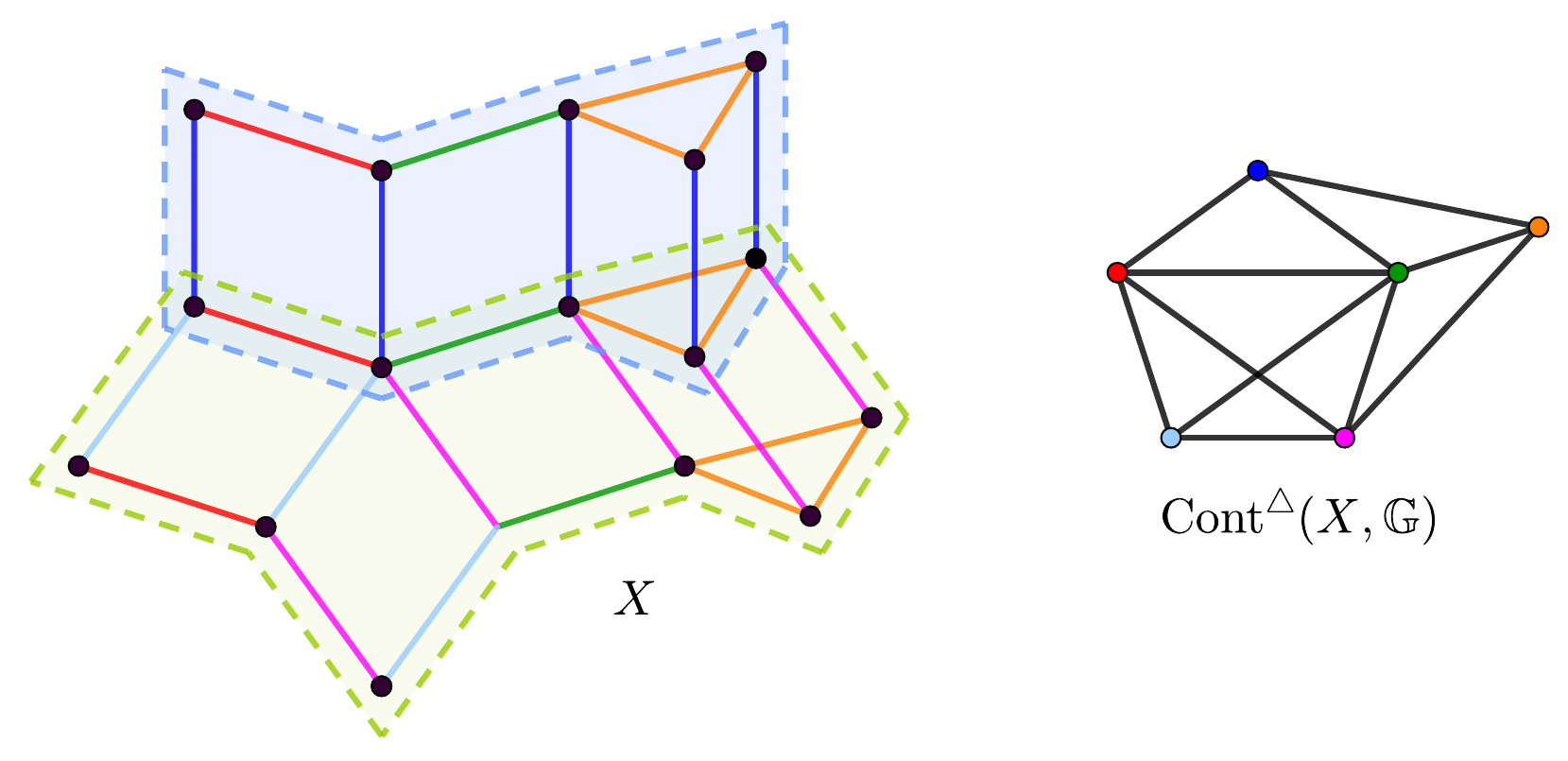}
\caption{A quasi-median graph $X$ endowed with a collection of gated subgraphs $\mathbb{G}$ and the corresponding $\mathbb{G}$-contact complex.}
\label{RelativeContIntro}
\end{center}
\end{figure}

\medskip \noindent
In this article, we identify the homotopy type of various relative contact complexes under weak assumptions. The picture to keep in mind is that, given a quasi-median graph $X$ and a collection of gated subgraphs $\mathbb{G}$ that covers $X$, typically the relative contact complex $\mathrm{Cont}^\triangle(X,\mathbb{G})$ is homotopy equivalent to the pointed sums of all the links of the vertices of $X$ in the cellular complex $\mathrm{ConeOff}(X^\square,\mathbb{G})$ obtained from the prism-completion $X^\square$ by gluing cones over the subgraphs of $\mathbb{G}$.  Given a vertex $x \in X$, we have a simple model for the homotopy type of the link of $x$ in our cone-off, namely the simplicial complex $\mathrm{sL}_\mathbb{G}(x)$ whose vertices are the \emph{cliques} (i.e.\ maximal complete subgraphs) of $X$ containing $x$ and whose simplices are given by collections of cliques contained in a common subgraph from $\mathbb{G}$.  

\medskip \noindent
As a simple case, when $\mathbb{G}= \{X\}$, the simplified links $\mathrm{sL}_\mathbb{G}(x)$ are all simplices. Following our leitmotiv, we therefore expect the corresponding relative contact complex, namely the contact complex $\mathrm{Cont}^\triangle(X)$, to be contractible. This turns out to be the case; 
see Section~\ref{section:Contact} for a short proof. A similar argument can be found in \cite{MR4554674} for finite median graphs. 

\medskip \noindent
The next statement records Theorems~\ref{thm:HomotopyCrossingComplex} and~\ref{thm:RelContact}, which constitute our main contribution to the study of complexes of hyperplanes in quasi-median graphs up to homotopy. 

\begin{thm}\label{thm:BigIntro}
Let $X$ be a quasi-median graph with no cut-vertex and $\mathbb{G}$ a collection of gated subgraphs. Assume that $\mathbb{G}= \{ \text{prisms}\}$ or that $\mathbb{G}$ is star-covering. The $\mathbb{G}$-contact complex $\mathrm{Cont}^\triangle(X,\mathbb{G})$ is homotopy equivalent to the pointed sum $\bigvee_{x \in X} \mathrm{sL}_\mathbb{G}(x)$. 
\end{thm}

\noindent
An immediate consequence of Theorem~\ref{thm:HomotopyCrossingComplex} is that the crossing complex of a quasi-median graph $X$ is connected if and only if $X$ has no cut-vertex. This fact has been observed several times in various situations \cite{MR1920184, MR2382356, MR1881707}. In other words, $\mathrm{Cross}^\triangle(X)$ is $0$-connected if and only if the vertices of the prism-completion $X^\square$ have $0$-connected links. One can think of Theorem~\ref{thm:BigIntro}, when applied to crossing complexes, as a generalisation of this phenomenon to higher dimensions. For instance, $\mathrm{Cross}^\triangle(X)$ is contractible if and only if links of vertices in $X^\square$ are contractible.

\paragraph{Acknowledgements.}

The authors thank Jingyin Huang for pointing out how Theorem~\ref{thm:RAAGmaxJoin} follows from results in~\cite{HuangI}. The authors thank Sangrok Oh for proof reading.
 The first author was partially supported by NSF grants DMS-2106906 and DMS-2340341.The third author  acknowledges funding by the Natural Sciences and Engineering Research Council of Canada NSERC.

\section{Preliminaries on quasi-median geometry}\label{section:QM}

\noindent
A central idea in geometric group theory is that, in order to study a given group, it is often fruitful to make it act on metric spaces with rich geometries. From this perspective, geometries with a combinatorial flavour are especially efficient because they are easier to construct in practice. Notable examples include buildings, small cancellation polygonal complexes, CAT(0) cube complexes (which coincide with cube-completions of median graphs), systolic complexes (which coincide with flag completions of bridged graphs). In this article, we are interested in another such combinatorial geometry: \emph{quasi-median graphs}. 

\medskip \noindent
Median graphs, better known as one-skeleta of CAT(0) cube complexes in geometric group theory, have a long history in metric graph theory. Formally introduced during the 1960-70s \cite{MR125807, MR0286705, MR0522928},  they have deep roots in graph theory and related fields; see the surveys \cite{MR2405677, MR2798499}. Since then, many families of graphs generalising median graphs have been investigated. Inspired by the characterisation of median graphs as retracts of hypercubes, one can introduce \emph{quasi-median graphs} \cite{MR0605838, MR1297190} as retracts of \emph{Hamming graphs} (i.e.\ products of complete graphs). From the perspective of geometric group theory, it was shown that median graphs coincide with one-skeleta of CAT(0) cube complexes \cite{MR1663779, MR1748966, Roller}. CAT(0) cube complexes have been enlightened in \cite{MR0919829} as a convenient source of CAT(0) spaces and then in \cite{MR1347406} as a relevant higher-dimensional generalisation of trees. Since then, CAT(0) cube complexes have played a prominent role in geometric group theory. Quasi-median graphs were introduced in geometric group theory by the second-named author in \cite{QM} and turn out to have applications in graph products of groups \cite{QM, GPacyl, Excentric, QMspecial} and their automorphism groups \cite{GMautGP, GPautacyl}, in Thompson's groups \cite{QM, QMandThompson}, and in wreath products of groups \cite{QM, GTbig, GTnote}. 

\medskip \noindent
It is worth mentioning that, similarly to median graphs, quasi-median graphs can be endowed with a higher-dimensional cellular structure that turns them into CAT(0) complexes. More precisely, prism-completions of quasi-median graphs, as defined below, can be endowed with a CAT(0) metric. 

\medskip \noindent
The definition of a quasi-median graph as a retract of a Hamming graph is not the simplest definition to work with nor to verify. In practice, there are other equivalent definitions, which are easier to deal with but also more technical and more difficult to digest. We refer for instance to \cite{MR1297190} for more details. Nevertheless, similarly to median graphs, the key point is to understand how hyperplanes behave. This is what we describe now.

\begin{definition}
Let $X$ be a graph. A \emph{hyperplane} $J$ is an equivalence class of edges with respect to the reflexive-transitive closure of the relation that identifies two edges whenever they belong to a common $3$-cycle or are opposite in some $4$-cycle.  A hyperplane $J$ \emph{crosses} a subgraph $K$ if there is an edge of $K$ that belongs to $J$.  We denote by $X \backslash \backslash J$ the graph obtained from $X$ by removing (the interiors of) all the edges of $J$. A connected component of $X \backslash \backslash J$ is a \emph{sector}. The \emph{carrier} of $J$, denoted by $N(J)$, is the subgraph induced by all the edges of $J$. The \emph{fibres} of $J$ are the connected components of $N(J) \backslash\backslash J$. Two distinct hyperplanes $J_1$ and $J_2$ are \emph{transverse} if there exist two edges $e_1 \subset J_1$ and $e_2 \subset J_2$ spanning a $4$-cycle in $X$; and they are \emph{tangent} if they are not transverse but $N(J_1) \cap N(J_2) \neq \emptyset$. 
\end{definition}
\begin{figure}
\begin{center}
\includegraphics[width=0.5\linewidth]{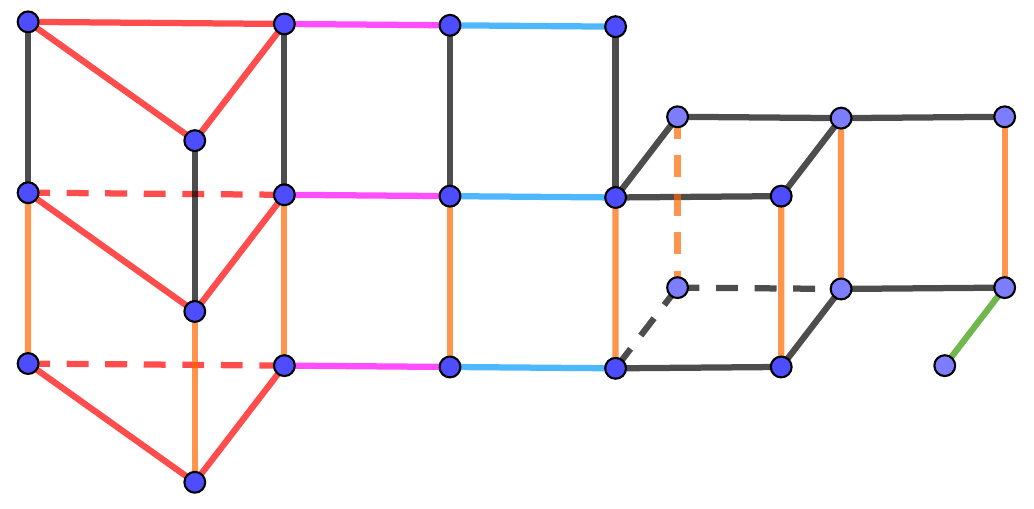}
\caption{A quasi-median graph and some of its hyperplanes. The orange hyperplane is transverse to the red, blue, and purple hyperplanes. The green and orange hyperplanes are tangent but not contiguous. The purple hyperplane is contiguous to the red and blue hyperplanes. The red and blue hyperplanes are neither transverse, nor tangent, nor contiguous.  See Section~\ref{section:RelativeContact} for the definition of \emph{contiguity}. }
\label{figure3}
\end{center}
\end{figure}

\noindent
See Figure~\ref{figure3} for examples of hyperplanes in a quasi-median graph. The (strong) connection between the graph-metric and hyperplanes is motivated by the following statement (which can be found verbatim in \cite[Propositions~2.15 \& 2.30]{QM}, but also in the literature of metric graph theory using a different terminology):

\begin{thm}\label{thm:BigQM}
Let $X$ be a quasi-median graph. 
\begin{itemize}
	\item[(i)] Every hyperplane $J$ separates $X$, i.e.\ $X \backslash \backslash J$ contains at least two connected components.
	\item[(ii)] Carriers, fibres, and sectors are gated subgraphs.
	\item[(iii)] A path in $X$ is geodesic if and only if it crosses every hyperplane at most once, i.e.\ it contains at most one edge from each hyperplane.
	\item[(iv)] The distance between two vertices coincides with the number of hyperplanes separating them.
\end{itemize}
\end{thm}

\noindent
Recall that, in a graph $X$, a subgraph $Y \subset X$ is \emph{gated} if every vertex $x \in X$ admits a \emph{gate} (or a \emph{projection}) in $Y$, i.e.\ a vertex $y \in Y$ such that, for every vertex $z \in Y$, at least one geodesic connecting $x$ to $z$ passes through $y$. Notice that, if the gate exists, it coincides with the unique vertex of $Y$ at minimal distance. One can think of gatedness as a strong convexity condition. Notice that the intersection of gated subgraphs is gated. Also, as gated subgraphs are retracts of $X$, gated subgraphs in quasi-median graphs are also quasi-median graphs in their own right. 

\medskip \noindent
Gated subgraphs in quasi-median graphs are characterized in terms of local properties as follows. Recall that a \emph{clique} is a maximal complete subgraph. Cliques are gated in quasi-median graphs \cite{QM}.  

\begin{prop}[{\cite[Proposition 2.6]{QM}\label{lem:LocallyGated}}]
 Let $X$ be a quasi-median graph.  A connected subgraph $Y$ is gated  if and only if it satisfies the following properties:
 \begin{description}
     \item[(clique absorption)] if $C$ is a clique with an edge in $Y$, then $C$ is contained in $Y$; and 
     \item[(local convexity)] any $4$-cycle in $X$ with two adjacent edges in $Y$ is contained in $Y$.
\end{description}
 \end{prop}
  
\noindent 
Another interesting property satisfied by gated subgraphs is the so-called Helly property. The following statement is well-known, see for instance \cite[Proposition 2.8] {QM}.

\begin{lemma}\label{lem:Helly}
In a graph, any finite collection of pairwise intersecting gated subgraphs has a non-empty intersection. 
\end{lemma}

\noindent
It is useful to understand how projections to gated subgraphs interact with hyperplanes. The following summarises  results from \cite[Proposition 2.33,  Lemmas~2.34 \& 2.36, and Corollary 2.37]{QM}.

\begin{thm}\label{thm:ProjQM}
Let $X$ be a quasi-median graph and $Y,Z \subset X$ two gated subgraphs.
\begin{itemize}
	\item For every vertex $x \in X$, a hyperplane separating $x$ from its projection to $Y$ separates $x$ from $Y$.
	\item For all vertices $x,y \in X$, the hyperplanes separating their projections to $Y$ coincide with the hyperplanes separating $x,y$ and crossing $Y$. As a consequence, the projection to $Y$ is $1$-Lipschitz.
	\item For all vertices $y \in Y$ and $z \in Z$ minimising the distance between $Y$ and $Z$, the hyperplanes separating $y$ and $z$ coincide with the hyperplanes separating $Y$ and $Z$.
     \item The hyperplanes crossing the projection of $Y$ to $Z$ are exactly the hyperplanes crossing both $Y$ and $Z$.
\end{itemize}
\end{thm}

\noindent  
In a quasi-median graph, given a vertex $x$ contained in an edge of a hyperplane $J$, the unique clique containing this edge is called the \emph{clique of $J$ containing $x$.} Note that, as a consequence of the definition of a hyperplane, all the edges of this clique belong to $J$, so the clique is contained in $N(J)$. These cliques are central to most of our arguments.

\begin{cor}\label{cor:CrossAndClique}
Let $X$ be a quasi-median graph, $Y \subset X$ a gated subgraph, $J$ a hyperplane  crossing $Y$, and $x \in N(J) \cap Y$ a vertex. Then the clique of $J$ containing $x$ is contained in~$Y$. 
\end{cor}

\begin{proof}
Indeed, suppose that the clique $C$ of $J$ containing $x$ is not contained in $Y$. Since $Y$ is gated, Proposition~\ref{lem:LocallyGated} implies that no edge of $C$ is contained in $Y$. Let $z$ be a vertex of $C$ distinct from $x$. Then $J$ separates $x$ and $z$ by  Theorem~\ref{thm:BigQM}. By Theorem~\ref{thm:ProjQM}, since $Y$ is gated, $J$ separates $z$ and $Y$. Hence $J$ does not cross $Y$, which is a contradiction.  
\end{proof}

\noindent
In the same way that one can think of median graphs as made of (hyper)cubes, one can think of quasi-median graphs as being made of \emph{prisms}. More precisely, a \emph{prism} is a subgraph that decomposes as a product of cliques. It is worth mentioning that  prisms are also gated in quasi-median graphs \cite[Lemma~2.80]{QM}.  

\begin{lemma}\label{lem:PrismAbsorption}
Let $X$ be a quasi-median graph. If $J$ is a hyperplane crossing a prism $P$, then $P\subset N(J)$. 
\end{lemma}

\begin{proof}
Suppose that $P$ decomposes as a finite product of cliques $C_1\times \cdots \times C_n$, and fix a vertex $x_i$ of $C_i$ for each $i$. For convenience, we identify $C_i$, in the product, with all vertices of the form $(x_1,\ldots,x_{j-1},y , x_{j+1},\ldots,x_{n})$ with $y\in C_i$. Since $J$ crosses $P$, we can assume that $J$ contains an edge of $C_1$, and hence all edges of $C_1$ belong to $J$. To prove that $P\subset N(J)$, we show that, for any  vertex $\bar y=(y_1,y_2, \ldots, y_j ,\ldots ,y_n)$ and any edge $e$ between $\bar y$ and a vertex $\bar y'=(y_1,y_2, \ldots, y_j' ,\ldots, y_n)$ with $y_j\neq y_j'$, the edge $e$ belongs to $J$. We argue by induction on the distance between $\bar y$ and $C_1$ in $P$.  We have already addressed distance zero. Suppose the distance between $\bar y$ and $C_1$ is positive (and larger than or equal to the distance between $\bar y'$ and $C_1$). Then there is an index $i\neq 1$ such that $x_i\neq y_i$. Assume $i\neq j$; the case $i=j$ is analogous modulo the observation that $x_j\neq y_j, y_j'$. Let $y_i'$ be a vertex of $C_i$ distinct from $y_i$. Then the vertices
$\bar y_0=(y_1,\ldots, x_i,\ldots ,y_j\ldots, y_n)$,
$\bar y_1=(y_1,\ldots, x_i, \ldots, y_j', \ldots, y_n)$,
$\bar y' =(y_1,\ldots, y_i,\ldots ,y_j',\ldots, y_n)$, and $\bar y=(y_1,\ldots, y_i,\ldots ,y_j,\ldots, y_n)$ form a $4$-cycle in $P$. By induction, the edges $\bar y_0 \bar y_1$ and $\bar y_0 \bar y$ belong to $J$, and then Proposition~\ref{lem:LocallyGated} implies that the edge $\bar y \bar y'$ belongs to $J$. 
\end{proof}
 
\noindent
The \emph{prism-completion} $X^\square$ of a quasi-median graph $X$ refers to the cellular complex obtained from $X$ by filling every prism with an appropriate product of simplices. The \emph{cubical dimension} of a quasi-median graph is the maximal number of factors in a prism. Prism-completions of quasi-median graphs can be naturally endowed with a CAT(0) geometry \cite{QM}, like median graphs. In this article, we will only require the following weaker assertion:

\begin{prop}\label{prop:QMcontractible}
Prism-completions of quasi-median graphs are contractible.
\end{prop}

\noindent
There is a close connection between prisms and pairwise transverse hyperplanes. Hyperplanes all crossing the same prism must be pairwise transverse, and the converse may also hold:

\begin{thm}[{\cite[Proposition~2.79]{QM}}]\label{thm:PrismsHyp}
Let $X$ be a quasi-median graph of finite cubical dimension. The map
$$\{J_1, \ldots, J_n\} \mapsto \bigcap\limits_{i=1}^n N(J_i)$$
induces a bijection between the maximal collections of pairwise transverse hyperplanes and the maximal prisms of $X$.
\end{thm}

\noindent
For infinite collections of pairwise transverse hyperplanes, it may not be possible to find an infinite-dimensional prism that is crossed by all the hyperplanes. This is why the quasi-median graph in Theorem~\ref{thm:PrismsHyp} is assumed to have finite cubical dimension.

\medskip \noindent
Related to Theorem~\ref{thm:PrismsHyp}, one can describe the intersection of the carriers of a collection of pairwise transverse hyperplanes (when non-empty). In the next statement, given a collection of pointed graphs $\{(X_i,x_i), i \in I\}$, we denote by $\prod_{i \in I} (X_i,x_i)$ the graph whose vertices are the sequences $(a_i)_{i \in I}$ for which $a_i=x_i$ for all but finitely many $i \in I$ and whose edges connect two sequences whenever they differ at a single index $j$ of $I$ and the corresponding vertices of $X_j$ are connected by an edge. This definition is motivated by the fact that, even though every $X_i$ is connected, the product $\prod_{i \in I} X_i$ cannot be connected if $I$ is infinite (assuming that the $X_i$ are all non-empty). In this case, $\prod_{i \in I} (X_i,x_i)$ coincides with the connected component of $\prod_{i \in I} X_i$ containing the vertex $(x_i)_{i \in I}$.

\begin{prop}\label{prop:InterHyperplanesProduct}
Let $X$ be a quasi-median graph and $\mathcal{J}$ a collection of pairwise transverse hyperplanes. Suppose there exists a vertex $o \in \bigcap_{J \in \mathcal{J}} N(J)$, and let $C_J$ be the clique of $J$ containing $o$ for every hyperplane $J \in \mathcal{J}$. Then the map
$$\left\{ \begin{array}{ccc} \bigcap\limits_{J \in \mathcal{J}} N(J) & \to & F \times \prod\limits_{J \in \mathcal{J}} (C_J,o) \\ x & \mapsto & \left(\mathrm{proj}_F(x), (\mathrm{proj}_{C_J}(x))_{J \in \mathcal{J}} \right) \end{array} \right.$$
is a graph isomorphism, where $F$ denotes the intersection of all the fibres of the hyperplanes from $\mathcal{J}$ containing $o$. 
\end{prop}

\noindent
This proposition is a rather straightforward consequence of the following general decomposition criterion. In its statement and  proof, we use $\mathcal{H}(\cdot)$ to denote the set of all the hyperplanes crossing a given subgraph.

\begin{lemma}\label{lem:ProductDecomposition}
Let $X$ be a quasi-median graph and $\{Y_i \,:\, i \in I\}$ a collection of gated subgraphs. Assume that the following conditions hold:
\begin{itemize}
	\item the intersection $Y:= \bigcap_{i \in I} Y_i$ is non-empty;
	\item $\mathcal{H}(X) = \bigsqcup_{i \in I} \mathcal{H}(Y_i)$;
	\item for all distinct $i,j \in I$, every hyperplane in $\mathcal{H}(Y_i)$ is transverse to every hyperplane in $\mathcal{H}(Y_j)$.
\end{itemize}
Then, for every vertex { $o\in \bigcap_{i\in I} Y_i$} , the map
$$\left\{ \begin{array}{ccc} X & \to & \prod\limits_{i \in I} (Y_i,o) \\ x & \mapsto & \left( \mathrm{proj}_{Y_i}(x) \right)_{i \in I} \end{array} \right.$$
defines a graph isomorphism.
\end{lemma}

\begin{proof}
Let $\Psi$ denote the map given by the lemma. We start by proving that $\Psi$ is an isometric embedding. We will use the following notation: for all vertices $a,b \in X$, we denote by $\mathcal{H}( a | b)$ the set of the hyperplanes separating $a$ and $b$. For all vertices $x,y \in X$, we deduce from Theorem~\ref{thm:ProjQM} that
$$\begin{array}{lcl} d(\Psi(x),\Psi(y)) & = & \displaystyle \sum\limits_{i \in I} d \left( \mathrm{proj}_{Y_i}(x), \mathrm{proj}_{Y_i}(y) \right) = \sum\limits_{i \in I} | \mathcal{H}(x|y) \cap \mathcal{H}(Y_i) | \\ \\ & = & \displaystyle \left| \bigsqcup\limits_{i \in I} \mathcal{H}(x|y) \cap \mathcal{H}(Y_i) \right| = |\mathcal{H}(x|y)| = d(x,y). \end{array}$$
Thus $\Psi$ is indeed an isometric embedding. It remains to verify that $\Psi$ is surjective. 

\medskip \noindent
Fix a vertex $(y_i)_{i \in I} \in \prod_{i \in I}(Y_i,o)$. Let $y_1, \ldots, y_n$ denote the finitely many coordinates $y_i$ satisfying $\mathcal{H}(o|y_i) \neq \emptyset$. (Recall that $y_i=o$ for all but finitely many $i \in I$.) Let $\mathcal{H}$ denote the set of the hyperplanes separating $o$ from some $y_i$, i.e.\ $\mathcal{H}:= \bigcup_i \mathcal{H}(o|y_i)$. Notice that $\mathcal{H}$ is finite. For all $1 \leq i \leq n$ and $J \in \mathcal{H}(o|y_i)$, let $J^+$ denote the sector delimited by $J$ that contains $y_i$. 

\begin{claim}\label{claim:InterNotEmpty}
The intersection $Q:= \bigcap_{J \in \mathcal{H}} J^+$ is non-empty.
\end{claim}

\noindent
By Lemma~\ref{lem:Helly}, it suffices to verify that, for all hyperplanes $J_1,J_2 \in \mathcal{H}$, the intersection $J_1^+ \cap J_2^+$ is non-empty. If $J_1$ and $J_2$ are transverse, there is nothing to prove, so assume that they are not transverse. As a consequence, there must exist some $1 \leq i \leq n$ such that $J_1$ and $J_2$ both belong to $\mathcal{H}(o|y_i)$. But then $y_i \in J_1^+ \cap J_2^+$. This proves Claim~\ref{claim:InterNotEmpty}.

\medskip \noindent
Now, set $x:= \mathrm{proj}_Q(o)$. We want to prove that $\Psi(x)=(y_i)_i$. This will be a straightforward consequence of the following observation:

\begin{claim}\label{claim:Hypseparating}
The equality $\displaystyle \mathcal{H}(o|x) = \bigsqcup\limits_{i=1}^n \mathcal{H}(o|y_i)$ holds.
\end{claim}

\noindent  
Suppose  $J\not\in  \bigsqcup_{i=1}^n \mathcal{H}(o|y_i)$, and let $J^+$ denote the sector containing $o$ and the $y_i$. Then, for all $1 \leq i \leq n$ and $K\in \mathcal{H}(o|y_i)$, we have $y_i \in K^+\cap J^+$.
By the Helly property for gated subgraphs, Lemma~\ref{lem:Helly}, it follows that $J^+\cap Q$ is non-empty. Since $o\in J^+$, we have that $J$ does not separate $o$ and $x= \mathrm{proj}_Q(o)$ as a consequence of   Theorem~\ref{thm:ProjQM}.  Conversely, if  $J\in  \mathcal{H}(o|y_i)$ for some $i$, then $J$ separates $o$ from $Q$, and hence $J$ separates $o$ from $x$, completing the proof of Claim~\ref{claim:Hypseparating}.  

\medskip \noindent
To conclude, note that, for every $1 \leq i \leq n$, $\mathcal{H}(o|y_i) \subset \mathcal{H}(Y_i)$. Indeed, by Theorem~\ref{thm:BigQM}, for all $1 \leq i \leq n$ and $J\in\mathcal{H}(o|y_i)$, since $Y_i$ is gated and $o,y_i\in Y_i$, any geodesic in $Y$ between $o$ and $y_i$ contains an edge of $J$ and is contained in $Y_i$;  therefore $J\in \mathcal{H}(Y_i)$.  We deduce from Claim~\ref{claim:Hypseparating} and Theorem~\ref{thm:ProjQM} that
$$\mathcal{H}\left( \mathrm{proj}_{Y_i}(x) | o \right) = \mathcal{H}(o|x) \cap \mathcal{H}(Y_i) = \mathcal{H}(o|y_i)$$
for every $1 \leq i \leq n$. Since the $Y_i$ are gated, this implies that $\mathrm{proj}_{Y_i}(x)= y_i$ for every $i \in I$, and so $\Psi(x)=(y_i)_i$, as desired.
\end{proof}

\begin{proof}[Proof of Proposition~\ref{prop:InterHyperplanesProduct}.]
It suffices to verify that the assumptions of Lemma~\ref{lem:ProductDecomposition} apply to $\bigcap_{J \in \mathcal{J}} N(J)$ with respect to $F$ and the $C_J$, $J \in \mathcal{J}$. First, the intersection $F \cap \bigcap_{J \in \mathcal{J}} C_J$ is non-empty since it contains the vertex $o$. Next, given a hyperplane $K$  crossing $\bigcap_{J \in \mathcal{J}} N(J)$, either $K$ belongs to $\mathcal{J}$, in which case $K \in \mathcal{H}(C_K)$; or $K$ is transverse to all the hyperplanes in $\mathcal{J}$, which implies that it intersects $F$, i.e.\ $K \in \mathcal{H}(F)$. This shows that
$$\mathcal{H} \left( \bigcap\limits_{J \in \mathcal{J}} N(J) \right) = \mathcal{H}(F) \sqcup \bigsqcup\limits_{J \in \mathcal{J}} \mathcal{H}(C_J).$$
Finally, any two distinct hyperplanes in $\mathcal{J}$ are transverse by assumption, and it is clear that a hyperplane in $\mathcal{J}$ is transverse to all the hyperplanes crossing $F$. 
\end{proof}

\noindent
We conclude this section with a last preliminary lemma that will be used later.
 
\begin{lemma}\label{lem:GoodPrism}
Let $X$ be a quasi-median graph and $J$ a hyperplane. If $P$ is a prism such that $P\subset N(J)$, then there is a prism $Q$ such that $P\subset Q\subset N(J)$ and $Q$ is crossed by~$J$.
\end{lemma}

\begin{proof}
If $P$ is crossed by $J$, then it suffices to set $Q:=P$. Suppose that $J$ does not cross $P$. Let $o\in P\cap N(J)$, and let $C$ be the clique of $J$ containing $o$. By Proposition~\ref{prop:InterHyperplanesProduct}, there is a graph isomorphism $\Phi\colon (F,o)\times (C,o) \to N(J)$ where $F$ is the fibre of $J$ containing $o$. 
Observe that the projection of $P$ to $C$ does not contain edges of $C$, since otherwise $J$ would cross $P$. It follows that the projection of $P$ to $C$ is the single vertex $o$. In other words, $P$ is contained in $F$.
In particular, $P$ is contained in the subgraph $Q$ of $N(J)$ obtained by 
considering the image of the product $P\times C$ by the isomorphism $\Phi$. Since $P$ is a prism and $C$ is a clique, it follows $Q$ is a prism.
\end{proof}

\section{Homotopy types of complexes of hyperplanes}\label{section:ComplexesHyp}

\noindent
By a complex of hyperplanes, we vaguely refer to a simplicial complex whose vertices are the hyperplanes of some quasi-median graph and whose simplices are given by collections of hyperplanes that interact in some specific way. 

\medskip \noindent
Several graphs of hyperplanes have been already studied. For instance, crossing graphs appeared independently in the literature in several contexts, including \cite{MR1153934, MR1379364} (as \emph{incompatibility graphs}),  \cite{Roller} (as \emph{transversality graphs}), and \cite{MR1920184, MR3217625} (as crossing graphs). Tangency graphs also appear in \cite{MR2135469, MR2305571} as \emph{$\tau$-graphs}; and contact graphs appear in \cite{ThetaGraphs, MR2355723} (as $\Theta$-graphs) and in \cite{MR3217625} (as contact graphs). 

\medskip \noindent
However, very few higher-dimensional complexes of hyperplanes seem to have been studied. (With the exception of the crossing complex, that can also be found in \cite{MR4554674} for finite median graphs.) In addition to the examples previously mentioned, we introduce \emph{contiguity complexes} of quasi-median graphs in Section~\ref{section:RelativeContact}, which do not seem to have been considered elsewhere.

\paragraph{A word about crossing and contact complexes.} From the perspective of geometric group theory, crossing and contact complexes play a particular role compared to the other complexes of hyperplanes, due to the analogy between curve graphs of surfaces and contact graphs as motivated in \cite{MR3217625} (for median graphs). 

\medskip \noindent
Recall that, given a (say closed) surface $\Sigma$, the \emph{curve complex} $\mathscr{C}(\Sigma)$ is the graph whose vertices are the (non-trivial) simple closed curves of $\Sigma$ modulo homotopy and whose simplices are given by pairwise disjoint curves. Curve complexes have been quite useful in the study of mapping class groups from various perspectives. For instance, they provide convenient geometric models that can be used to extract some negative curvature \cite{MR1714338, MR2367021}; collectively, they can be used to recover the entire structure of a given mapping class group, as an isometry group \cite{MR1460387} or through a hierarchical structure \cite{MR1791145}; they are algorithmically friendly \cite{MR2705485, MR2970053, MR3378831}; and they have played an important role in computations of cohomology groups \cite{MR0830043}. In geometric group theory, curve complexes have been often used as a motivation for various constructions and results. Now, it is clear that the connection between mapping class groups and their curve complexes is just one instance of a much more general phenomenon that commonly occurs in geometric group theory.

\medskip \noindent
In the spirit of the previous paragraph, let us mention that contact graphs of (quasi-)median graphs are quasi-trees, and, in particular, hyperbolic. They can sometimes be used to construct hierarchical structures \cite{MR3650081}, to get Ivanov-type theorems \cite{MR4387161}, and they are algorithmically friendly \cite{GTranslation}. 

\medskip \noindent
Every graph (not necessarily connected) can be realised as the crossing graph of a median graph. This phenomenon has been noticed several times independently \cite{Roller, MR1920184, MR3217625}. For instance, given a graph $\Gamma$, the graph $\Delta(\Gamma)$ whose vertices are the complete (possibly empty) subgraphs  of $\Gamma$ and whose edges connect two subgraphs whenever one can be obtained from the other by adding a single vertex is a median graph whose crossing graph is isomorphic to $\Gamma$. On the other hand, contact graphs have a much more restricted structure. For instance, contact graphs of quasi-median graphs are always connected and are $3$-hyperbolic \cite{GTranslation} (see \cite{MR3217625} for the median case). This is a major difference between crossing and contact graphs. 

\medskip \noindent
However, as highlighted in \cite{MR4057355} (for median graphs), the difference between contact and crossing graphs becomes less apparent when one restricts oneself to $2$-connected graphs of bounded degrees. For instance, in this case contact and crossing graphs turn out to be quasi-isometric, and so indistinguishable from the point of view of large-scale geometry. 

\begin{thm}[\cite{GTranslation}]\label{thm:CrossQM}
Let $X$ be a quasi-median graph. 
\begin{itemize}
	\item[(i)] $\mathrm{Cont}(X)$ is connected, $3$-hyperbolic, and quasi-isometric to a tree.
	\item[(ii)] $\mathrm{Cross}(X)$ is connected if and only if $X$ is $2$-connected. 
	\item[(iii)] If $X$ is $2$-connected and has degree $\leq N$, then $\mathrm{Cross}(X)$ is $(3+N/2)$-hyperbolic and quasi-isometric to a tree.
	\item[(iv)] If $X$ is $2$-connected and has bounded degree, then $\mathrm{Cross}(X)$ and $\mathrm{Cont}(X)$ are quasi-isometric.
\end{itemize}
\end{thm}

\paragraph{Back to general complexes of hyperplanes.} Given a quasi-median graph $X$, our goal in this section is to identify the homotopy types of various complexes of hyperplanes. In each case, the general strategy that we will follow is the following: 
\begin{enumerate}
	\item For every hyperplane $J$ of $X$, we will introduce some contractible subspace $S(J)$ of the prism-completion $X^\square$ of $X$. 
	\item The nerve complex of $\{S(J), \text{ $J$ hyperplane}\}$ will coincide with the complex of hyperplanes we will be interested in.
	\item Applying the nerve theorem, we will be able to identify, up to homotopy, our complex of hyperplanes with a specific subspace of $X^\square$.
	\item Using the contractibility of $X^\square$, we will conclude that this subspace is homotopy equivalent to a pointed sum of specific spaces.
\end{enumerate}

\noindent
A key ingredient in this strategy is Leray's nerve theorem. Recall that, given a cover $\mathcal U=\{U_i\}_{i\in \mathcal I}$ of a simplicial complex $Z$ by subcomplexes, the \textit{nerve} of $\mathcal U$, denoted $\mathcal N(\mathcal U)$, is the simplicial complex whose vertex-set is $\mathcal I$ and whose simplices are the finite subsets $\sigma\subset \mathcal I$  such that the intersection $\cap_{i\in \sigma} U_i$ is non-empty.

\begin{thm}[Leray's Nerve Theorem, {\cite[Lemma 1.1]{MR607041}}]
Let $\mathcal U=\{U_i\}_{i\in \mathcal I}$ be a cover of a simplicial complex $Z$ by subcomplexes such that, for every finite $\mathcal J\subset \mathcal I$, the set $\bigcap_{i\in \mathcal J} U_i$ is either empty or contractible.  Then $\mathcal N(\mathcal U)$ is homotopy equivalent to $Z$.
\end{thm}

\noindent
In Section~\ref{section:Contact} we apply this strategy to  contact complexes, and in Section~\ref{section:Crossing} to crossing complexes. Finally, in Section~\ref{section:RelativeContact}, we introduce and study a family of complexes that interpolate between crossing and contact graphs, including contiguity complexes.

\subsection{Contact complexes}\label{section:Contact}

\noindent
In this section, we focus on \emph{contact complexes} of quasi-median graphs. See Figure~\ref{ContCross} for an example. 

\begin{definition}
Let $X$ be a quasi-median graph. The \emph{contact complex} $\mathrm{Cont}^\triangle(X)$ is the simplicial complex whose vertices are the hyperplanes of $X$ and whose simplices are given by collections of hyperplanes pairwise in contact.
\end{definition}

\begin{figure}[h!]
\begin{center}
\includegraphics[width=0.7\linewidth]{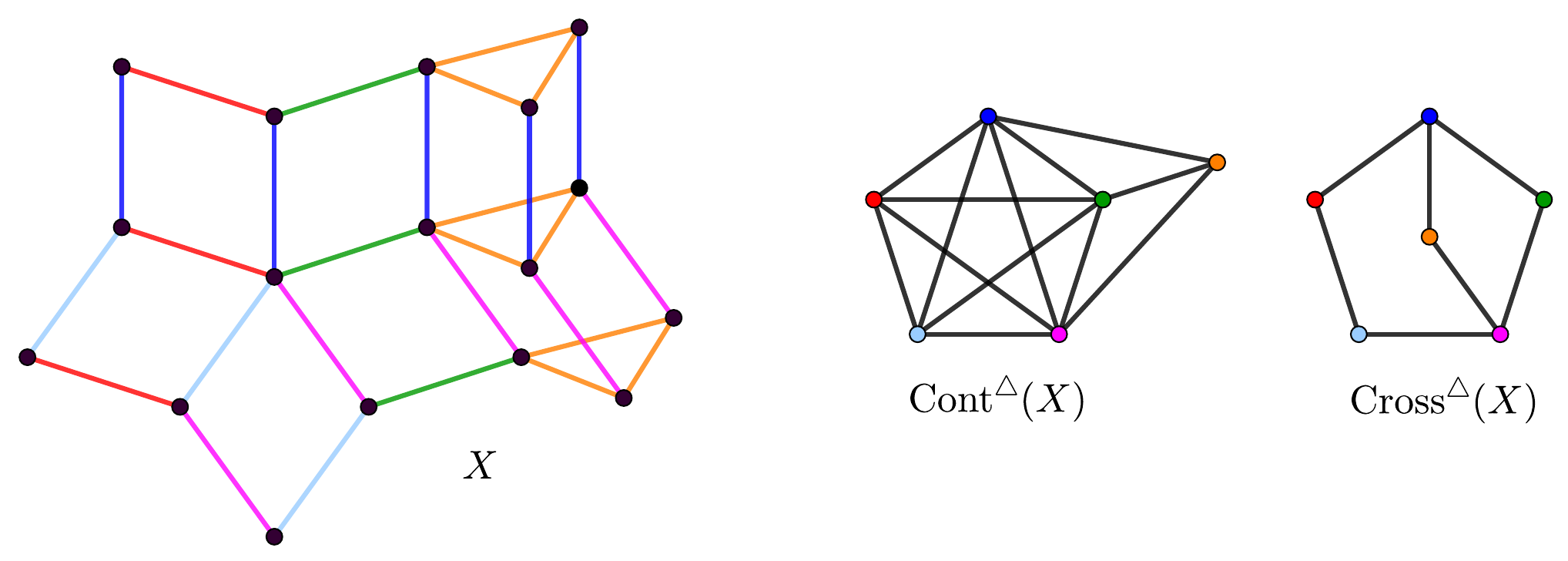}
\caption{A quasi-median graph, and its contact and crossing complexes.}
\label{ContCross}
\end{center}
\end{figure}
\noindent
Let us verify the following observation (which  can also be found in \cite{MR4554674} for finite median graphs, based on a similar argument):

\begin{prop}\label{prop:ContactContractible}
The contact complex $\mathrm{Cont}^\triangle(X)$ of every quasi-median graph $X$ is contractible.
\end{prop}

\begin{proof}
Let $X$ be a quasi-median graph. Because two hyperplanes are in contact if and only if their carriers intersect, the nerve complex of the covering $\{ N^\square(J), \text{ $J$ hyperplanes}\}$ of $X^\square$ coincides with the contact complex $\mathrm{Cont}^\triangle(X)$. Given a collection $\mathcal{J}$ of hyperplanes, if $\bigcap_{J \in \mathcal{J}} N^\square(J)$ is non-empty, then it coincides with the prism completion of the gated subgraph $\bigcap_{J \in \mathcal{J}} N(J)$. By Proposition~\ref{prop:QMcontractible}, such an intersection has to be contractible. We deduce from the nerve theorem that $\mathrm{Cont}^\triangle (X)$ is homotopy equivalent to $X^\square$, and we conclude from Proposition~\ref{prop:QMcontractible} that our contact complex is contractible.  
\end{proof}

\subsection{Crossing complexes}\label{section:Crossing}

\noindent
In this section, we focus on \emph{crossing complexes} of quasi-median graphs. See Figure~\ref{ContCross} for an example. 

\begin{definition}
Let $X$ be a quasi-median graph. The \emph{crossing complex} $\mathrm{Cross}^\triangle(X)$ is the simplicial complex whose vertices are the hyperplanes of $X$ and whose simplices are given by collections of pairwise transverse hyperplanes. 
\end{definition}

\noindent 
Our goal is to identify, up to homotopy, crossing complexes of quasi-median graphs with pointed sums. For this purpose, we need to introduce some terminology. Given a quasi-median graph $X$, the \textit{link} of a vertex $x$ in the prism-completion $X^\square$, which will be denoted by $\mathrm{link}_{X^\square}(x)$, refers to the simplicial complex whose vertices are the edges of $X$ containing $x$ and whose simplices are given by collections of edges contained in a common prism. It can be thought of as the simplicial structure induced by the cellular structure of $X^\square$ on a small sphere around the vertex $x$. Up to homotopy, this simplicial complex can be simplified. The \textit{simplified link} $\mathrm{slink}_{X^\square}(x)$ of our vertex $x$ will refer to the simplicial complex whose vertices are the cliques of $X$ containing $x$ and whose simplices are given by collections of cliques contained in a common prism.  The rest of the section is dedicated to the proof of the following statement.

\begin{thm}\label{thm:HomotopyCrossingComplex}
Let $X$ be a quasi-median graph. The crossing complex $\mathrm{Cross}^\triangle(X)$ of $X$ is homotopy equivalent to the union of pointed sums 
$$\bigsqcup\limits_{Y \in\mathcal Y} \ \bigvee_{x \in Y} \mathrm{link}_{Y^\square}(x),$$ 
where $\mathcal Y$ is the set of 2-connected components of $X$;
or, equivalently, to the union
$$\bigsqcup\limits_{Y \in \mathcal Y} \ \bigvee_{x \in Y} \mathrm{slink}_{Y^\square}(x).$$
\end{thm}

\noindent
We begin by observing that links and simplified links are homotopy equivalent, which justifies why the two unions from our theorem are indeed homotopy equivalent.  This is a particular case of Lemma~\ref{lem:LinkVsSlink} below, so we do not repeat an argument here.

\begin{lemma}\label{lem:Loc}
Let $X$ be a quasi-median graph. For every vertex $x \in X$, the link $\mathrm{link}_{X^\square}(x)$ and the simplified link $\mathrm{slink}_{X^\square}(x)$ are homotopy equivalent.
\end{lemma}

\noindent
 Now, let us prove the following observation:

\begin{prop}\label{prop:PuncturedQM}
Let $X$ be a quasi-median graph. The crossing complex $\mathrm{Cross}^\triangle(X)$ is homotopy equivalent to $X^\square \backslash X$.
\end{prop}

\noindent
A similar statement can be found in \cite{MR4554674} for finite median graphs. Our proof follows essentially the same strategy. 

\begin{proof}[Proof of Proposition~\ref{prop:PuncturedQM}.]
Given a hyperplane $J$, we set
$$N^\square_-(J) := N^\square(J) \backslash \bigcup\limits_{F \text{ fibre of } J} F^\square.$$
The proposition will follow by applying the nerve theorem to $\{ N^\square_-(J) \mid J \text{ hyperplane}\}$.

\begin{claim}\label{claim:SmallCarriersCover}
We have $X^\square \backslash X = \bigcup\limits_{J \text{ hyperplane}} N^\square_-(J).$
\end{claim}

\noindent
Given a point $x \in X^\square\backslash X$, let $P$ be a prism of minimal cubical dimension whose prism completion contains $x$. Because $x$ is not a vertex, $P$ must have cubical dimension $\geq 1$. This implies that $P$ is crossed by at least one hyperplane, say $J$. If $x$ does not belong to $N^\square_-(J)$, then there must exist a fibre $F$ of $J$ such that $x \in (P \cap F)^\square$, contradicting the minimality of $P$. 

\begin{claim}\label{claim:SimplexTransversality}
For all distinct hyperplanes $J_1$ and $J_2$, $N^\square_-(J_1)$ and $N^\square_-(J_2)$ intersect if and only if $J_1$ and $J_2$ are transverse. 
\end{claim}

\noindent
If $J_1$ and $J_2$ are transverse, then there exists a prism $P$ crossed by both $J_1$ and $J_2$. An interior point of $P$ necessarily belongs to $N^\square_-(J_1) \cap N^\square_-(J_2)$. Conversely, assume that $N^\square_-(J_1) \cap N^\square_-(J_2) \neq \emptyset$. A fortiori, $N(J_1)$ and $N(J_2)$ intersect, so $J_1$ and $J_2$ are either transverse or tangent. If they are tangent, then $J_2$ is contained in a single sector delimited by $J_1$, so $N(J_1) \cap N(J_2)$ must be contained in a fibre of $J_1$, say $F$. But this is not possible since $F^\square$ is disjoint from $N^\square_-(J_1)$. Therefore, $J_1$ and $J_2$ must be transverse. 

\begin{claim}\label{claim:ContractibleIntersection}
For every finite collection of pairwise transverse hyperplanes $\mathcal{J}$, the intersection $\bigcap_{J \in \mathcal{J}} N^\square_-(J)$ is contractible.
\end{claim}

\noindent
Fix a vertex $o \in \bigcap_{J \in \mathcal{J}} N(J)$.  Such a vertex exists because, as gated subgraphs, hyperplane-carriers satisfy the Helly property. For every $J \in \mathcal{J}$, the vertex $o$ belongs to a (unique) clique $C_J \subset J$.  Let $F_J$ denote the fibre of $J$ containing $o$, and set $F:= \bigcap_{J \in \mathcal{J}} F_J$. According to Proposition~\ref{prop:InterHyperplanesProduct}, the map
$$\Psi : \left\{ \begin{array}{ccc} \bigcap\limits_{J \in \mathcal{J}} N(J) & \to & F \times \prod\limits_{J \in \mathcal{J}} C_J \\ x & \mapsto & \left(\mathrm{proj}_F(x), (\mathrm{proj}_{C_J}(x))_{J \in \mathcal{J}} \right) \end{array} \right.$$
is a graph isomorphism. A fortiori, $\Psi$ induces a homeomorphism between the prism-completions of the two quasi-median graphs. Our goal now is to understand the images under $\Psi$ of fibre-hyperplanes in order to conclude that $\Psi$ induces a homeomorphism between our intersection $\bigcap_{J \in \mathcal{J}} N_-^\square(J)$ and some product of contractible complexes. 

\medskip \noindent
Notice that, for every hyperplane $J \in \mathcal{J}$ and for every fibre $L$ of $J$, we have
$$\Psi\left(L\cap \bigcap_{I\in\mathcal{J}}N(I) \right) = F \times \{x\} \times \prod\limits_{I \in \mathcal{J} \backslash \{J\}} C_I$$
where $x$ is the vertex such that $L \cap C_J = \{x\}$. This is due to the fact that $L$ coincides with the vertices of $N(J)$ whose projection on $C_J$ is $x$. It follows that 
$$\Psi\left(L^\square\cap \bigcap_{I\in\mathcal{J}}N(I)^\square\right) = F^\square \times \{x\} \times  \prod\limits_{I \in \mathcal{J} \backslash \{J\}} C_I^\square,$$
where, for convenience, we identify $\Psi$ with the map it induces on the prism-completions. Consequently, the map $\Psi$ induces a homeomorphism
$$\bigcap\limits_{J \in \mathcal{J}} N^\square_-(J) \to F^\square \times \prod\limits_{J \in \mathcal{J}} C_J^\square\backslash C_J.$$
Notice that each $C_J^\square \backslash C_J$ is the complement of the vertices in a simplex (of dimension $\geq 1$) and that $F^\square$ is the prism completion of a quasi-median graph, as $F$ is a gated subgraph of $X$. Therefore, the right-hand side of the expression above is a product of contractible spaces, concluding the proof of Claim~\ref{claim:ContractibleIntersection}. 

\medskip \noindent
Claims~\ref{claim:SmallCarriersCover} and~\ref{claim:ContractibleIntersection} show that the nerve theorem applies, and hence $X^\square \backslash X$ is homotopy equivalent to the nerve complex of $\{ N^\square_-(J) \mid J \text{ hyperplane}\}$. According to Claim~\ref{claim:SimplexTransversality}, the latter coincides with the crossing complex $\mathrm{Cross}^\triangle(X)$, completing the proof of the proposition. 
\end{proof}

\noindent
We are now ready to prove the main theorem of this section, which uses the following fact from algebraic topology.

\begin{fact}\label{fact:AlgTopo}
Let $C$ be a CW-complex and $A, B \subset C$ two subcomplexes satisfying $A \cup B = C$. If $C$ and $B$ are both contractible, then $A$ deformation retracts on $A \cap B$. 
\end{fact}

\begin{proof}
Note that, by the Whitehead Theorem, it suffices to show that the inclusion map $A \cap B \to A$ induces an isomorphism on $\pi_n$ for every $n \geq 0$. Since $B$ and $C$ are contractible,  $\pi_{n}(A, A \cap B)$ is isomorphic to $\pi_{n}(C, B) = \{1\}$  for all $n$ by ~\cite[Theorem~4.23, Chapter~4]{MR1867354}). Then the long exact sequence of homotopy groups for the pair $(A, A\cap B)$ shows that $\pi_n(A\cap B)\to \pi_n(A)$ is an isomorphism, concluding the proof.  
\end{proof}

\begin{proof}[Proof of Theorem~\ref{thm:HomotopyCrossingComplex}.]
Because the decomposition
$$\mathrm{Cross}^\triangle(X)= \bigsqcup\limits_{Y \text{ $2$-conn.\ comp.\ of } X} \mathrm{Cross}^\triangle(Y)$$
is clear, we can and shall assume in the rest of the proof of $X$ is $2$-connected. 

\medskip \noindent
Let $T \subset X^\square$ be a spanning tree in the one-skeleton. Fix a small $\epsilon>0$ and consider the thickening
$$T^+ := T \cup \bigcup\limits_{x \in X} B(x,\epsilon).$$
Here, we endow $X^\square$ with the length metric that extends the natural metrics on the prisms obtained from Euclidean realisations as polytopes. (This is a CAT(0) metric \cite{QM}, but we will not use this fact.) The key point to keep in mind is that the balls $B(x,\epsilon)$ are pairwise disjoint and that each sphere $S(x,\epsilon)$ is homeomorphic to $\mathrm{link}_{X^\square}(x)$. 

\medskip \noindent
Clearly, $T^+$ is contractible. As $X^\square$ is contractible as well, 
Fact~\ref{fact:AlgTopo} implies that $T^+\backslash X$ is a deformation retract of $X^\square \backslash X$. But $T^+\backslash X$ is  homotopy equivalent to $\bigvee_{x \in X} S(x,\epsilon)$, or equivalently to $\bigvee_{x \in X} \mathrm{link}_{X^\square}(x)$. The desired conclusion then follows from Proposition~\ref{prop:PuncturedQM} and Lemma~\ref{lem:Loc}. 
\end{proof}

\subsection{Relative contact complexes}\label{section:RelativeContact}

\noindent
In this section, we introduce and study \emph{relative contact complexes} of quasi-median graphs. See Figure~\ref{RelativeCont} for an example. 
\begin{figure}
\begin{center}
\includegraphics[width=0.6\linewidth]{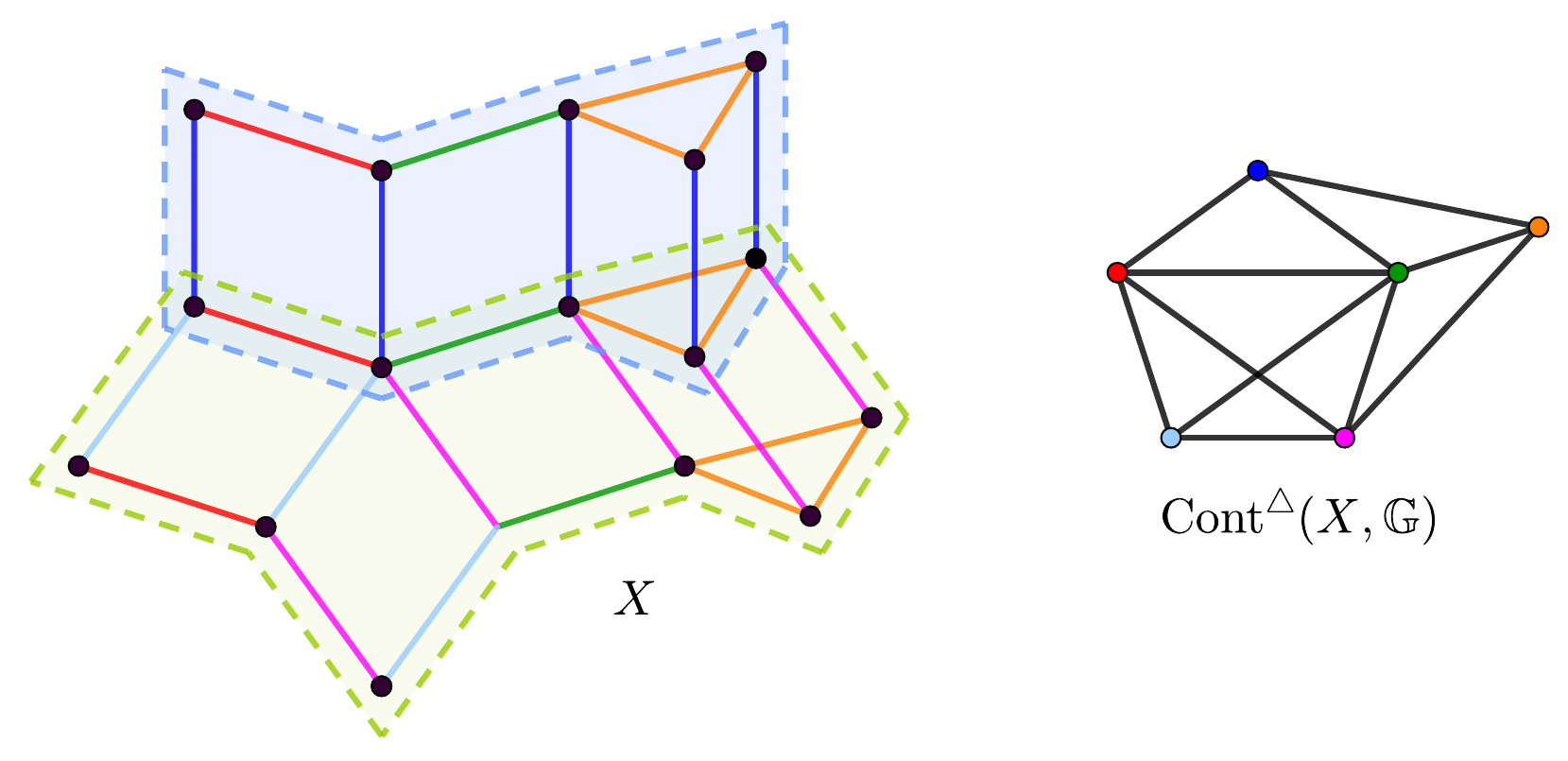}
\caption{A quasi-median graph $X$ endowed with a collection of gated subgraphs $\mathbb{G}$ and the corresponding $\mathbb{G}$-contact complex.}
\label{RelativeCont}
\end{center}
\end{figure}

\begin{definition}
Let $X$ be a quasi-median graph and $\mathbb{G}$ a collection of gated subgraphs. The \emph{$\mathbb{G}$-contact complex} $\mathrm{Cont}^\triangle(X,\mathbb{G})$ is the graph whose vertices are the hyperplanes of $X$ and whose simplices are given by hyperplanes pairwise in contact that all cross a common subgraph from $\mathbb{G}$. 
\end{definition}

\noindent
In practice, we will always assume that $\mathbb{G}$ is \emph{prism-covering}, i.e.\ every prism of $X$ is contained in some subgraph from $\mathbb{G}$. One can think of relative contact complexes as an interpolation between crossing complexes (when $\mathbb{G}= \{ \text{prisms}\}$) and contact complexes (when $\mathbb{G}= \{X\}$). 

\medskip \noindent
A particular case of interest, distinct from crossing and contact complexes, is given by \emph{contiguity complexes}. See Figure~\ref{Contg} for an example. 

\begin{definition}
Let $X$ be a quasi-median graph. Two hyperplanes are \emph{contiguous} whenever their carriers contain a common clique. The \emph{contiguity complex} $\mathrm{Contg}^\triangle(X)$ is the $\mathbb{G}$-contact complex $\mathrm{Cont}^\triangle(X,\mathbb{G})$ where
$$\mathbb{G}:= \{ \text{prisms containing } C \mid C \text{ clique of } X\},$$
 i.e.\ the simplicial complex whose vertices are the hyperplanes of $X$ and whose simplices are given by collections of hyperplanes with a common clique in their carriers.
\end{definition}

\noindent
For instance, the contiguity complex of a $(2 \times 2)$-grid is the two-skeleton of a tetrahedron, so topologically a $2$-sphere. We emphasize that, as illustrated by Figure~\ref{Contg}, a collection of pairwise contiguous hyperplanes may not be globally contiguous. Therefore, contrary to crossing and contact complexes, a contiguity complex may not be flag. 
\begin{figure}
\begin{center}
\includegraphics[width=0.6\linewidth]{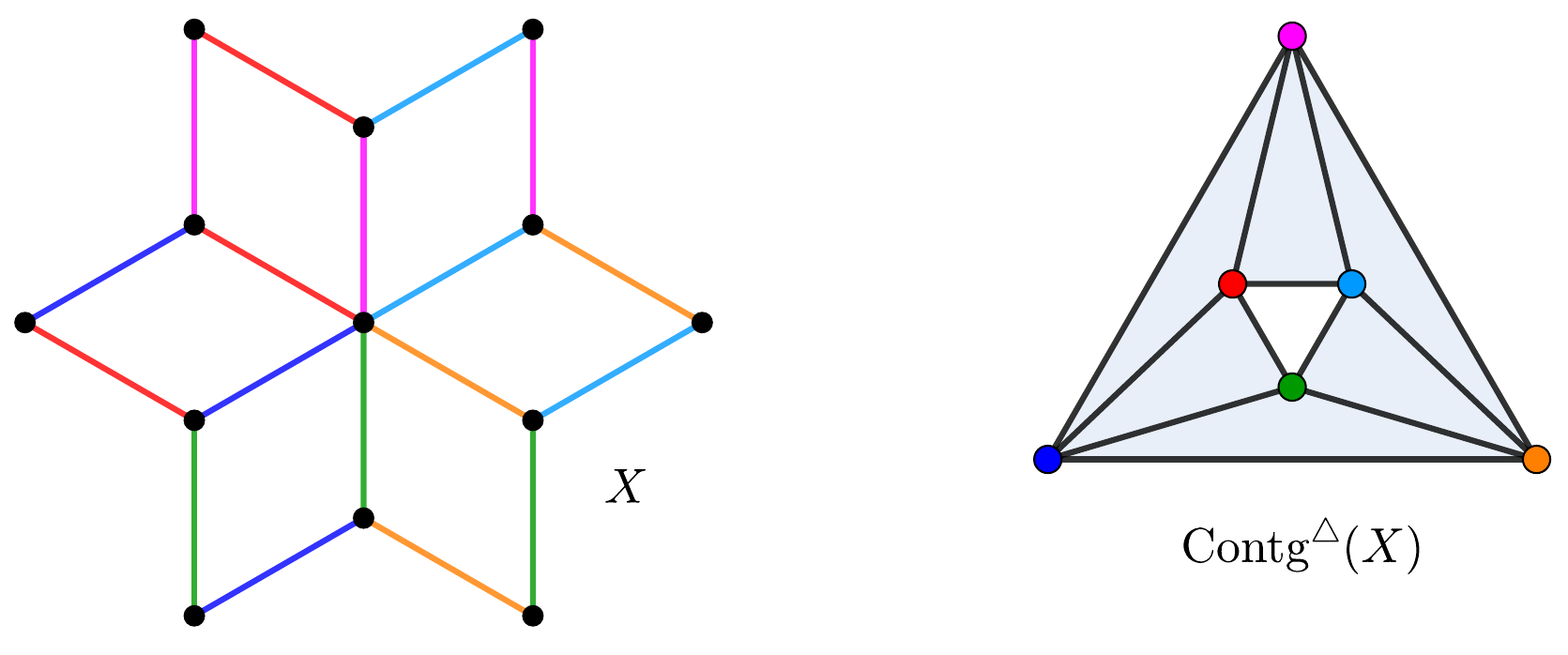}
\caption{A quasi-median graph and its contiguity complex.}
\label{Contg}
\end{center}
\end{figure}

\medskip \noindent
The main result of this section deals with contact complexes relative to \emph{star-covering} collections of gated subgraphs, i.e.\ for every clique $C$, there exists $Y \in \mathbb{G}$ such that $Y$ contains all the prisms containing $C$.  For a star-covering collection $\mathbb G$ and for every vertex $x \in X$, we let $\mathrm{sL}_\mathbb{G}(x)$ be the simplicial graph whose vertices are the cliques of $X$ containing $x$ and whose simplices are given by collections of cliques contained in a common subgraph from $\mathbb{G}$.  Let $\mathrm{L}_\mathbb{G}(x)$ denote the simplicial complex whose vertices are the edges of $X$ containing $x$ and whose simplices are given by collections of edges contained in a common subgraph from $\mathbb{G}$.

\begin{thm}\label{thm:RelContact}
Let $X$ be a $2$-connected quasi-median graph and $\mathbb{G}$ a star-covering collection of gated subgraphs. The relative contact complex $\mathrm{Cont}^\triangle(X,\mathbb{G})$ is homotopically equivalent to the pointed sum $\bigvee_{x \in X} \mathrm{sL}_\mathbb{G}(x).$  Equivalently, $\mathrm{Cont}^\triangle(X,\mathbb{G})$ is homotopy equivalent to the pointed sum $\bigvee_{x \in X} \mathrm{L}_\mathbb{G}(x)$. 
\end{thm}

\noindent
Notice that, when $\mathbb{G}= \{X\}$, each $\mathrm{sL}_\mathbb{G}(x)$ is a simplex, so our theorem is compatible  with Proposition~\ref{prop:ContactContractible}.  It is also worth noticing that the second pointed sum from our proposition makes sense because the complexes $\mathrm{L}_\mathbb{G}(x)$ are all connected:

\begin{lemma}\label{lem:ContConnected}
If $X$ is $2$-connected and $\mathbb{G}$ is prism-covering, then $\mathrm{L}_\mathbb{G}(x)$ is connected for every $x \in X$. 
\end{lemma}

\begin{proof}
Fix a vertex $x \in X$ and let $e,f$ be two edges of $X$ containing $x$. As the prism-completion $X^\square$ of $X$ is simply connected, a local cut-point in $X^\square$ is automatically a global cut-point. Therefore, since $X$ is assumed to be $2$-connected, links of vertices in $X^\square$ must be connected, which implies that there exists a sequence of edges containing $x$, say
$$e=a_1,\ a_2,\ \ldots, \ a_{n-1}, \ a_n=f,$$
such that $a_i$ and $a_{i+1}$ are contained in a common prism for every $1 \leq i \leq n-1$. Because $\mathbb{G}$ is prism-covering, every prism must be contained in a subgraph from $\mathbb{G}$. Consequently, $a_i$ and $a_{i+1}$ must be adjacent in $\mathrm{L}_\mathbb{G}(x)$ for every $1 \leq i \leq n-1$. Thus, we have proved that $e$ and $f$ can be connected by a path in $\mathrm{L}_\mathbb{G}(x)$. 
\end{proof}

\noindent
Before turning to the proof of Theorem~\ref{thm:RelContact}, we gather a few facts relating hyperplanes and relative contact complexes.

\begin{prop}\label{prop:GContactSimplices}
Let $X$ be a quasi-median graph and $\mathbb{G}$ be a collection of gated subgraphs. For a finite collection of hyperplanes $\mathcal{J}$ of $X$, the following statements are equivalent:
\begin{itemize}
     \item $\mathcal{J}$ is a simplex of $\mathrm{Cont}^\triangle(X,\mathbb{G})$;
     \item there exists $Y\in\mathbb{G}$ and   $x\in Y\cap   \bigcap_{J\in \mathcal{J}}N(J)$ such that $Y$ contains the clique of $J$ containing $x$  for each $J\in\mathcal{J}$.
\end{itemize}
Moreover, if 
$\mathbb{G}$ is a star-covering collection of gated subgraphs, the above statements are equivalent to: 
 \begin{itemize}
     \item for every $x\in \bigcap_{J\in \mathcal{J}}N(J)$, there exists  $Y \in \mathbb{G}$ that contains the clique of $J$ containing $x$ for each $J\in\mathcal{J}$.
 \end{itemize}
\end{prop} 

\begin{proof}
Suppose $\mathcal{J}=\{J_1,\ldots,J_n\}$ is a simplex of $\mathrm{Cont}^\triangle(X,\mathbb{G})$. Then there exists $Y\in \mathbb{G}$ such that $\{N(J_1),\ldots, N(J_n)\}\cup\{Y\}$ is a collection of pairwise intersecting gated subgraphs, and the Helly property implies that $Y\cap \bigcap_{i=1}^n N(J_i)$ is non-empty. Let $x\in Y\cap \bigcap_{i=1}^n N(J_i)$. Since $J_i$ crosses $Y$, the clique of $J_i$ containing $x$ is contained in $Y$ by Corollary~\ref{cor:CrossAndClique}. That the second and third statements imply the first one follows directly from the definitions.  

\medskip \noindent
Let us prove that the first statement implies the third one assuming that  $\mathbb{G}$ is star-covering. Let $\mathcal{J}= \{J_1, \ldots, J_n\}$ be a simplex of $\mathrm{Cont}^\triangle(X,\mathbb{G})$ and $x\in \bigcap_{J\in \mathcal{J}}N(J)$.  
The hyperplanes in $\mathcal{J}$ are pairwise in contact, and  there exists $Z \in \mathbb{G}$ such that   each $J\in\mathcal{J}$ crosses $Z$. If $x \in Z\cap \bigcap_{i=1}^nN(J)$, then
Corollary~\ref{cor:CrossAndClique} implies that the clique of $J$ containing $x$ is contained in $Z$ for each $J\in \mathcal{J}$; in this case, it suffices to set $Y:=Z$. Suppose that $x \not\in Z\cap \bigcap_{i=1}^nN(J_i)$ and let $J$ be a hyperplane separating $x$ from $Z$, which we can choose to be tangent to $x$. Let $C$ be the clique of $J$ containing $x$.  Since $\mathbb{G}$ is star-covering, there exists $Y \in \mathbb{G}$  containing all the prisms that contain $C$. For every $1 \leq i \leq n$, let $C_i$ be the clique of $J_i$ containing $x$. Since $J$ is necessarily transverse to $J_i$, Proposition~\ref{prop:InterHyperplanesProduct} implies that $C$ and $C_i$ span a prism containing $x$, and hence $C_i$ is contained in $Y$.
\end{proof}

\noindent 
The rest of the section is dedicated to the proof of Theorem~\ref{thm:RelContact}.   We fix a quasi-median graph $X$ and a collection $\mathbb{G}$ of gated subgraphs.  The equivalence of the two statements of Theorem~\ref{thm:RelContact} will be proven in Lemma~\ref{lem:LinkVsSlink} by showing that for each vertex $x\in X$, $\mathrm{sL}_\mathbb{G}(x)$ is homotopy equivalent to $\mathrm{L}_\mathbb{G}(x)$.

\medskip \noindent
Let $X^\odot$ be the \emph{perforation} of $X$, i.e.\ the space obtained from the prism-completion $X^\square$ of $X$ by removing a small open ball around each vertex of a fixed radius $\epsilon<1/2$, if we endow $X^\square$ with a length metric that extends the Euclidean metrics on its prisms. Given a vertex $x \in X$, the sphere $S(x,\epsilon)$ can be identified with the link of $x$ in the prism-completion $X^\square$. In other words, we can think of $S(x,\epsilon)$ as the simplicial complex whose vertices are the edges of $X$ containing $x$ and whose simplices are given by collections of edges contained in a common prism of $X$. 

\medskip \noindent
 When $\mathbb{G}$ is prism-covering, the complex $\mathrm{L}_\mathbb{G}(x)$ naturally contains $S(x,\epsilon)$ as a subcomplex, which allows us to define
$$X^\mathbb{G}:= X^\odot \cup \bigcup\limits_{x \in X} \mathrm{L}_\mathbb{G}(x)$$
where each $\mathrm{L}_\mathbb{G}(x)$ is glued to $X^\odot$ over $S(x,\epsilon)$. 

\medskip \noindent
We refer the reader to Figure~\ref{Perfo} for an illustration of our construction in a specific case. In order to prove Theorem~\ref{thm:RelContact}, we start by showing that our relative contact complex is homotopy equivalent to the specific complex just constructed:
\begin{figure}
\begin{center}
\includegraphics[width=0.9\linewidth]{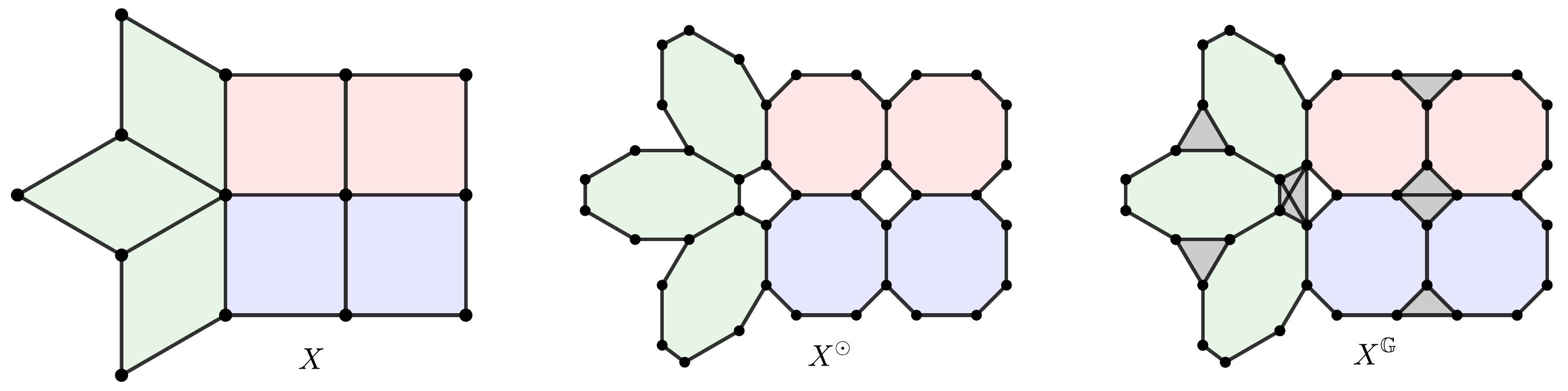}
\caption{A quasi-median graph $X$ endowed with a collection $\mathbb{G}$ of gated subgraphs, the perforation $X^\odot$, and the corresponding complex $X^\mathbb{G}$.}
\label{Perfo}
\end{center}
\end{figure}

\begin{prop}\label{prop:ContactXG}
If $\mathbb{G}$ is star-covering, then the $\mathbb{G}$-contact complex $\mathrm{Cont}^\triangle(X,\mathbb{G})$ is homotopy equivalent to $X^\mathbb{G}$.
\end{prop}

\begin{figure}
\begin{center}
\includegraphics[width=0.9\linewidth]{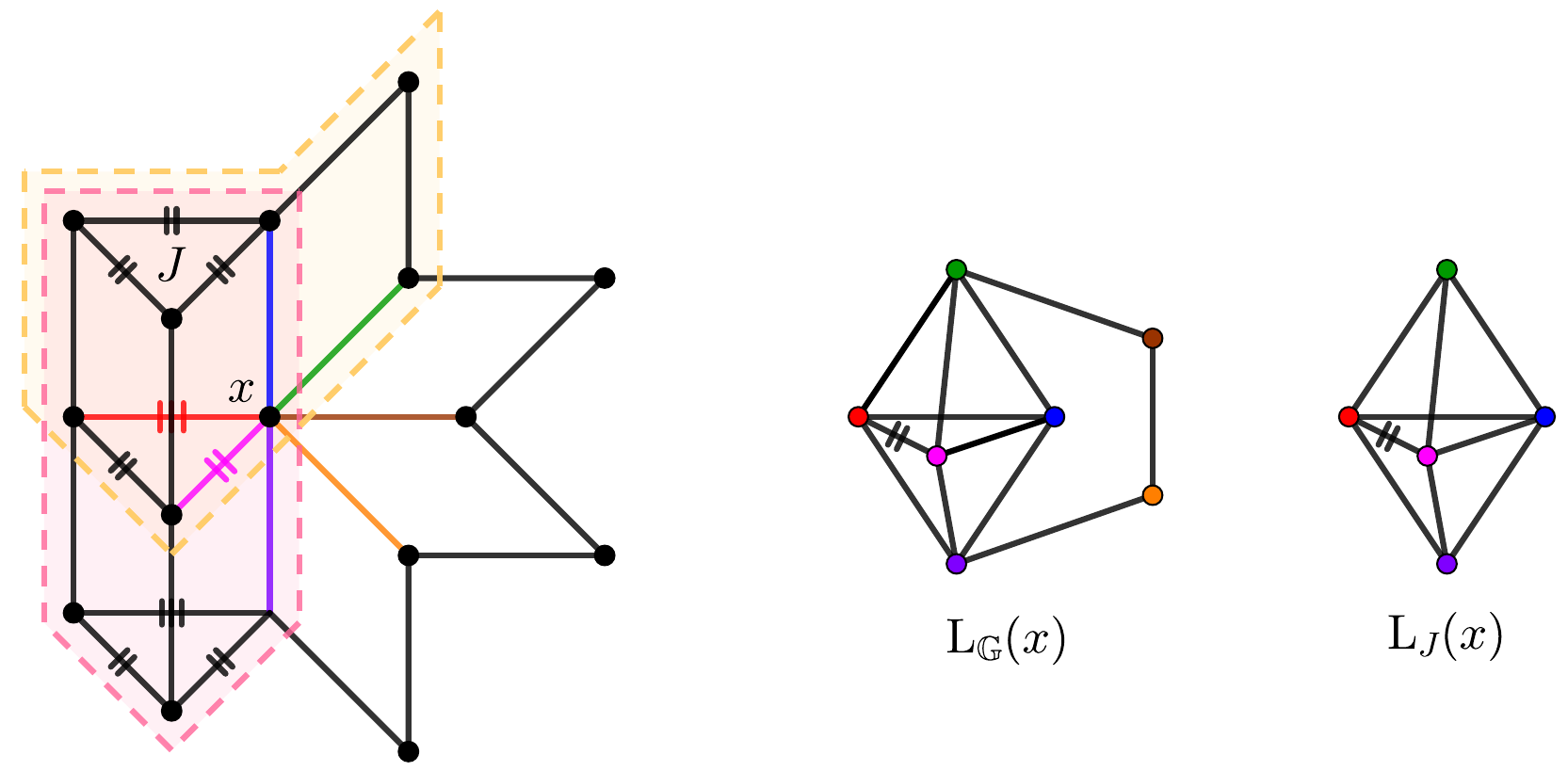}
\caption{The complexes $\mathrm{L}_\mathbb{G}(x)$ and $\mathrm{L}_J(x)$ where $\mathbb{G}$ is the collection of the prisms together with the yellow and pink subgraphs.}
\label{LJ}
\end{center}
\end{figure}
\begin{proof}
For all hyperplane $J$ and vertex $x \in N(J)$, let $L_J(x)$ denote the subcomplex of $\mathrm{L}_\mathbb{G}(x)$ given by the union of all the simplices containing the simplex corresponding to the clique of $J$ that contains $x$. See Figure~\ref{LJ}. Set 
$$S(J):= \left( N(J)^\square \cap X^\odot \right) \cup \bigcup\limits_{x \in N(J)} L_J(x).$$
Notice that $ \Omega:= \{ S(J) \mid J \text{ hyperplane}\}$ is a cover of $X^\mathbb{G}$. Indeed, let $p \in X^\mathbb{G}$ be an arbitrary point. If $p \in X^\odot$, then it belongs to the interior of some prism of cubical dimension $\geq 1$, hence $p \in S(J)$ where $J$ is an arbitrary hyperplane crossing this prism by Lemma~\ref{lem:PrismAbsorption}. Otherwise, if $p \in X^\mathbb{G} \backslash X^\odot$, then $p \in \mathrm{L}_\mathbb{G}(x)$ for some vertex $x \in X$. A simplex of $\mathrm{L}_\mathbb{G}(x)$ containing $p$ yields a subgraph $Y \in \mathbb{G}$ with an edge  containing $x$. Let $J$ be the hyperplane defined by this edge. Since $Y$ is gated, Proposition~\ref{lem:LocallyGated} implies that $Y$ contains the clique of $J$  containing $x$, and therefore $p \in S(J)$. 

\medskip \noindent
Our goal  is to apply the nerve theorem to our cover $\Omega$ of $X^\mathbb{G}$. First, let us verify that our nerve complex is isomorphic to $\mathrm{Cont}^\triangle(X, \mathbb{G})$. In other words, given a collection of hyperplanes $J_1, \ldots, J_n$, we claim that $S(J_1) \cap \cdots \cap S(J_n)$ is non-empty if and only if $J_1, \ldots, J_n$ span a simplex in $\mathrm{Cont}^\triangle(X, \mathbb{G})$. 

\medskip \noindent
If $J_1, \ldots, J_n$ span a simplex in $\mathrm{Cont}^\triangle(X, \mathbb{G})$, then there exist a subgraph $Y \in \mathbb{G}$ crossed by $J_1, \ldots, J_n$ and a vertex $x \in Y$ such that $x \in N(J_1) \cap \cdots \cap N(J_n)$. Then the simplex of $\mathrm{L}_\mathbb{G}(x)$ given by $Y$ is contained in $S(J_1) \cap \cdots \cap S(J_n)$; see Corollary~\ref{cor:CrossAndClique}. Conversely, assume that $S(J_1) \cap \cdots \cap S(J_n)$ is non-empty, and fix a point $p$ in this intersection. If $p \in X^\odot$, then the smallest prism containing $p$, say $P$, must be contained in the carriers of $J_1, \ldots, J_n$. In other words, each $J_i$ crosses a prism containing $P$; see Lemma~\ref{lem:GoodPrism}. Since $\mathbb{G}$ is star-covering, there exists some $Y \in \mathbb{G}$ that contains all the prisms containing $P$. In particular, all $J_i$ cross $Y$. Clearly, $J_1, \ldots, J_n$ are pairwise in contact in $Y$, proving that they span a simplex in $\mathrm{Cont}^\triangle(X, \mathbb{G})$. Otherwise, if $p \notin X^\odot$, then there exists some $x \in X$ such that $p \in \mathrm{L}_\mathbb{G}(x)$. The smallest simplex of $\mathrm{L}_\mathbb{G}(x)$ containing $p$ corresponds to a subgraph $Y \in \mathbb{G}$ satisfying $x \in Y$. Since this simplex must belong to $L_{J_i}(x)$ for every $1 \leq i \leq n$, it follows that $J_1, \ldots, J_n$ all cross $Y$. This implies that $J_1, \ldots, J_n$ span a simplex in $\mathrm{Cont}^\triangle(X,\mathbb{G})$. 

\medskip \noindent
Finally, given a collection of hyperplanes $J_1, \ldots, J_n$ spanning a simplex in $\mathrm{Cont}^\triangle(X,\mathbb{G})$, let us verify that the intersection $S(J_1) \cap \cdots \cap S(J_n)$ is contractible. We can write our intersection as 
\begin{equation}\label{Decomposition}
\left( \left( N(J_1) \cap  \cdots \cap N(J_n) \right) ^\square \cap X^\odot \right) \cup \bigcup\limits_{x \in N(J_1) \cap \cdots \cap N(J_n)}  \bigcap\limits_{i=1}^n L_{J_i}(x).
\end{equation}
The picture to keep in mind is that $(N(J_1) \cap \cdots\cap N(J_n))^\square \cap X^\odot$ is the perforation of $(N(J_1) \cap \cdots \cap N(J_n))^\square$, and that, for every vertex $x \in N(J_1) \cap \cdots \cap N(J_n)$, we glue $L_{J_1}(x) \cap \cdots \cap L_{J_n}(x)$ on the link of $x$ (when thinking of $\mathrm{L}_\mathbb{G}(x)$ as containing the link $S(x,\epsilon)$ of $x$). Our goal is to show that this gluing operation amounts, up to homotopy, to gluing a cone over the link of $x$ in $(N(J_1) \cap \cdots \cap N(J_n))^\square\cap X^\odot$, proving that our intersection is homotopically equivalent to $(N(J_1) \cap \cdots \cap N(J_n))^\square$, and therefore contractible. 


\medskip \noindent
 Let us argue that $L_{J_1}(x) \cap \cdots \cap L_{J_n}(x)$ is contractible. 
 Since $\mathbb{G}$ is star-covering, Proposition~\ref{prop:GContactSimplices} implies that there is 
$Y \in \mathbb{G}$ containing the clique of $J_i$ containing $x$, for every $J_i$.
It follows that, in $\mathrm{L}_\mathbb{G}(x)$, the subcomplex $S$ induced by the edges starting from $x$ and crossed by one of the  $J_1, \ldots, J_n$ is a simplex, since all these edges belong to $Y$. The intersection $L_{J_1}(x) \cap \cdots \cap L_{J_n}(x)$ then corresponds to the union of all the simplices containing this simplex $S$. Thus our intersection deformation retracts onto  $S$. 

\medskip \noindent
 Regard $S(x,\epsilon)$ as a subcomplex of $\mathrm{L}_\mathbb{G}(x)$. Let us show that $S(x,\epsilon)\cap (N(J_1) \cap \cdots \cap N(J_n))^\square$ is a subcomplex of $L_{J_1}(x) \cap \cdots \cap L_{J_n}(x)$. 
Let $\sigma$ be a simplex of $S(x,\epsilon)\cap (N(J_1) \cap \cdots \cap N(J_n))^\square$. Then $\sigma$ consists of all edges containing $x$ 
contained in a prism $P$ of $N(J_1) \cap \cdots \cap N(J_n)$. 
For each $J_i$, the prism $P$
is contained in a prism $Q_i \subset N(J_i)$ such that $Q_i$ is crossed by $J_i$; see Lemma~\ref{lem:GoodPrism}. Since prisms are gated, Corollary~\ref{cor:CrossAndClique} implies that $Q_i$ contains the clique of $J_i$ containing $x$. The prism $Q_i$ thus induces a simplex of $L_{J_i}$ that contains the simplex
$\sigma$ as a face. It follows that $\sigma$ is is simplex of
$L_{J_1}(x) \cap \cdots \cap L_{J_n}(x)$. 

\medskip \noindent
Therefore, in the decomposition (\ref{Decomposition}), adding the intersection $L_{J_1}(x) \cap \cdots \cap L_{J_n}(x)$ to the perforation of $(N(J_1) \cap \cdots \cap N(J_n))^\square$ amounts to putting back the ball of radius $\epsilon$ centred at $x$, proving that $S(J_1) \cap \cdots \cap S(J_n)$ is homotopy equivalent to $(N(J_1) \cap \cdots \cap N(J_n))^\square$, and hence contractible. 
\end{proof}

\noindent
Thanks to Proposition~\ref{prop:ContactXG}, we are now ready to identify our relative contact complex as a pointed sum, namely: 

\begin{prop}\label{prop:ContPointedSum}
If $X$ is $2$-connected and $\mathbb{G}$ star-covering, then $\mathrm{Cont}^\triangle(X,\mathbb{G})$ is homotopy equivalent to $\bigvee_{x \in X} \mathrm{L}_\mathbb{G}(x)$. 
\end{prop}

\noindent

\begin{proof}[Proof of Proposition~\ref{prop:ContPointedSum}.]
We know from Proposition~\ref{prop:ContactXG} that $\mathrm{Cont}^\triangle (X,\mathbb{G})$ is homotopy equivalent to $X^\mathbb{G}$, so we  focus on $X^\mathbb{G}$. Fix a maximal subtree $T \subset X$. Fact~\ref{fact:AlgTopo} implies that
\[  
 X^\odot =  
 X^\square\setminus \bigcup_{x\in X} B(x,\epsilon)
\]
deformation retracts onto
\[ T^\odot := \left( T\setminus \bigcup_{x\in X} B(x,\epsilon)\right)\cup\bigcup_{x\in X} S(x,\epsilon),\]
since the contractible space
$X^\square$ is the union of $X^\odot$ and the contractible subspace $T \cup \bigcup_{x\in X} \overline{B(x,\epsilon)}$, and 
$T^\odot$ is the intersection of $X^\odot$ and $T \cup \bigcup_{x\in X} \overline{B(x,\epsilon)}$.   It follows that there is a deformation retraction of 
\[ X^\mathbb{G}= X^\odot \cup \bigcup_{x\in X} L_\mathbb{G}(x) \]
onto 
\[ T^\mathbb{G} := T^\odot \cup \bigcup_{x\in X} L_\mathbb{G}(x) .\]
 Observe that  $T^\mathbb{G}$ is  a tree of spaces whose edge-spaces are single points and whose vertex-spaces are the $\mathrm{L}_\mathbb{G}(x)$. Since the $\mathrm{L}_\mathbb{G}(x)$ are connected according to Lemma~\ref{lem:ContConnected}, we conclude that $X^\mathbb{G}$, and a fortiori $\mathrm{Cont}^\triangle (X, \mathbb{G})$, is homotopy equivalent the $\bigvee_{x \in X} \mathrm{L}_\mathbb{G}(x)$. 
\end{proof}



\noindent
In order to conclude the proof of Theorem~\ref{thm:RelContact}, it only remains to show that the complexes from the two pointed sums in the statement are homotopy equivalent, namely:

\begin{lemma}\label{lem:LinkVsSlink}
Let $X$ be a quasi-median graph and $\mathbb{G}$ a prism-covering collection of gated subgraphs. For every vertex $x \in X$, $\mathrm{L}_\mathbb{G}(x)$ is homotopy equivalent to $\mathrm{sL}_\mathbb{G}(x)$. 
\end{lemma}

\begin{proof}
For every clique $C$ of $X$ containing $x$, let $\Delta(C)$ denote the simplex of $\mathrm{L}_\mathbb{G}(x)$ given by the edges of $C$ containing $x$, and let $\Delta^+(C)$ denote the union of $\Delta(C)$ with the interiors of all the simplices in $\mathrm{L}_{\mathbb G}(x)$ containing $\Delta(C)$. Our goal is to apply the nerve theorem to $\{\Delta^+(C) \mid C \text{ clique containing } x\}$. 

\begin{claim}\label{claim:NerveDeltaOne}
For any collection of cliques $C_1, \ldots, C_n$ containing $x$, the intersection $\Delta^+(C_1)\cap \cdots \cap \Delta^+(C_n)$ is non-empty if and only if $C_1, \ldots, C_n$ span a prism contained in a subgraph from $\mathbb{G}$.
\end{claim}

\noindent
We can assume without loss of generality that $C_1, \ldots, C_n$ are pairwise distinct, which implies that $\Delta(C_1), \ldots, \Delta(C_n)$ are pairwise disjoint by Proposition~\ref{lem:LocallyGated}. A point in $\Delta^+(C_1) \cap \cdots \cap \Delta^+(C_n)$ must then belong to the interiors of $n$ simplices containing  $\Delta(C_1), \ldots, \Delta(C_n)$ respectively. But, in a simplicial complex, two simplices whose interiors intersect must coincide, so it follows that there must exist a simplex in $\mathrm{L}_\mathbb{G}(x)$ containing all $\Delta(C_1), \ldots, \Delta(C_n)$. We conclude that $C_1, \ldots, C_n$ are contained in a common prism lying in some subgraph from $\mathbb{G}$. Conversely, it is clear that, if $C_1, \ldots, C_n$ are contained in a common prism lying in some subgraph from $\mathbb{G}$, say $P$, then $P$ defines a simplex in $\mathrm{L}_\mathbb{G}(x)$ containing all $\Delta(C_1), \ldots, \Delta(C_n)$. 

\begin{claim}\label{claim:NerveDeltaTwo}
If $C_1, \ldots, C_n$ are cliques containing $x$, then $\Delta^+(C_1) \cap \cdots \cap \Delta^+(C_n)$ is either empty or contractible.
\end{claim}

\noindent
Assume that the intersection is non-empty. Also, we assume without loss of generality that $C_1, \ldots, C_n$ are pairwise distinct. If $n=1$, then the conclusion is clear, so we assume that $n \geq 2$. It follows from Claim~\ref{claim:NerveDeltaOne} that $C_1,\ldots, C_n$ pairwise span a prism contained in a subgraph from $\mathbb{G}$, from which we deduce that $C_1, \ldots, C_n$ span a prism, say $P$. Let $\Delta(P)$ denote the simplex of $\mathrm{L}_\mathbb{G}(x)$ given by $P$. This is the simplex spanned by $\Delta(C_1), \ldots, \Delta(C_n)$. A point in $\Delta^+(C_1) \cap \cdots \cap \Delta^+(C_n)$ has to belong to the interiors of simplices containing respectively $\Delta(C_1), \ldots, \Delta(C_n)$. As before, using the fact that, in a simplicial complex, two simplices whose interiors intersect must coincide, we deduce that $\Delta^+(C_1) \cap \cdots \cap \Delta^+(C_n)$ coincides with the union of the interiors of all the simplices that contain $\Delta(C_1), \ldots, \Delta(C_n)$, or equivalently, that contain $\Delta(P)$. Thus this subspace must be contractible, concluding the proof of Claim~\ref{claim:NerveDeltaTwo}.

\medskip \noindent
As a consequence of Claim~\ref{claim:NerveDeltaTwo}, the nerve theorem applies to the covering $\{\Delta^+(C) \mid C \text{ clique containing } x\}$, and so $\mathrm{L}_\mathbb{G}(x)$ is homotopy equivalent to the corresponding nerve complex. According to Claim~\ref{claim:NerveDeltaOne}, the latter coincides with $\mathrm{sL}_\mathbb{G}(x)$. 
\end{proof}

\begin{proof}[Proof of Theorem~\ref{thm:RelContact}.]
The theorem is an immediate consequence of Proposition~\ref{prop:ContPointedSum} and Lemma~\ref{lem:LinkVsSlink}. 
\end{proof}

\section{Applications to coset intersection complexes}

\noindent
Given a graph $\Gamma$ and a collection of groups $\mathcal{G}= \{ G_u \mid u \in V(\Gamma)\}$ indexed by the vertices of $\Gamma$, the \emph{graph product} $\Gamma \mathcal{G}$ is defined as the quotient
$$\left( \Conv_{u \in V(\Gamma)} G_u \right) / \langle\langle [g,h] = 1, \ \{u,v\} \in E(\Gamma), g \in G_u, h \in G_v \rangle\rangle,$$
where $E(\Gamma)$ denotes the edge-set of $\Gamma$ \cite{GreenGP}. Usually, one says that graph products interpolate between direct sums (when the vertex-groups pairwise commute, i.e.\ $\Gamma$ is a complete graph) and free products (when no two distinct vertex-groups commute, i.e.\ $\Gamma$ has no edge). They include right-angled Artin groups (when all the vertex-groups are infinite cyclic) and right-angled Coxeter groups (when all the vertex-groups are cyclic of order two). 

\medskip \noindent
\textbf{Convention.} In all our article, vertex-groups of graph products will be always assumed to be non-trivial. There is no real loss of generality since every graph product can be naturally described as a graph product all of whose vertex-groups are non-trivial (and isomorphic to the previous vertex-groups---we just need to remove the trivial vertex-groups). 

\begin{definition}
Let $\Gamma$ be a graph and $\mathcal{G}$ a collection of groups indexed by $V(\Gamma)$. A \emph{parabolic subgroup} of $\Gamma \mathcal{G}$ is a subgroup of the form
$$\langle \Lambda \rangle : = \langle G_u, \ u \text{ vertex of } \Lambda \rangle$$
for some subgraph $\Lambda$ of $\Gamma$.
\end{definition}

\noindent
 In this section, our goal is to investigate the structure of certain \emph{coset intersection complexes} of graph products (with respect to some parabolic subgroups), defined by the first and third authors in \cite{CIC}.

\begin{definition}
Let $G$ be a group and $\mathcal{P}$ a collection of subgroups. The \emph{coset intersection complex} $\mathcal{K}(G,\mathcal{P})$ is the simplicial complex whose vertices are the cosets of subgroups from $\mathcal{P}$ and whose simplices are given by collections $g_1P_1, \ldots, g_nP_n$ of cosets such that $g_1P_1g_1^{-1} \cap \cdots \cap g_n P_n g_n^{-1}$ is infinite. 
\end{definition}

\noindent
Our study will be based on the quasi-median geometry of graph products combined with the results proved in the previous sections of this article.

\paragraph{Quasi-median geometry of graph products.} As shown in \cite{QM, QMspecial}, there is a close connection between quasi-median graphs and graph products of groups. On the one hand, every graph product $\Gamma \mathcal{G}$ admits a natural quasi-median Cayley graph, namely
$$\mathrm{QM}(\Gamma, \mathcal{G}):= \mathrm{Cayl} \left( \Gamma \mathcal{G}, \bigcup\limits_{u \in V(\Gamma)} G_u \backslash \{1\} \right).$$
Conversely, it can be proved that every quasi-median graph can be realised as a gated subgraph of such a Cayley graph. 

\medskip \noindent
The cliques and prisms of $\mathrm{QM}(\Gamma, \mathcal{G})$ can be described as follows. We refer the reader to \cite[Section~8]{QM} for more details.

\begin{lemma}\label{lem:CliquesPrismsGP}
Let $\Gamma$ be a graph and $\mathcal{G}$ a collection of groups indexed by the vertices of $\Gamma$. The cliques of $\mathrm{QM}(\Gamma, \mathcal{G})$ are given by the cosets $g \langle u \rangle$ where $g \in \Gamma \mathcal{G}$ and $u \in V(\Gamma)$. The prisms of $\mathrm{QM}(\Gamma, \mathcal{G})$ are given by the cosets $g \langle \Lambda \rangle$ where $g \in \Gamma \mathcal{G}$ and  $\Lambda \subset \Gamma$ is a complete subgraph.
\end{lemma}

\begin{cor}
Let $\Gamma$ be a graph and $\mathcal{G}$ a collection of groups indexed by the vertices of $\Gamma$. The cubical dimension of $\mathrm{QM}(\Gamma, \mathcal{G})$ is $\mathrm{clique}(\Gamma):= \max \{ \# \Lambda \mid \Lambda \subset \Gamma \text{ complete}\}$. 
\end{cor}

\noindent
The structure of hyperplanes can also be described. A useful observation is that all the edges of a given hyperplane are labelled by generators coming from the same vertex-group. Consequently, hyperplanes are naturally labelled by the vertices of $\Gamma$. 

\begin{lemma}[{\cite[Lemma~8.12]{QM}}]\label{lem:LabelHyp}
Let $\Gamma$ be a graph and $\mathcal{G}$ a collection of groups indexed by the vertices of $\Gamma$. In $\mathrm{ QM}(\Gamma, \mathcal{G})$, two transverse hyperplanes have adjacent labels.
\end{lemma}

\noindent
Given a group $G$ acting on a quasi-median graph, the \emph{rotative-stabiliser} of a hyperplane $J$ is
$$\mathrm{stab}_\circlearrowright(J):= \{ g \in G \mid gC=C \text{ for every clique } C \subset J \}.$$
We again refer the reader to \cite[Section~8]{QM} for more details. 

\begin{lemma}\label{lem:HypInQMGP}
Let $\Gamma$ be a graph and $\mathcal{G}$ a collection of groups indexed by the vertices of $\Gamma$. Fix a vertex $u \in \Gamma$, and let $J_u$ denote the hyperplane of $\mathrm{QM}(\Gamma, \mathcal{G})$ containing the clique $\langle u\rangle$. The following hold.
\begin{itemize}
	\item The carrier of $J_u$ is the subgraph $\langle \mathrm{star}(u) \rangle \subset \mathrm{QM}(\Gamma, \mathcal{G})$.
	\item The stabiliser of $J_u$ is the subgroup $\langle \mathrm{star}(u) \rangle \leq \Gamma \mathcal{G}$. 
	\item The rotative-stabiliser of $J_u$ is the subgroup $\langle u \rangle \leq \Gamma \mathcal{G}$. It freely and transitively permutes the sectors delimited by $J_u$. 
\end{itemize}
\end{lemma}

\noindent
Let us conclude our preliminaries with the following straightforward observation:

\begin{lemma}\label{lem:MQTwoConnected}
Let $\Gamma$ be a graph and $\mathcal{G}$ a collection of groups indexed by $\Gamma$. The graph $\mathrm{QM}(\Gamma, \mathcal{G})$ is $2$-connected if and only if $\Gamma$ is connected.
\end{lemma}

\begin{proof}
By \cite[Claim~3.3]{GTranslation},  a vertex of $\mathrm{QM}(\Gamma, \mathcal{G})$ is a cut-vertex if and only if its simplified link is disconnected. But, according to Lemma~\ref{lem:CliquesPrismsGP}, the cliques containing a vertex $x$ are the cosets $x \langle u\rangle$, with $u \in \Gamma$, and two such cliques $x \langle u \rangle$ and $x \langle v \rangle$ span a prism if and only if $u$ and $v$ are adjacent in $\Gamma$. Therefore, the one-skeleton of the simplified link of $x$ is isomorphic to $\Gamma$. The desired conclusion follows. 
\end{proof}

\paragraph{A word about the crossing complex.} The crossing complex $\mathrm{Cross}^\triangle (\Gamma, \mathcal{G})$ of the quasi-median graph $\mathrm{QM}(\Gamma, \mathcal{G})$ can be described purely algebraically as the simplicial complex whose vertices are the conjugates of the vertex-groups and whose simplices are given by pairwise commuting subgroups. 

\medskip \noindent
It is worth mentioning that, when restricting to right-angled Artin groups, the crossing graph coincides with Kim and Koberda's \emph{extension graph}. In \cite{MR3039768}, extension graphs are introduced and exploited to deal with the \emph{embedding problem} for right-angled Artin groups: given two graphs $\Phi$ and $\Psi$, is it possible to determine whether $A(\Phi)$ is isomorphic to a subgroup of $A(\Psi)$? The results in \cite{MR3039768} are reproved and generalised to graph products by the second-named author in \cite{QM} from a geometric perspective thanks to the quasi-median geometry of graph products. Thus, our crossing graphs here can be thought of as a generalisation of the extension graphs from \cite{MR3039768}. (Note that, independently, this generalisation appears in \cite{ThGP}.)

\subsection{Clique subgroups}

\noindent
We start our investigation of coset intersection complexes of graph products with respect to parabolic subgroups by considering clique subgroups. Our main result in this direction is:

\begin{thm}\label{thm:IntroCosetInterComplex}
Let $\Gamma$ be a finite connected graph and $\mathcal{G}$ a collection of infinite groups indexed by the vertices of $\Gamma$. Set $\mathcal{P}:= \{ \langle \Lambda \rangle \mid \Lambda \subset \Gamma \text{ clique}\}$. The coset intersection complex $\mathcal{K}(\Gamma \mathcal{G},\mathcal{P})$ is a quasi-tree and is homotopy equivalent to a bouquet of infinitely many copies of the flag completion $\Gamma^\triangle$ of $\Gamma$. 
\end{thm}

\noindent
The key observation is that the coset intersection complex we are interested in turns out to be closely connected to the crossing complex. More precisely:

\begin{prop}\label{prop:CosetInterCrossing}
Let $\Gamma$ be a finite connected graph and $\mathcal{G}$ a collection of infinite groups indexed by the vertices of $\Gamma$. Set $\mathcal{P}:= \{ \langle \Lambda \rangle \mid \Lambda \subset \Gamma \text{ clique}\}$. The coset intersection complex $\mathcal{K}(\Gamma \mathcal{G},\mathcal{P})$ is quasi-isometric and homotopy equivalent to the crossing complex $\mathrm{Cross}^\triangle(\Gamma, \mathcal{G})$ of the quasi-median graph $\mathrm{QM}(\Gamma , \mathcal{G})$. 
\end{prop}

\noindent
We start by proving the following preliminary observation:

\begin{lemma}\label{lem:InterStabPrisms}
Let $\Gamma$ be a graph and $\mathcal{G}$ a collection of groups indexed by the vertices of $\Gamma$. For every collection of prisms $\{ P_i \subset \mathrm{QM}(\Gamma, \mathcal{G}), \ i \in I \}$, 
$$\bigcap\limits_{i \in I} \mathrm{stab}(P_i) = \langle \mathrm{stab}_\circlearrowright (J), \ J \text{ crosses all the $P_i$} \rangle.$$
\end{lemma}

\begin{proof}
We begin by proving a couple of general observations.

\begin{claim}\label{claim:PrismStabRot}
For every prism $P$ in $\mathrm{QM}(\Gamma, \mathcal{G})$,
$$\mathrm{stab}(P)= \bigoplus\limits_{J \text{ hyp.\ inter.\ } P} \mathrm{stab}_\circlearrowright (J).$$
\end{claim}

\noindent
According to Lemma~\ref{lem:CliquesPrismsGP}, up to translating $P$ by an element of $\Gamma \mathcal{G}$, we can write $P= \langle \Lambda \rangle$ for some complete subgraph $\Lambda \subset \Gamma$. Then, it follows from Lemma~\ref{lem:HypInQMGP} that
$$\mathrm{stab}(P)= \langle \Lambda \rangle = \bigoplus\limits_{u \in \Lambda} \langle u \rangle = \bigoplus\limits_{u \in \Lambda} \mathrm{stab}_\circlearrowright (J_u).$$
Since the $J_u$ for $u \in \Lambda$ are exactly the hyperplanes intersecting $P$, the desired conclusion follows.

\begin{claim}\label{claim:RotStabSector}
If $K$ and $L$ are two transverse hyperplanes, then $\mathrm{stab}_\circlearrowright (L)$ stabilises each sector delimited by $K$. 
\end{claim}

\noindent
Because $K$ and $L$ are transverse, there must exist a prism crossed by both $K$ and $L$. According to Lemma~\ref{lem:CliquesPrismsGP}, up to translating such a prism by an element of $\Gamma \mathcal{G}$, we can assume that $K=J_u$ and $L=J_v$ for some adjacent vertices $u,v \in \Gamma$. We already know from Lemma~\ref{lem:HypInQMGP} that $\mathrm{stab}_\circlearrowright(L)=\langle v \rangle$ stabilises $K$. Given an element $g \in \mathrm{stab}_\circlearrowright(L)$ and a vertex $x \in \langle u \rangle$, there is a $4$-cycle $(1,x,g \cdot x,g \cdot 1)$ in $\mathrm{QM}(\Gamma , \mathcal{G})$. Since the edges $[1,x]$ and $[g \cdot 1,g \cdot x]$ both belong to $J_u=K$, the vertices $x$ and $g \cdot x$ belong to the same sector delimited by $K$. Thus, $g$ stabilises the sector delimited by $K$ that contains $x$. Since this is true for every element $g \in \mathrm{stab}_\circlearrowright(L)$ and every vertex $x \in \langle u \rangle$, this concludes the proof of Claim~\ref{claim:RotStabSector}.

\medskip \noindent
Now, let us prove the equality from the statement of the lemma. The inclusion 
$$\langle \mathrm{stab}_\circlearrowright (J), \ J \text{ crosses all the $P_i$} \rangle \subset \bigcap\limits_{i \in I} \mathrm{stab}(P_i)$$
follows immediately from Claim~\ref{claim:PrismStabRot}. For the reverse inclusion, let $g \in \bigcap_{i \in I} \mathrm{stab}(P_i)$ be an element. Fix an $\ell \in I$. Let $J_1, \ldots, J_n$ denote the hyperplanes crossing all the $P_i$, and let $H_1, \ldots, H_m$ denote the hyperplanes crossing $P_\ell$ but not all the $P_i$. According to Claim~\ref{claim:PrismStabRot}, we can decompose $g$ as a product $g_1 \cdots g_n h_1 \cdots h_m$ where $g_i \in \mathrm{stab}_\circlearrowright(J_i)$ for every $1 \leq i \leq n$ and $h_i \in \mathrm{stab}_\circlearrowright(H_i)$ for every $1 \leq i \leq m$. Moreover, the $g_i$ and $h_j$ pairwise commute. Assume for contradiction that one $h_j$ is non-trivial, say $h_1$. We know that $H_1$ does not cross some $P_k$, $k \in I$. Because $h_1$ sends the sector delimited by $H_1$ that contains $P_k$ to a distinct sector delimited by $H_1$, we know that $H_1$ separates $P_k$ and $h_1P_k$. We deduce from Claim~\ref{claim:RotStabSector} that $H_1$ also separates $P_k$ and $gP_k$. Thus, $g$ does not stabilise $P_k$, contrary to our assumption. Therefore, all the $h_j$ must be trivial, and so $g$ must belong to $\langle \mathrm{stab}_\circlearrowright(J_i), \ 1 \leq i \leq n \rangle$, as desired.
\end{proof}

\begin{proof}[Proof of Proposition~\ref{prop:CosetInterCrossing}.]
As a consequence of Lemma~\ref{lem:CliquesPrismsGP}, the cosets of the form $g \langle \Lambda \rangle$ with $\Lambda \subset \Gamma$ a clique coincide with the maximal prisms in $\mathrm{QM}(\Gamma, \mathcal{G})$. Hence, according to Theorem~\ref{thm:PrismsHyp}, there is a natural bijection between such cosets and the maximal collections of pairwise transverse hyperplanes in $\mathrm{QM}(\Gamma, \mathcal{G})$; or equivalently, with the maximal simplices in the crossing complex $\mathrm{Cross}^\triangle(\Gamma, \mathcal{G})$. Notice that, as a consequence of Lemma~\ref{lem:InterStabPrisms}, for any collection of cliques $\{ \Phi_i \subset \Gamma, \ i \in I \}$ and any collection of elements $\{ g_i \in \Gamma \mathcal{G}, \ i \in I\}$, the conjugates $g_i \langle \Phi_i \rangle g_i^{-1}$ have an infinite total intersection if and only if there exists a hyperplane crossing all the maximal prisms $g_i \langle \Phi_i \rangle$, which amounts to saying that the corresponding maximal simplices in $\mathrm{Cross}^\triangle(\Gamma, \mathcal{G})$ intersect.

\medskip \noindent
Thus, we have proved that the coset intersection complex $\mathcal{K}(\Gamma \mathcal{G}, \mathcal{P})$ coincides with the nerve complex of $\mathrm{Cross}^\triangle(\Gamma, \mathcal{G})$ relative to its maximal simplices. This implies that $\mathcal{K}(\Gamma \mathcal{G}, \mathcal{P})$ and $\mathrm{Cross}^\triangle(\Gamma, \mathcal{G})$ are both quasi-isometric and homotopy equivalent, by Leray's nerve theorem. 
\end{proof}

\noindent
In order to deduce Theorem~\ref{thm:IntroCosetInterComplex}, it remains to describe the crossing complexes of the quasi-median graphs associated to graph products.

\begin{prop}\label{prop:CrossingQMGamma}
Let $\Gamma$ be a connected graph and $\mathcal{G}$ a collection of groups indexed by the vertices of $\Gamma$. The crossing complex of $\mathrm{QM}(\Gamma, \mathcal{G})$ is homotopy equivalent to a bouquet of infinitely many copies of $\Gamma^\triangle$, the flag completion of $\Gamma$. 
\end{prop}

\begin{proof}
As a consequence of Theorem~\ref{thm:HomotopyCrossingComplex} and Lemma~\ref{lem:Loc}, it suffices to verify that, for every vertex $x$ of $\mathrm{QM}(\Gamma,\mathcal{G})$, $\mathrm{slink}(x)$ is homeomorphic to the flag completion $\Gamma^\triangle$ of $\Gamma$. Since $\Gamma \mathcal{G}$ acts vertex-transitively on $\mathrm{QM}(\Gamma, \mathcal{G})$, we can assume without loss of generality that $x=1$. By definition, $\mathrm{slink}(1)$ coincides with the simplicial complex whose vertices are the cliques containing $1$ and whose simplices are given by collections of cliques spanning prisms. But, according to Lemma~\ref{lem:CliquesPrismsGP}, these cliques are given by the cosets $\langle u \rangle$ for $u \in V(\Gamma)$, and a collection of cliques $\langle u_1 \rangle, \ldots, \langle u_n \rangle$ span a prism if and only if $u_1, \ldots, u_n$ span a complete subgraph in $\Gamma$. This is the description of the flag completion of $\Gamma$. 
\end{proof}

\begin{proof}[Proof of Theorem~\ref{thm:IntroCosetInterComplex}.]
The desired conclusion follows by combining Theorem~\ref{thm:CrossQM} with Propositions~\ref{prop:CosetInterCrossing} and~\ref{prop:CrossingQMGamma}. 
\end{proof}

\subsection{Other parabolic subgroups}

\noindent
We now turn to other parabolic subgroups of graph products. The main result of this section is the following:

\begin{thm}\label{thm:CosetInterComplexRelativeCont}
Let $\Gamma$ be a connected graph and $\mathcal{G}$ a collection of infinite groups indexed by $\Gamma$. Fix a collection $\mathscr{S}$ of subgraphs of $\Gamma$ such that  
\begin{itemize}
	\item for every $v \in \Gamma$, there exists $\Lambda \in \mathscr{S}$ such that $\mathrm{star}(v) \subset \Lambda$.
\end{itemize}
The coset intersection complex $\mathcal{K}(\Gamma \mathcal{G}, \{ \langle \Lambda \rangle ,\Lambda \in \mathscr{S} \})$ is homotopy equivalent to the pointed sum of infinitely many copies of the simplicial complex $\Gamma^\mathscr{S}$ whose vertex-set is $\Gamma$ and whose simplices are given by the subgraphs in $\mathscr{S}$. 
\end{thm}

\noindent
For instance, the theorem applies to maximal join subgroups, which we will consider in Section~\ref{section:QIinvariant}. The first step towards the proof of Theorem~\ref{thm:CosetInterComplexRelativeCont} is to give a geometric interpretation of the coset intersection complexes under consideration. This will be possible thanks to the following observation:

\begin{lemma}\label{lem:ForSkewer}
Let $\Gamma$ be a graph and $\mathcal{G}$ a collection of infinite groups indexed by $\Gamma$. For all subgraphs $\Lambda_1, \ldots, \Lambda_n \subset \Gamma$ and all elements $g_1, \ldots, g_n \in \Gamma \mathcal{G}$, the intersection
$$g_1 \langle \Lambda_1 \rangle g_1^{-1} \cap \cdots \cap g_n \langle \Lambda_n \rangle g_n^{-1}$$
is infinite if and only if there exists a hyperplane in $\mathrm{QM}(\Gamma, \mathcal{G})$ that crosses all the subgraphs $g_1 \langle \Lambda_1 \rangle, \ldots, g_n \langle \Lambda_n \rangle$. 
\end{lemma}


\begin{proof}
We argue by induction on $n$. If $n=1$, then $g_1 \langle \Lambda_1 \rangle g_1^{-1}$ is infinite if and only if $\Lambda_1$ is non-empty, which amounts to saying that the subgraph $g_1 \langle \Lambda_1 \rangle$ is not reduced to a single vertex of $\mathrm{QM}(\Gamma,\mathcal{G})$, or equivalently that $g_1 \langle \Lambda_1 \rangle$ is crossed by at least one hyperplane. Now, assume that $n \geq 2$. It follows from \cite[Theorem~1.7]{Mediangle} and its proof that there exist $g \in g_1 \langle \Lambda_1 \rangle$ and $\Xi \subset \Lambda_1$ such that the projection of $g_2 \langle \Lambda_2 \rangle$ on $g_1 \langle \Lambda_1 \rangle$ agrees with $g \langle \Xi \rangle$ and 
$$g_1 \langle \Lambda_1 \rangle g_1^{-1} \cap g_2 \langle \Lambda_2 \rangle g_2^{-1} = g \langle \Xi \rangle g^{-1}.$$
By induction, we know that
$$g \langle \Xi \rangle g^{-1} \cap g_3 \langle \Lambda_3 \rangle g_3^{-1} \cap \cdots \cap g_n \langle \Lambda_n \rangle g_n^{-1}$$
is infinite if and only if there exists a hyperplane crossing all $g \langle \Xi \rangle, g_3 \langle \Lambda_3 \rangle, \ldots, g_n \langle \Lambda_n \rangle$. But, by Theorem~\ref{thm:ProjQM}, the hyperplanes crossing the projection $g \langle \Xi \rangle$ are exactly the hyperplanes crossing both $g_1 \langle \Lambda_1 \rangle$ and $g_2 \langle \Lambda_2 \rangle$. Consequently, there exists a hyperplane crossing all $g \langle \Xi \rangle, g_3 \langle \Lambda_3 \rangle, \ldots, g_n \langle \Lambda_n \rangle$ if and only if there exists a hyperplanes crossing all $g_1 \langle \Lambda_1 \rangle, \ldots, g_n \langle \Lambda_n \rangle$. This concludes the proof. 
\end{proof}

\noindent
Lemma~\ref{lem:ForSkewer} motivates the following definition. See Figure~\ref{Skew} for an example.
\begin{figure}
\begin{center}
\includegraphics[width=0.6\linewidth]{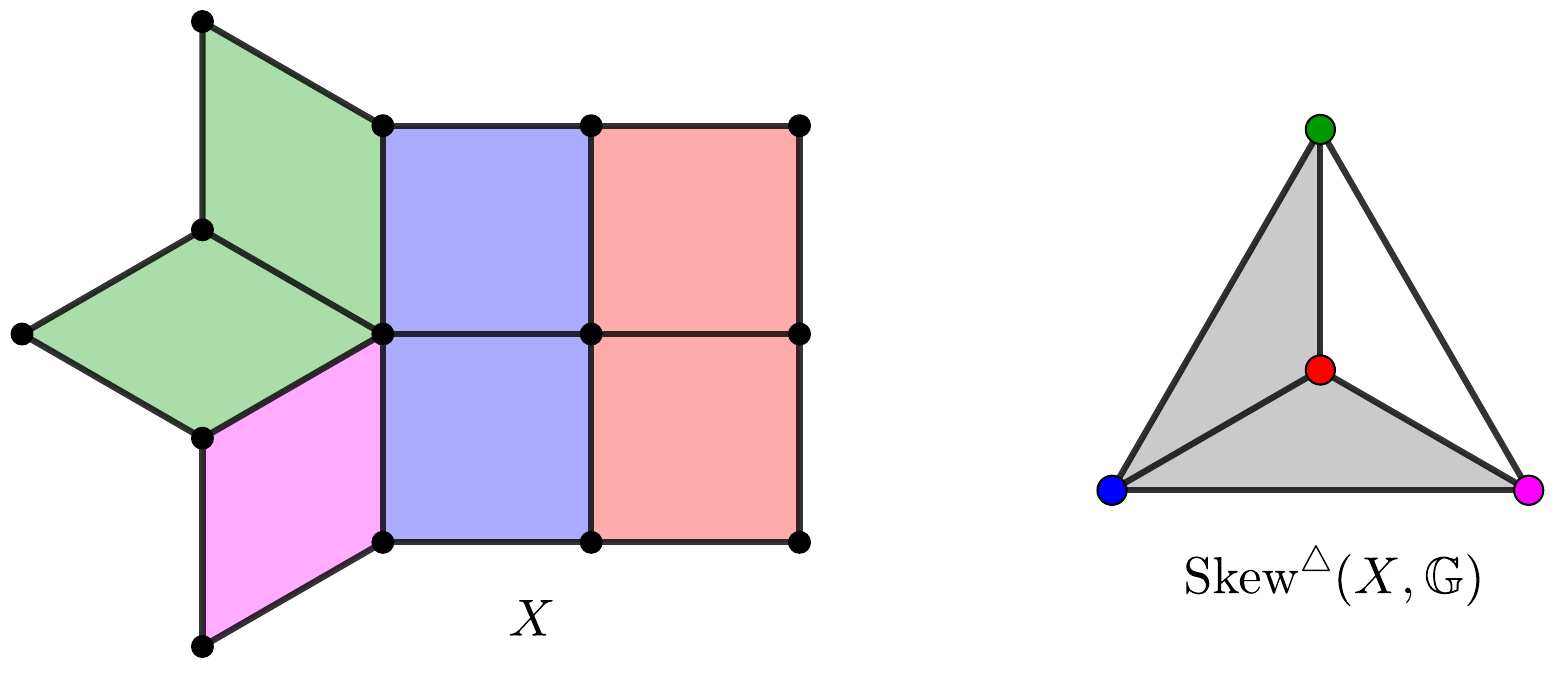}
\caption{A quasi-median graph $X$ endowed with a collection $\mathbb{G}$ of gated subgraphs and the corresponding skewering complex $\mathrm{Skew}^\triangle(X,\mathbb{G})$.}
\label{Skew}
\end{center}
\end{figure}

\begin{definition}
Let $X$ be a quasi-median graph and $\mathbb{G}$ a collection of gated subgraphs. The \emph{skewering complex} $\mathrm{Skew}^\triangle(X,\mathbb{G})$ is the simplicial complex whose vertex-set is $\mathbb{G}$ and whose simplices are given collections of subgraphs crossed by a common hyperplane.
\end{definition}

\noindent
Using this vocabulary, Lemma~\ref{lem:ForSkewer} then implies that coset intersection complexes with respect to parabolic subgroups can be described as skewering complexes of the corresponding quasi-median graphs. The connection with relative contact complexes is described by the following statement:

\begin{prop}\label{prop:SkewCont}
Let $X$ be a quasi-median graph and $\mathbb{G}$ a collection of gated subgraphs. The skewering complex $\mathrm{Skew}^\triangle(X,\mathbb{G})$ and the $\mathbb{G}$-contact complex $\mathrm{Cont}^\triangle(X, \mathbb{G})$ are homotopy equivalent.
\end{prop}


\noindent
In order to prove Proposition~\ref{prop:SkewCont}, the following observation will be needed:

\begin{lemma}\label{lem:SectionHyp}
Let $X$ be a quasi-median graph and $Y_1, \ldots, Y_n$ gated subgraphs. There exists a gated subgraph $Y$ such that a hyperplane of $X$ crosses $Y$ if and only if it crosses all $Y_1, \ldots, Y_n$. 
\end{lemma}

\begin{proof}
Let $Y_1'$ denote the projection of $Y_n$ on $Y_1$. According to Theorem~\ref{thm:ProjQM}, a hyperplane crosses $Y_1'$ if and only if it crosses both $Y_1$ and $Y_n$. Consequently, a hyperplane crosses all $Y_1, \ldots, Y_n$ if and only if it crosses all $Y_1', Y_2, \ldots, Y_ {n-1}$. Similarly, if $Y_1''$ denotes the projection of $Y_{n-1}$ on $Y_1'$, then a hyperplane crosses all $Y_1'', Y_2, \ldots, Y_{n-2}$ if and only if it crosses all $Y_1', Y_2, \ldots, Y_{n-1}$. After $n-1$ iterations of this argument, we find a gated subgraph $Y$ such that a hyperplane crosses $Y$ if and only if it crosses all $Y_1, \ldots, Y_n$. 
\end{proof}

\begin{proof}[Proof of Proposition~\ref{prop:SkewCont}.]
For every $Y \in \mathbb{G}$, let $\mathcal{J}_Y$ denote the subcomplex of the contact complex $\mathrm{Cont}^\triangle(X, \mathbb{G})$ induced by the set of the hyperplanes of $X$ that cross $Y$. Set $\mathcal{J}:= \{ \mathcal{J}_Y \mid Y \in \mathbb{G} \}$. First, notice that $\mathcal{J}$ covers $\mathrm{Cont}^\triangle(X, \mathbb{G})$. Indeed, a simplex in $\mathrm{Cont}^\triangle(X,\mathbb{G})$ corresponds to a collection of hyperplanes  pairwise in contact in $\mathcal{J}_Y$ for some $Y \in \mathbb{G}$. The map $Y \mapsto \mathcal{J}_Y$ is a homotopy equivalence from $\mathrm{Skew}^\triangle(X,\mathbb{G})$ to the nerve complex of $\mathcal{J}$ since it is surjective on $0$-cells and the preimage of any simplex is contractible. Indeed, $Y_1, \ldots, Y_n \in \mathbb{G}$ span a simplex in $\mathrm{Skew}^\triangle(X,\mathbb{G})$ if and only if there exists a hyperplane crossing all of them, which amounts to saying that $\mathcal{J}_{Y_1} \cap \cdots \cap \mathcal{J}_{Y_n} \neq \emptyset$, or that $\mathcal{J}_{Y_1}, \ldots, \mathcal{J}_{Y_n}$ span a simplex in the nerve complex of $\mathcal{J}$.

\medskip \noindent
Finally, given $Y_1, \ldots, Y_n \in \mathbb{G}$, let us justify that $\mathcal{J}_{Y_1} \cap \cdots \cap \mathcal{J}_{Y_n}$ is either empty or contractible. If the intersection is empty, there is nothing to prove; so we assume that it is non-empty. Clearly, the intersection coincides with the contact complex of the gated subgraph given by Lemma~\ref{lem:SectionHyp}. The desired contractibility follows from Proposition~\ref{prop:ContactContractible}. The lemma then follows from the nerve theorem applied to the covering $\mathcal{J}$ of $\mathrm{Cont}^\triangle(X,\mathbb{G})$. 
\end{proof}

\begin{proof}[Proof of Theorem~\ref{thm:CosetInterComplexRelativeCont}.]
For convenience, set $\mathcal{P}:= \{ \langle \Lambda \rangle, \Lambda \in \mathscr{S} \}$. It follows from Lemma~\ref{lem:ForSkewer} that the coset intersection complex $\mathcal{K}(\Gamma \mathcal{G}, \mathcal{P})$ is isomorphic to the skewering complex $\mathrm{Skew}^\triangle(\mathrm{QM}(\Gamma, \mathcal{G}), \mathbb{G})$, where $\mathbb{G}:= \{ g \langle \Lambda \rangle \mid g \in \Gamma \mathcal{G}, \Lambda \in \mathscr{S} \}$. 
%
%
%
Proposition~\ref{prop:SkewCont} shows that $\mathrm{Skew}^\triangle(\mathrm{QM}(\Gamma, \mathcal{G}), \mathbb{G})$ is homotopy equivalent to the $\mathbb{G}$-contact complex $\mathrm{Cont}^\triangle(\mathrm{QM}(\Gamma, \mathcal{G}), \mathbb{G})$. 

\medskip \noindent
To verify that $\mathbb{G}$ is star-covering, let $C$ be a clique of $\mathrm{QM}(\Gamma, \mathcal{G})$. Up to translating by an element of $\Gamma \mathcal{G}$, we know from Lemma~\ref{lem:CliquesPrismsGP} that we can write $C= \langle u \rangle$ for some vertex $u \in \Gamma$. Also by Lemma~\ref{lem:CliquesPrismsGP}, the prisms containing $C$ are exactly the $\langle \Lambda \rangle$ with $\Lambda \subset \Gamma$ a complete subgraph containing $u$. Therefore, the union of all the prisms containing $C$ is contained in $\langle \mathrm{star}(u) \rangle$. But, by assumption, there exists some $\Xi \in \mathscr{S}$ such that $\mathrm{star}(u) \subset \Xi$, so the union must be contained in $\langle \Xi \rangle \in \mathbb{G}$.

\medskip \noindent
Thus, Theorem~\ref{thm:RelContact} applies (also thanks to Lemma~\ref{lem:MQTwoConnected}) and shows that the contact complex $\mathrm{Cont}^\triangle(\mathrm{QM}(\Gamma , \mathcal{G}), \mathbb{G})$ is homotopy equivalent  to $\bigvee_{x \in \mathrm{QM}(\Gamma, \mathcal{G})} \mathrm{sL}_\mathbb{G}(x)$. Recall that $\mathrm{sL}_\mathbb{G}(x)$ is the simplicial complex with vertices corresponding to cliques of $\mathrm{QM}(\Gamma,\mathcal G)$ containing $x$ and simplices corresponding to collections of cliques contained in some common subgraph $\Lambda\in \mathscr{S}$.  Since $\Gamma \mathcal{G}$ acts vertex-transitively on $\mathrm{QM}(\Gamma, \mathcal{G})$, each $\mathrm{sL}_\mathbb{G}(x)$ is isomorphic to $\mathrm{sL}_\mathbb{G}(1)$, so it only remains to verify that $\mathrm{sL}_\mathbb{G}(1)$ is isomorphic to $\Gamma^{\mathscr{S}}$. 

\medskip \noindent
But, according to Lemma~\ref{lem:CliquesPrismsGP}, the cliques containing $1$ are exactly  $\langle u \rangle$ for $u \in \Gamma$. Finitely many such cliques $\langle u_1 \rangle, \ldots, \langle u_n \rangle$ are contained in some subgraph of $\mathbb{G}$ if and only if there exists $\Lambda \in \mathscr{S}$ containing $u_1, \ldots, u_n$. We conclude that $\mathrm{sL}_\mathbb{G}(1)$ can indeed be identified with $\Gamma^\mathscr{S}$. 
\end{proof}

\section{Applications to large-scale geometry of right-angled Artin groups}

The problem of distinguishing right-angled Artin groups up to quasi-isometry is still open in full generality.  Here, we use the characterisation of the homotopy types of coset intersection complexes obtained in Theorem~\ref{thm:CosetInterComplexRelativeCont} to construct new quasi-isometry invariants.  As illustrated in the introduction, these invariants can be used to distinguish the quasi-isometry types of right-angled Artin groups that could not previously be distinguished. We also consider, more generally, graph products of infinite groups.  In this context, we use coset intersection complexes to construct new commensurability invariants of such groups.

\subsection{QI-characteristic parabolic subgroups}

 \noindent In this section, we show that a quasi-isometry between right-angled Artin groups induces an isomorphism of their respective coset intersection complexes with respect to maximal join subgroups.  First, we show that maximal join subgroups are coarsely preserved under quasi-isometry. 

\medskip\noindent We thank Jingyin Huang for pointing out how the following statement follows from results in \cite{HuangI}.

\begin{thm}\label{thm:RAAGmaxJoin}
Let $\Gamma_1,\Gamma_2$ be two finite   graphs. For every quasi-isometry $\eta \colon A(\Gamma_1) \to A(\Gamma_2)$, there exists a constant $C \geq 0$ such that the following holds. For every element $g \in A(\Gamma_1)$ and every maximal join $\Phi \subset \Gamma_1$, there exist an element $h \in A(\Gamma_2)$ and a maximal join $\Psi \subset \Gamma_2$ such that $\eta (g \langle \Phi \rangle)$ lies at Hausdorff distance $\leq C$ from $h \langle \Psi \rangle$.
\end{thm}

\noindent
 Here, a maximal join refers to a subgraph that is maximal with respect to  inclusion among all joins.  Before turning to Theorem~\ref{thm:RAAGmaxJoin}, we prove the following observation:

\begin{lemma}\label{lem:FiniteHaussRAAG}
Let $\Gamma$ be a finite graph, $\Phi,\Psi \subset \Gamma$ two subgraphs, and $g \in A(\Gamma)$ an element. The following statements are equivalent:
\begin{itemize}
    \item[(i)] in $A(\Gamma)$, $\langle \Phi \rangle$ is contained in a neighbourhood of $g \langle \Psi \rangle$;
    \item[(ii)] $\langle \Phi \rangle \subset g \langle \Psi \rangle$;
    \item[(iii)] $\Phi \subset \Psi$ and $g \in \langle \mathrm{star}(\Phi) \rangle \langle \Psi \rangle$.
\end{itemize}
\end{lemma}

\begin{proof}
The implications $(iii) \Rightarrow (ii)$ and $(ii) \Rightarrow (i)$ are clear, so it suffices to prove that $(i) \Rightarrow (iii)$. 

\medskip \noindent
If $\Phi$ is not contained in $\Psi$, fix a vertex $u$ that belongs to $\Phi$ but not to $\Psi$. If $\pi_u \colon A(\Gamma) \to \langle u \rangle$ denotes the canonical retraction that sends all the generators but $u$ to $1$, then $\pi_u$ sends $\langle \Phi \rangle$ to the entire line $\langle u \rangle$ but it sends $g \langle \Psi \rangle$ to the single point $\{ \pi_u(g) \}$. Since $\pi_u$ is $1$-Lipschitz, this contradicts the fact that $\langle \Phi \rangle$ is contained in a neighbourhood of $g \langle \Psi \rangle$. Thus, we have proved that $\Phi \subset \Psi$. It remains to verify that $g \in \langle \mathrm{star}(\Phi) \rangle \langle \Psi \rangle$. 

\medskip \noindent
If $\langle \Phi \rangle$ and $g \langle \Psi \rangle$ intersect, then the inclusion $\langle \Phi \rangle \subset g \langle \Psi \rangle$ must hold since $\Phi \subset \Psi$. If so, then $g \langle \Psi \rangle$ contains $1$, hence $g \in \langle \Psi \rangle$. From now on, we assume that $\langle \Phi \rangle$ and $g \langle \Psi \rangle$ are disjoint. Also, we think of $A(\Gamma)$ as the median graph $\mathrm{Cayl}(A(\Gamma),\Gamma)$. In the rest of the proof, we assume some familiarity with the geometry of this graph, which is similar to those of the quasi-median graphs associated to graph products. 

\medskip \noindent
Notice that every hyperplane crossing $\langle \Phi \rangle$ has to cross $g \langle \Psi \rangle$. Otherwise, let $J$ be a hyperplane of $\langle \Phi \rangle$ that does not cross $g \langle \Psi \rangle$. If $u \in \Phi$ denotes the generator labelling the edges of $J$ and if $h \in N(J) \cap \langle \Phi \rangle$, then the $\langle huh^{-1} \rangle$-translates of $J$ are pairwise non-transverse and they all cross $\langle \Phi \rangle$. Then, for every $n \geq 1$, either $hu^n$ or $hu^{-n}$ is separated from $g \langle \Psi \rangle$ by at least $n-1$ hyperplanes, namely $J, hu^{\pm 1} J, \ldots, hu^{\pm (n-1)} J$. Since such a vertex belongs to $\langle \Phi \rangle$, this contradicts the assumption that $\langle \Phi \rangle$ is contained in a neighbourhood of $g \langle \Psi \rangle$, proving our claim.

\medskip \noindent
Fix two vertices $a \in \langle \Phi \rangle$ and $b \in g \langle \Psi \rangle$ at minimal distance. Consider a path connecting $1$ to $g$ that decomposes as the concatenation of a path 
$\alpha_1$ connecting $1$ to $a$ in $\langle \Phi \rangle$, a geodesic $\alpha_2$ connecting $a$ to $b$, and a path $\alpha_3$ connecting $b$ to $g$ in $g \langle \Psi \rangle$. Clearly, $\alpha_1$ is labelled by a word in $\langle \Phi \rangle$ and $\alpha_3$ is labelled by a word in $\langle \Psi \rangle$. As a consequence of Theorem~\ref{thm:ProjQM}, the hyperplanes crossing   $\alpha_2$  separate $\langle \Phi \rangle$ and $g \langle \Psi \rangle$, and consequently must be transverse to all the hyperplanes of $\langle \Phi \rangle$. Therefore, the hyperplanes crossing  $\alpha_2$  are labelled by vertices in $\mathrm{link}(\Phi )$. Thus, $\alpha_2$ is labelled by a word in $\mathrm{link}(\Phi)$. We conclude that $g$, which is equal in $A(\Gamma)$ to the word labelling $\alpha_1 \alpha_2 \alpha_3$, belongs to
$$\langle \Phi \rangle \langle \mathrm{link}(\Phi) \rangle \langle \Psi \rangle = \langle \mathrm{star}( \Phi) \rangle \langle \Psi \rangle,$$
as desired. 
\end{proof}

\begin{proof}[Proof of Theorem~\ref{thm:RAAGmaxJoin}.]
By \cite[Theorem~2.10]{HuangI}, it suffices to show that every maximal join $\Phi$ of a right-angled Artin group $A(\Gamma)$ is stable in the sense of \cite[Definition~3.9]{HuangI}.  To see that this is indeed sufficient, note that the definition of stability implies that, for any $g\in G$, if $\eta\colon A(\Gamma) \to A(\Gamma')$ is a quasi-isometry, there exists some $\Phi'\subseteq \Gamma'$ and $h\in A(\Gamma')$ such that $\eta(g\langle \Phi\rangle)$ and $h\langle \Phi'\rangle$ are at uniformly finite Hausdorff distance.  By \cite[Theorem~2.9]{HuangI}, the subgraph $\Phi'$ is a join.  Suppose $\Omega\supseteq \Phi'$ is a maximal join, and let $\eta^{-1}$ be any quasi-inverse of $\eta$.  Again by stability, we see that $\eta^{-1}(h\langle\Omega\rangle)$ is at finite Hausdorff distance from a join $k \langle\Phi''\rangle$ in $A(\Gamma)$. So $\langle \Phi \rangle$ is contained in a neighbourhood of $k \langle \Phi'' \rangle$. It follows from Lemma~\ref{lem:FiniteHaussRAAG} that $\langle \Phi \rangle \subset k \langle \Phi'' \rangle$, and in particular that $\Phi \subset \Phi''$. Hence $\Phi = \Phi''$ by maximality of $\Phi$. Then, we also deduce from Lemma~\ref{lem:FiniteHaussRAAG} that $\langle \Phi \rangle = k \langle \Phi'' \rangle$. Necessarily, the Hausdorff distance between $h \langle \Phi' \rangle$ and $h \langle \Omega \rangle$ must be finite. We deduce from Lemma~\ref{lem:FiniteHaussRAAG} that $\Omega= \Phi'$. A fortiori, $\Phi'$ is a maximal join, and the result follows.

\medskip\noindent 
Thus our goal is to show that a maximal join $\Phi\subseteq \Gamma$ is stable.  A key fact that we will use repeatedly is that, for any vertex $v\in \Gamma$, either $\mathrm{star}(v)$ or $\mathrm{link}(v)$ is stable \cite[Lemma~5.1]{HuangI}.

\medskip\noindent 
Let $\Phi=\Phi_1 * \Phi_2 * \cdots * \Phi_k$ be the minimal join decomposition of $\Phi$, that is, a decomposition such that $\Phi_1$ is a complete graph, and, if $i\neq 1$, then $\Phi_i$ is not a complete graph and does not split as a join.  We consider two cases, depending on whether $\Phi_1=\emptyset$ or not. \\

\noindent{\bf Case 1: $\Phi_1\neq \emptyset$.}  Fix $v \in \Phi_1$.  Then $\Phi\subseteq \mathrm{star}(v)$, and, since $\mathrm{star}(v)$ is a join, the maximality of $\Phi$ implies that $\Phi=\mathrm{star}(v)$.  If $\mathrm{star}(v)$ is stable, then we are done, so suppose that $\mathrm{link}(v)=\Phi-\{v\}$ is stable.  Let 
\[
V=\{w\in \Gamma \mid w \textrm{ is adjacent (but not equal) to every vertex in }\mathrm{link}(v)\}
\]
be the full orthogonal complement of $\mathrm{link}(v)$.  Then \cite[Lemma~3.12]{HuangI} implies that $V * \mathrm{link}(v)$ is stable.  
The vertex $v$ is in  $V$, and so $\Phi\subseteq V * \mathrm{link}(v)$.  Maximality of $\Phi$ again implies that $\Phi=V * \mathrm{link} (v)$, and we conclude that $\Phi$ is stable. \\

\noindent{\bf Case 2: $\Phi_1=\emptyset$.}  For each $i$, let $W_i=\Phi-\Phi_i$, and, for each $w\in W_i$, let $\Omega_w$ be either $\mathrm{star}(w)$ or $\mathrm{link}(w)$, whichever is stable.  Finally, let \[
\Omega_i=\bigcap_{w\in W_i}\Omega_w.
\] 
Since $\Omega_i$ is a finite intersection of stable subgraphs, it is stable by \cite[Lemma~3.10]{HuangI}.  Our first goal is to show that $\Phi_i=\Omega_i$.  

\medskip\noindent The first step is to show that $\Omega_i\cap W_i=\emptyset$. Let $v\in W_i$, so that $v\in \Phi_j$ for some $j\neq i$.  There exists a vertex $w\in \Phi_j-\{v\}$ that is not adjacent to $v$, as otherwise $\Phi_j$ would split as a join $\{v\} * (\Phi_j - \{v\})$, contradicting that $j\neq 1$.  Since $w$ is not adjacent to $v$, it follows that $v\not\in \Omega_w$.  Since $w\in W_i$, this implies that $v\not\in \Omega_i$, and so $\Omega_i\cap W_i=\emptyset$, as desired.

\medskip\noindent  Since $\Omega_i\cap W_i=\emptyset$, it follows from the definition of $\Omega_i$ that every vertex of $\Omega_i$ is in the link of every vertex in $W_i$.  Thus $\Omega_i$ and $W_i$ form a join $W_i * \Omega_i$. Notice that
$$\Phi_i \subset \bigcap\limits_{w \in W_i} \mathrm{link}(w) \subset \bigcap\limits_{w \in W_i} \Omega_w = \Omega_i.$$
Therefore $\Phi\subseteq W_i * \Omega_i$, and the maximality of $\Phi$ implies that $\Phi = W_i * \Omega_i$.  The definition of $W_i$ thus implies that $\Phi_i = \Omega_i$.

\medskip\noindent We have shown that each $\Phi_i$ is stable.  Since $\Phi$ is a maximal join, it is also stable by \cite[Lemma~3.18]{HuangI}.
\end{proof}

\noindent We are now ready to show  the invariance of coset intersection complexes.
\begin{cor}\label{cor:CosetInterComplexJoinQI}
Let $\Gamma_1,\Gamma_2$ be two finite graphs. Every quasi-isometry $\eta \colon A(\Gamma_1) \to A(\Gamma_2)$ induces an isomorphism between the coset intersection complexes
$$\mathcal{K} \left( A(\Gamma_1), \{ \langle \Phi \rangle, \Phi \subset \Gamma_1 \text{ maximal join} \} \right) \to \mathcal{K} \left( A(\Gamma_2), \{ \langle \Psi \rangle, \Psi \subset \Gamma_2 \text{ maximal join} \} \right).$$
\end{cor}

\begin{proof}
By Theorem~\ref{thm:RAAGmaxJoin}, the quasi-isometry $\eta$ is a quasi-isometry of pairs 
\[
\eta\colon \left( A(\Gamma_1), \{ \langle \Phi \rangle, \Phi \subset \Gamma_1 \text{ maximal join} \} \right) \to  \left( A(\Gamma_2), \{ \langle \Psi \rangle, \Psi \subset \Gamma_2 \text{ maximal join} \} \right) 
\]
in the sense of~\cite[Definition 2.4]{CIC}. 
Let $\dot \eta\colon \mathcal{K} \left( A(\Gamma_1), \{ \langle \Phi \rangle, \Phi \subset \Gamma_1 \text{ maximal join} \} \right) \to \mathcal{K} \left( A(\Gamma_2), \{ \langle \Psi \rangle, \Psi \subset \Gamma_2 \text{ maximal join} \} \right)$ be an induced map on  coset intersection complexes.  By \cite[Proposition~4.9]{CIC}, the map $\dot \eta$ is a simplicial map, and we will show that it is an isomorphism.  

\medskip\noindent 
First, note that, for any finite graph $\Gamma$ and any subgraph $\Lambda\subset \Gamma$, the commensurator $\mathrm{Comm}(\langle \Lambda \rangle)$ of $\langle \Lambda \rangle$ is equal to its normalizer $\mathrm{N}(\langle \Lambda \rangle)$  in $A(\Gamma)$ \cite[Theorem~0.1]{Godelle} (see also Lemma~\ref{lem:SelfCommensurated} below).  Moreover, $\mathrm{N}(\langle \Lambda \rangle)=\langle \Lambda \rangle \times \langle \mathrm{link}(\Lambda) \rangle$ \cite[Proposition~3.13]{MR3365774}.  If $\Lambda$ is a maximal join, then $\mathrm{link}(\Lambda)=\emptyset$, and so $N(\langle \Lambda \rangle)=\langle \Lambda \rangle$.  In the terminology of \cite{CIC}, this is equivalent to saying that the group pairs 
$$\left( A(\Gamma_1), \{ \langle \Phi \rangle, \Phi \subset \Gamma_1 \text{ maximal join} \} \right) \text{ and } \left( A(\Gamma_2), \{ \langle \Psi \rangle, \Psi \subset \Gamma_2 \text{ maximal join} \} \right)$$ 
are both reduced.  Therefore, by \cite[Proposition~4.9]{CIC}, the map $\dot \eta$ is an isomorphism, as desired.
\end{proof}

\noindent
A finite collection of subgroups $\mathcal{P}$ of a finitely generated group $G$ is called \emph{qi-characteristic} if each subgroup $P$ has finite index in $\mathsf{Comm}_G(P)$ and the collection of left cosets $\{gP\mid g\in G,\quad P\in\mathcal{P}\}$ is preserved up to uniform Hausdorff distance by any self quasi-isometry of $G$, see~\cite{MR4520684}. 
Since any maximal join subgroup of a $A(\Gamma)$ is self commensurated~\cite[Theorem~0.1]{Godelle},  Theorem~\ref{thm:RAAGmaxJoin} implies that the collection of maximal join subgroups  $\{ \langle \Phi \rangle, \Phi \subset \Gamma \text{ maximal join} \}$  of $A(\Gamma)$ is qi-characteristic. This property has the following consequence.

\begin{cor}
 Let $\Gamma$ be a finite  graph. If $G$ is a finitely generated group quasi-isometric to $A(\Gamma)$, then $G$ admits a cellular action on  
$\mathcal{K}\left( A(\Gamma), \{ \langle \Phi \rangle, \Phi \subset \Gamma \text{ maximal join} \} \right)$ with finitely many $G$-orbits of $0$-cells. 
\end{cor}

\begin{proof}
Since $\mathcal{P}=\{ \langle \Phi \rangle, \Phi \subset \Gamma \text{ maximal join} \}$ of $A(\Gamma)$ is a qi-characteristic collection of subgroups of $A(\Gamma)$, a consequence of~\cite[Theorem 1.1]{MR4520684} is  that there is a finite collection of subgroups $\mathcal{Q}$ of $G$ such that the group pairs $(A(\Gamma), \mathcal{P})$ and   $(G,\mathcal{Q})$ are quasi-isometric. By~\cite[Theorem~1.5 \& Proposition~2.9]{CIC}, we can assume that $Q=\mathsf{Comm}_G(Q)$ for each $Q\in \mathcal{Q}$. It follows that the group pairs $(A(\Gamma),\mathcal{P})$ and $(G,\mathcal{Q})$ are quasi-isometric and reduced. By~\cite[Proposition~4.9]{CIC}, the coset intersection complexes   $\mathcal{K}(A(\Gamma),\mathcal{P})$ and $\mathcal{K}(G,\mathcal{Q})$ are isomorphic as simplicial complexes. Since $\mathcal{Q}$ is a finite collection, $G$ acts with finitely many orbits of vertices in $\mathcal{K}(G,\mathcal{Q})$.
\end{proof}

\subsection{Quasi-isometric and commensurability invariants}\label{section:QIinvariant}

\noindent 
In this section, we gather quasi-isometry and commensurability invariants that can be read off the defining graph of a graph product.  We begin with the case of right-angled Artin groups.  Recall that, given a simplicial graph $\Gamma$, the space $\Gamma^{\bowtie}$ is the simplicial complex whose vertex-set is $\Gamma$ and whose simplices are given by joins.

\begin{thm}\label{thm:BowtieQIinv}
Let $\Gamma_1,\Gamma_2$ be two finite connected graphs. If $A(\Gamma_1)$ and $A(\Gamma_2)$ are quasi-isometric, then $\bigvee_\mathbb{N} \Gamma_1^{\bowtie}$ and $\bigvee_\mathbb{N} \Gamma_2^{\bowtie}$ are homotopy equivalent.
\end{thm}

\begin{proof}
If $A(\Gamma_1)$ and $A(\Gamma_2)$ are quasi-isometric, then we know from Corollary~\ref{cor:CosetInterComplexJoinQI} that the corresponding coset intersection complexes relative to maximal join subgroups are isomorphic. But, according to Theorem~\ref{thm:CosetInterComplexRelativeCont}, these two complexes are respectively homotopy equivalent to $\bigvee_\mathbb{N} \Gamma_1^{\bowtie}$ and $\bigvee_\mathbb{N} \Gamma_2^{\bowtie}$, concluding the proof.
\end{proof}

\noindent We now turn our attention to more general graph products of infinite groups.  In this setting, we again use the coset intersection complex with respect to maximal join subgroups, and we obtain the following commensurability invariant, rather than the quasi-isometry invariant obtained for right-angled Artin groups.  

\begin{thm}\label{thm:CommensurableGP}
Let $\Gamma_1$ (resp.\ $\Gamma_2$) be a finite connected graph and $\mathcal{G}_1$ (resp.\ $\mathcal{G}_2$) a collection of  infinite groups indexed by $\Gamma_1$ (resp.\ $\Gamma_2$). If $\Gamma_1 \mathcal{G}_1$ and $\Gamma_2 \mathcal{G}_2$ are commensurable, then $\bigvee_\mathbb{N} \Gamma_1^{\bowtie}$ and $\bigvee_\mathbb{N} \Gamma_2^{\bowtie}$ are homotopy equivalent.
\end{thm}

\noindent The proof of the above theorem relies on the relationship between the respective coset intersection complexes with respect to maximal join subgroups; see Lemma~\ref{lem:CommensurableCIC}. The proof of the lemma assumes familiarity with the well-known  \emph{normal form theorem} for graph products, see~\cite{GreenGP}.

\begin{lemma}\label{lem:SelfCommensurated}
 Let $\Gamma$ be a a finite connected graph and $\mathcal{G}$ a collection of infinite groups indexed by $\Gamma$. Then any maximal join subgroup of $\Gamma\mathcal{G}$ is self-commensurated.
\end{lemma}

\begin{proof}
    Suppose first that $P_1$ and $P_2$ are two parabolic subgroups of $\Gamma\mathcal{G}$ such that $P_1$ is a finite-index subgroup of $P_2$.  We will show that $P_1=P_2$.  By conjugating, we may assume that $P_1=\langle \Phi \rangle$ and $P_2=g \langle \Lambda \rangle g^{-1}$ for $\Phi,\Lambda \subseteq \Gamma$ and $g\in \Gamma\mathcal{G}$. It follows from~\cite[Corollary~3.8]{MR3365774} that $\Phi \subseteq \Lambda$.  If $\Phi \subsetneq \Lambda$, then there exists a vertex $u\in \Lambda \setminus \Phi$.  Define a map $\pi_u\colon \Gamma\mathcal{G} \to G_u$ that is the identity map on $G_u$ and sends $G_v$ to $1$ for all $v\neq  u$. Since $u\in \Lambda$, we see that $\pi_u(P_2)
    = G_u$, which is infinite by assumption.  On the other hand, $u\not\in \Phi$ implies that $\pi_u(P_1)=1$, which contradicts that $P_1$ is finite index in $P_2$.  Therefore $\Lambda = \Phi$, and so $P_1=P_2$. 

    \medskip\noindent We now show that, if $P$ is a parabolic subgroup, then $\textrm{Comm}(P)=N(P)$.  It is clear that $N(P)\subseteq \textrm{Comm}(P)$, so suppose $g\in \textrm{Comm}(P)$.  Then $P'= gPg^{-1} \cap P$ is a finite-index subgroup of $P$.  Since the intersection of parabolic subgroups is parabolic \cite[Proposition~3.4]{MR3365774}, the subgroup $P'$ is parabolic.  By the argument above, it follows that $P'=P$, and so $g\in N(P)$.

    \medskip\noindent  Finally, suppose $P$ is a maximal join subgroup.  By conjugating, we may assume $P=\langle \Phi \rangle$.  We have $N(P)=\langle \Phi, \textrm{lk}(\Phi) \rangle$ \cite[Proposition~3.13]{MR3365774}, but $\textrm{lk}(\Phi)=\emptyset$, since $\Phi$ is a maximal join.  Therefore, we have $\textrm{Comm}(P)=N(P)=\langle \Phi \rangle =P$, and $P$ is self-commensurated, as desired.
\end{proof}

\begin{lemma}\label{lem:vcyclic_centralizer}
    Let $\Gamma$ be a finite graph and $G=\Gamma\mathcal G$ be the graph product of a family of non-trivial groups $\mathcal G$ indexed by $\Gamma$. If $H \leq G$ is a subgroup such that no conjugate of $H$ is contained in a vertex group or a join subgroup, then there exists $h \in H$ such that the centraliser of $h^k$ is virtually cyclic for every $k \neq 0$.
\end{lemma}

\begin{proof}
    Let $G_0$ denote the parabolic closure of $H$ in $G$.  The assumption on $H$ implies that  $G_0$ is neither conjugate to a vertex group nor a join subgroup. Applying~\cite[Corollary~6.20]{MR3368093}  to $G_0$ yields an isometric action of $G_0$ on a tree and an element $h\in H$ that is a loxodromic WPD element with respect to this action.  By \cite[Corollary~6.9]{Osin}, the centraliser $C_{G_0}(h^k)$ of $h^k$ in $G_0$ is virtually cyclic for all $k\neq 0$.  

    \medskip\noindent The centraliser $C_G(h^k)$ of $h^k$ in $G$ is the product of $C_{G_0}(h^k)$ and $\langle \textrm{lk}(G_0)\rangle$.  If $\textrm{lk}(G_0)$ is non-empty, then $H$ is contained in a join subgroup, which contradicts our assumption.  Thus $C_G(h^k)=C_{G_0}(h^k)$ is virtually cyclic.
\end{proof}

\begin{lemma}\label{lem:CommensurableCIC}
Let $\Gamma_1$ (resp.\ $\Gamma_2$) be a finite connected graph with at least two vertices and $\mathcal{G}_1$ (resp.\ $\mathcal{G}_2$) a collection of  infinite  groups indexed by $\Gamma_1$ (resp.\ $\Gamma_2$). If $\Gamma_1 \mathcal{G}_1$ and $\Gamma_2 \mathcal{G}_2$ are commensurable, then the coset intersection complexes
$$\mathcal{K} \left( \Gamma_1 \mathcal{G}_1, \{ \langle \Phi \rangle, \Phi \subset \Gamma_1 \text{ maximal join}\} \right) \quad \text{and} \quad\mathcal{K} \left( \Gamma_2 \mathcal{G}_2, \{ \langle \Psi \rangle, \Psi \subset \Gamma_2 \text{ maximal join}\} \right)$$
are isomorphic. 
\end{lemma}

\begin{proof}
Let $A_i=\Gamma_i\mathcal{G}_i$ for $i=1,2$, let $G=A_1 \cap A_2$,   
let $\mathcal{J}_i=\{ \langle \Phi \rangle \mid \Phi \subset \Gamma_i \text{ maximal join}\}$  and let $T_i$ be a right transversal of the subgroup $G$ of $\Gamma_i\mathcal{G}_i$. Observe that $\{tJ\colon t \in T_i,\ J\in \mathcal{J}_i\}$ is a collection of representatives of $G$-orbits of the $G$-set of left cosets $A_i/\mathcal{J}_i=\{gJ\mid g\in  A_i,\ J\in\mathcal{J}_i\}$. 

\medskip\noindent 
Let us argue that, for each left coset $t\langle\Delta\rangle$, where $t\in T_1$ and $\Delta$ is a maximal join of $\Gamma_1$, there exist $g\in A_2$ and a maximal join $\Lambda$ of $\Gamma_2$ such that 
\begin{equation}\label{eq:MaximalJoinsMap}
   t\langle \Delta \rangle t^{-1} \cap G \leq g \langle \Lambda \rangle g^{-1}. 
\end{equation}
To simplify notation, let $H$ denote the subgroup $   t\langle \Delta \rangle t^{-1}$ of $A_1$. Since $\Gamma_2$ is a finite graph, it suffices to show that $H\cap G$ is contained in a join subgroup of $A_2$.  There are three cases to consider. First, if $\Gamma_2$ is a join, then trivially $H\cap G$ is a subgroup of a join subgroup of $A_2$. Second, if  $H\cap G$ is conjugate to a subgroup of a vertex of group of $A_2$, then $H\cap G$ is a subgroup of a join subgroup of $A_2$ since $\Gamma_2$ is a connected graph with at least two vertices. For the third case, suppose
$H\cap G$ is not a subgroup of a conjugate of a vertex group of $A_2$ and that $\Gamma_2$ is not a join. Suppose, by contradiction, that $H\cap G$ is not contained in a join subgroup of $A_2$. Applying Lemma~\ref{lem:vcyclic_centralizer} to the subgroup $H\cap G$, 
we obtain  $h\in  H\cap G$ such that the centraliser $C_{A_2}(h^k)$ is virtually infinite cyclic for every $k\neq 0$. In particular, $C_{H\cap G}(h^k)$ is virtually  infinite cyclic,  as well.  
Since $\Delta$ is a join of $\Gamma_1$, it follows that the group $H$ is a product $H_1\times H_2$ of infinite subgroups. Let us assume that $h=(h_1,h_2)$.  Up to replacing $h$ with some power $h^k$, we can assume that $h_1$ and $h_2$ are each either infinite-order or trivial. Since $H\cap G$ is a finite-index subgroup of $H$, we have that $C_{H}(h)=C_{H_1}(h_1)\times C_{H_2}(h_2)$ is a virtually infinite cyclic subgroup. But  $C_{H_1}(h_1)$ and $C_{H_2}(h_2)$ are both infinite groups, and so  $C_H(h)$ is not virtually cyclic, which contradicts that $C_{H\cap G}(h)$ is virtually cyclic. Therefore $H\cap G$ must be contained in a join subgroup of $A_2$. This completes the proof of the  claim at the beginning of the paragraph.  

\medskip\noindent It follows that there is a $G$-equivariant function   $f\colon A_1/J_1 \to A_2/J_2$ such that if $t\in T_1$ and $\Delta$ is a maximal join of $\Gamma_1$, then     $f(t\langle \Delta \rangle)=g\langle \Lambda \rangle$ where $\Lambda$ is a maximal join of $\Gamma_2$,  $g\in A_2$,  and~\eqref{eq:MaximalJoinsMap} holds. By symmetry, there is a $G$-equivariant map  $f' \colon A_2/J_2 \to A_1/J_1$ such that, for any $t\langle \Delta \rangle \in A_2/\mathcal{J}_2$, if $f'(t\langle \Delta \rangle)= g\langle \Lambda \rangle$ then~\eqref{eq:MaximalJoinsMap} holds as well.  To conclude the proof, we prove that $f$ is a simplicial isomorphism  
\[ f\colon \mathcal{K} \left( \Gamma_1 \mathcal{G}_1, \{ \langle \Phi \rangle, \Phi \subset \Gamma_1 \text{ maximal join}\} \right) \to  \mathcal{K} \left( \Gamma_2 \mathcal{G}_2, \{ \langle \Psi \rangle, \Psi \subset \Gamma_2 \text{ maximal join}\} \right).\]
Let us first argue that $f$ is a simplicial map. Let $\{a_1J_1,\ldots , a_nJ_n\} \subset A_1/\mathcal{J}_1$ be a simplex of $\mathcal{K} \left(A_1, \mathcal{J}_1\right)$, and suppose that $f(a_iJ_i)=b_iK_i$.  Then $\bigcap_{i=1}^n a_iJ_ia_i^{-1}$ is an infinite subgroup of $A_1$. Since $G$ is finite index in $A_1$, the intersection $G\cap \bigcap_{i=1}^n a_iJ_ia_i^{-1}$ is infinite as well. By  definition of $f$, we have that $\bigcap_{i=1}^n a_iJ_ia_i^{-1} \cap G \leq \bigcap_{i=1}^n b_iK_ib_i^{-1}$ and hence $\{b_1K_1, \ldots , b_nK_n\}$ is a simplex of $\mathcal{K}\left( A_2, \mathcal{J}_2\right)$. Hence $f$ is a simplicial map and, by an analogous argument, $f'$ is as well. To conclude that $f$ is an isomorphism, we show that $f'$ is its inverse. Let $aJ\in A_1/\mathcal{J}_1$ and suppose that $f(aJ)=bK$ and $f'(bK)=cL$. By the definitions of $f$ and $g$, we have that
\[ aJa^{-1}\cap G \leq bKb^{-1}\cap G \leq cLc^{-1}.\]
Since $J$ and $L$ are maximal join subgroups of $A_1$, it follows that $aJa^{-1}=cLc^{-1}$. Indeed, it follows from \cite[Corollary~3.8]{MR3365774} that the support of $J$ is contained in that of $L$, and hence by maximality $J=L$. The equality $aJa^{-1}=cLc^{-1}$ follows by the fact that maximal join subgroups are self-commensurated; see Lemma~\ref{lem:SelfCommensurated}.  Therefore $f'\circ f$ is the identity map on $A_1/\mathcal{J}_1$, and similarly $f\circ f'$ is the identity map on $A_2/\mathcal{J}_2$. It follows that $f$ and $g$ are simplicial isomorphisms. 
\end{proof}

\noindent 
We are now ready to prove Theorem~\ref{thm:CommensurableGP}.
\begin{proof}[Proof of Theorem~\ref{thm:CommensurableGP}.]
If our graph products $\Gamma_1 \mathcal{G}_1$ and $\Gamma_2 \mathcal{G}_2$ are commensurable, then we know from Lemma~\ref{lem:CommensurableCIC} that the corresponding coset intersection complexes relative to maximal join subgroups are isomorphic. But, according to Theorem~\ref{thm:CosetInterComplexRelativeCont}, these two complexes are respectively homotopy equivalent to $\bigvee_\mathbb{N} \Gamma_1^{\bowtie}$ and $\bigvee_\mathbb{N} \Gamma_2^{\bowtie}$. This concludes the proof of our theorem.
\end{proof}

\noindent 
We end with one final quasi-isometry invariant of right-angled Artin groups.

\begin{thm}\label{thm:RAAGCrossingQI}
Let $\Gamma_1,\Gamma_2$ be two finite connected graphs that do not contain two non-adjacent vertices $u,v$ satisfying $\mathrm{link}(u) \subset \mathrm{star}(v)$. If $A(\Gamma_1)$ and $A(\Gamma_2)$ are quasi-isometric, then $\bigvee_\mathbb{N} \Gamma_1^{\triangle}$ and $\bigvee_\mathbb{N} \Gamma_2^{\triangle}$ are homotopically equivalent.
\end{thm}

\begin{proof}
Under our assumptions, \cite[Theorem~2.20]{HuangII} shows that the crossing complexes of $\mathrm{QM}(\Gamma_1,\mathcal{G}_1)$ and $\mathrm{QM}(\Gamma_2, \mathcal{G}_2)$ are isomorphic, and a fortiori homotopy equivalent. Thus, Theorem~\ref{thm:IntroCosetInterComplex} yields the desired conclusion. 
\end{proof}

\subsection{Many bowtie complexes}\label{section:ManyBowtie}

\noindent
It is natural to ask how many different values the quasi-isometric invariant provided by Theorem~\ref{thm:BowtieQIinv} can take. In other words, are there many simplicial complexes in 
$$\{\Gamma^{\bowtie} \mid \Gamma \text{ finite graph} \}$$ 
that are not homotopy equivalent? In this section, we propose an elementary construction that shows that, up to homotopy equivalence, every finite simplicial complex belongs to this family. This is the content of Corollary~\ref{cor:ManyBowtie}. As a consequence, it follows from Theorem~\ref{thm:BowtieQIinv} that, in some sense, there are ``as many'' right-angled Artin groups as there are one-ended finitely presented groups. See Corollary~\ref{cor:ManyRAAG} for a precise statement. 

\medskip \noindent
Our proof of Corollary~\ref{cor:ManyBowtie} is based on the following theorem. Recall that, given a poset $(P,\leq)$, the \emph{comparability graph} $\kappa(P,\leq)$ is the graph whose vertex-set is $P$ and whose edges connect two elements whenever they are $\leq$-comparable.

\begin{thm}\label{thm:LatticeHomotopy}
Let $(P, \leq)$ be a finite join-semilattice. The complexes $\kappa(P, \leq)^{\bowtie}$ and $\kappa(P,\leq)^\triangle$ are homotopy equivalent. 
\end{thm}

\noindent
We start by recording a couple of preliminary lemmas.

\begin{lemma}\label{lem:JoinLattice}
Let $(P,\leq)$ be a finite join-semilattice. In $\kappa(P,\leq)$, joins are contained in stars of vertices.
\end{lemma}

\begin{proof}
Let $A \ast B$ be a join in $\kappa(P, \leq)$. 

\medskip \noindent
As a first case, assume that $A$ is bounded below, which amounts to saying that the meet $\wedge A$ exists. Set $B^-:= \{ b \in B \mid b \leq \wedge A\}$. If $B^-$ is non-empty, define $x:= \vee B^-$; otherwise, set $x:= \wedge A$. We claim that $A \ast B$ is contained in the star of $x$.

\medskip \noindent
It is clear that, for every $a \in A$, $x \leq \wedge A \leq a$; hence $x$ and $a$ adjacent in $\kappa(P , \leq)$. Given a $b \in B$, either $b \leq a$ for every $a \in A$, in which case $b \leq \wedge A$, hence $b \in B^-$ and consequently $b \leq x$; or $b \geq a$ for some $a \in A$, in which case $x \leq a \leq b$. In either case, $x$ and $b$ are adjacent in $\kappa(P,\leq)$. Thus, we have proved that $x$ is adjacent or equal to all the points in $A \cup B$. In other words, our join $A \ast B$ is contained in the star of $x$ in $\kappa(P,\leq)$, as desired.

\medskip \noindent
Symmetrically, if $B$ is bounded below, we show that our join $A \ast B$ is contained in the star of some vertex in $\kappa(P,\leq)$. From now on, we assume that neither $A$ nor $B$ is bounded below.

\medskip \noindent
Fix a $b \in B$. Since $b$ cannot be a lower bound of $A$, there must exist some $a \in A$ such that $a < b$. Since $a$ cannot be a lower bound of $B$, there must exist some $b' \in B$ such that $b'<a$. Since $b'$ cannot be a lower bound of $A$, there must exist some $a' \in A$ such that $a'<b'$. Thus, from a pair $a < b$, we have found an even smaller pair $a'<b'$. By iterating the argument, we contradict the fact that $P$ is finite. 
\end{proof}

\begin{lemma}\label{lem:ContractibleComp}
Let $(P,\leq)$ be a finite poset. For every subset $S \subset P$, $\kappa(\mathrm{Comp}(S),\leq)^\triangle$ is either empty or contractible, where 
$$\mathrm{Comp}(S):= \{ p \in P \mid p \text{ is $\leq$-comparable to all the points in } S\}.$$ 
\end{lemma}

\begin{proof}
First, assume that $S \cap \mathrm{Comp}(S) \neq \emptyset$. Fixing a point $s_0 \in S \cap \mathrm{Comp}(S)$, $\kappa(\mathrm{Comp}(S),\leq)$ decomposes as a join between $\{s_0\}$ and $\kappa(\mathrm{Comp}(S)\backslash \{s_0\}, \leq)$. Thus, $\kappa(\mathrm{Comp}(S), \leq)^\triangle$ is a cone. A fortiori, it is either empty or contractible, as desired.

\medskip \noindent
Next, assume that $S \cap  \mathrm{Comp}(S) = \emptyset$. Fix an arbitrary $s_0 \in S$ and decompose $\mathrm{Comp}(S)$ as the union of
$$C^-:= \{ p \in \mathrm{Comp}(S) \mid p \leq s_0\} \text{ and } C^+:= \{ p \in \mathrm{Comp}(S) \mid p \geq s_0\}.$$
Notice that, if $s \in C^- \cap C^+$, then $s \leq s_0 \leq s$ hence $s_0 = s \in S \cap  \mathrm{Comp}(S)$, contradicting our assumption $S \cap  \mathrm{Comp}(S) = \emptyset$. Thus, $C^-$ and $C^+$ are disjoint. Notice also that every element in $C^-$ is $\leq$-smaller than every element in $C^+$. Consequently, $\kappa(\mathrm{Comp}(S),\leq)$ decomposes as a join between $\kappa(C^-,\leq)$ and $\kappa(C^+,\leq)$. Notice that $C^-$ (resp.\ $C^+$) coincides with $\mathrm{Comp}(S\backslash \{s_0\})$ in the poset $(\{p \in P \mid p \leq s_0\},\leq)$ (resp.\ $(\{p \in P \mid p \geq s_0\},\leq)$), so, by arguing by induction, we can assume that $\kappa(C^-,\leq)^\triangle$ and $\kappa(C^+,\leq)^\triangle$ are empty or contractible. This implies that $\kappa(\mathrm{Comp}(S),\leq)^\triangle$ is also either empty or contractible, as desired. 
\end{proof}

\begin{proof}[Proof of Theorem~\ref{thm:LatticeHomotopy}.]
As a consequence of Lemma~\ref{lem:JoinLattice}, the complex $\kappa(P,\leq)^{\bowtie}$ coincides with nerve complex of $\kappa(P,\leq)^\triangle$ covered by $\{ \kappa(\mathrm{Comp}(p), \leq)^\triangle, p \in P\}$. According to Lemma~\ref{lem:ContractibleComp}, the Nerve Theorem applies and shows that the two simplicial complexes we are interested in are indeed homotopy equivalent. 
\end{proof}

\begin{cor}\label{cor:ManyBowtie}
Every finite simplicial complex $X$ is homotopy equivalent to 
$$\kappa(\{ \text{simplices of } X\}, \subseteq)^{\bowtie}.$$
\end{cor}

\begin{proof}
The complex $\kappa(\{ \text{simplices of } X\}, \subseteq)^\triangle$ coincides with the first barycentric subdivision of $X$, so it is homeomorphic to $X$. Then, the desired conclusion follows from Theorem~\ref{thm:LatticeHomotopy}. 
\end{proof}

\noindent
By combining Corollary~\ref{cor:ManyBowtie} with Theorem~\ref{thm:BowtieQIinv}, we obtain the following consequence, which proves Corollary~\ref{cor:InjMap}:

\begin{cor}\label{cor:ManyRAAG}
For every finitely presented group $G$, fix a finite simplicial complex $K_G$ whose fundamental group is $G$. Then, the map
$$\left\{ \begin{array}{ccc} \{\text{one-ended f.\ p.\ group}\} / \text{isom.} & \to & \{\text{right-angled Artin groups}\}/ \text{quasi-isometry} \\ G & \mapsto & A( \kappa( \{ \text{simplices of } K_G\}, \subset)) \end{array} \right.$$
is injective.
\end{cor}

\begin{proof}
Let $\Lambda$ denote the map given in the statement of our corollary. If, given two one-ended finitely presented groups $G$ and $H$, the right-angled Artin groups $\Lambda(G)$ and $\Lambda(H)$ are quasi-isometric, it follows from Theorem~\ref{thm:BowtieQIinv} that $\bigvee_\mathbb{N} K_G$ and $\bigvee_\mathbb{N} K_H$ are homotopy equivalent. Consequently, the fundamental groups of these two pointed sums, namely the free products $\ast_\mathbb{N} G$ and $\ast_\mathbb{N} H$ respectively, must be isomorphic. Since $G$ and $H$ are one-ended, this amounts to saying that $G$ and $H$ are isomorphic, as desired. 
\end{proof}

\noindent
The construction of a right-angled Artin group from a given finitely presented, one-ended group $G$ is quite flexible, as one can choose \textit{any} simplicial complex $K_G$ with fundamental group $G$. We give two families of examples to demonstrate this flexibility and prove Corollary~\ref{cor:InjMap}. 

\begin{ex}\label{ex:BowtieOne}
    Let $\{G_g\mid g\in \mathbb N\}$ be any collection of one-ended, finitely presented groups that are not isomorphic if $g\neq h$.  For concreteness, one could take, for example, $G_g$ to be the fundamental group of a surface of genus $g$ with $g\geq 2$.  Let $K_g'$ be a simplicial complex with fundamental group $G_g$, fix a vertex $v\in K_g'$, and let $K_g$ be formed from $K_g'$ by adding a single edge $\{v,w\}$ with the vertex $w$ of degree 1.  In the comparability graph $\Gamma_g=\kappa(\{\textrm{simplices of }K_g, \subseteq\})$, the star of $w$ does not contain $w$, and hence separates $\Gamma$.  Thus the corresponding right-angled Artin group $A(\Gamma_g)$ has infinite outer automorphism group \cite{HuangI} and is not of weak type I, in the sense of \cite{HuangII}.  Moreover, the intersection $\lk(v)\cap\lk(w)$ in $\Gamma_g$ is the edge $\{v,w\}$, which separates $\Gamma_g$, and so $A(\Gamma_g)$ is not type II, in the sense of \cite{HuangII}.  In particular, the groups $A(\Gamma_g)$ do not fit into Huang's framework \cite{HuangI,HuangII}.  By Corollary~\ref{cor:ManyRAAG}, however, $A(\Gamma_g)$ and $A(\Gamma_h)$ are not quasi-isometric when $g\neq h$. 
\end{ex}

\begin{ex}\label{ex:BowtieTwo}
    Let $\{G_g\mid g\in \mathbb N\}$ be any collection of one-ended, finitely presented groups that are not isomorphic if $g\neq h$. Embed the Cayley 2-complex $K_g$ of $G_g$  associated to a finite presentation into $\mathbb R^N$ for some sufficiently large $N$.  By taking a neighbourhood, we can assume that $K_g$ is a manifold with boundary.  Suppose $K_g$ is $d$--dimensional, and fix a $(d-1)$--simplex $S$ in the boundary $\partial K_g$, so that $S$ is  contained in a unique $d$--simplex $\sigma$.  In $\Gamma_g=\kappa(\{\textrm{simplices of }K_g, \subseteq\})$, the star of $S$ is contained in the star of $\sigma$, and so $\textrm{Out}(A(\Gamma_g))$ contains transvections and hence is infinite order.  Hence $A(\Gamma_g)$ does not fit into the framework of \cite{HuangI}.  Modify $K_g$ by adding a vertex $z$ and forming the join $z*S$; this adds one additional $d$--simplex $\tau$ to form a new complex $K_g'$ with the same fundamental group. In the corresponding comparability graph $\Gamma_g'$, the star of $S$ separates, and so $A(\Gamma_g')$ still has infinite outer automorphism group and is not of type I.  Moreover, $\lk(\sigma)\cap \lk(\tau)=\{S\}$, which disconnects $\Gamma_g'$, and so $A(\Gamma_g')$ is not of type II.  Hence the groups $A(\Gamma_g')$ do not fit into Huang's framework \cite{HuangI,HuangII}.
\end{ex}

\section{Open questions}

\noindent
In this final section, we record some questions that are naturally raised by the results proved so far.

\paragraph{Other relative contact complexes.} Given a quasi-median graph $X$ and a prism-covering collection $\mathbb{G}$ of gated subgraphs, we have identified the homotopy type of the $\mathbb{G}$-contact complex $\mathrm{Cont}(X,\mathbb{G})$ when $\mathbb{G}= \{ \text{prisms}\}$ (Theorem~\ref{thm:HomotopyCrossingComplex}) and when $\mathbb{G}$ is star-covering (Theorem~\ref{thm:RelContact}). Roughly speaking, we have 
$$\text{crossing complex} \subset \text{contiguity complex} \subset \text{contact complex},$$
and our results deal with crossing complexes and everything lying between contiguity and contact complexes. However, relative contact graphs lying between crossing and contiguity complexes remain to be investigated. 

\medskip \noindent
A difference between the proofs of Theorems~\ref{thm:HomotopyCrossingComplex} and~\ref{thm:RelContact} comes from the definitions of the subspaces associated to hyperplanes: in the first case, fibres are ignored, but they are included in our subspaces in the second case. In both cases, the key property to have is that two subspaces intersect if and only if the corresponding hyperplanes are adjacent in the relative contact complex under consideration.

\medskip \noindent
\begin{minipage}{0.3\linewidth}
\includegraphics[width=0.9\linewidth]{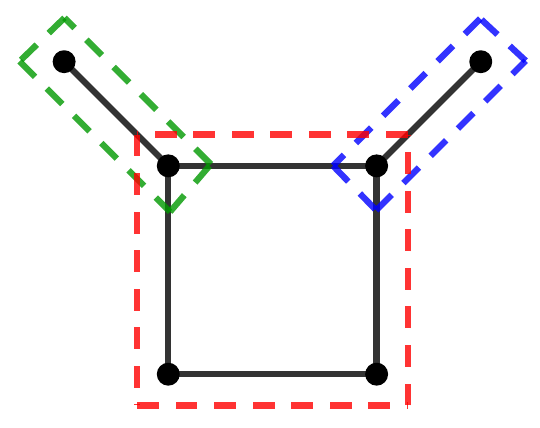}
\end{minipage}
\begin{minipage}{0.69\linewidth}
However, in full generality, three hyperplanes may have a fibre in common while only two of them are adjacent in the relative contact complex. As a consequence, one is tempted to both keep and remove the fibre from our subspaces. As a concrete example, our results do not allow us to identify the homotopy type of the coset intersection complex of the right-angled Artin group illustrated on the left.
\end{minipage}

\begin{problem}
Generalise Theorem~\ref{thm:RelContact} to arbitrary prism-covering collections. 
\end{problem}

\paragraph{Other quasi-isometric invariants.} Most of our results apply to graph products of infinite groups. The only step that is specific to right-angled Artin groups is Theorem~\ref{thm:RAAGmaxJoin}. Therefore, it is natural to ask:

\begin{question}
Let $\Gamma_1,\Gamma_2$ be two finite graphs and $\mathcal{G}_1,\mathcal{G}_2$ two collections of finitely generated infinite groups indexed by $\Gamma_1,\Gamma_2$. Given a quasi-isometry $\eta : \Gamma_1 \mathcal{G}_1 \to \Gamma_2 \mathcal{G}_2$, does there exist a constant $C \geq 0$ such that $\eta$ sends every coset of a maximal join subgroup of $\Gamma_1 \mathcal{G}_1$ at Hausdorff distance $\leq C$ from the coset of a maximal join subgroup of $\Gamma_2 \mathcal{G}_2$?
\end{question} 

\noindent
A positive answer is plausible. It would allow one to generalise Theorem~\ref{thm:IntroQI} to graph products of arbitrary infinite groups.

\medskip \noindent
In another direction, Theorem~\ref{thm:IntroQI} is obtained by considering coset intersection complexes relative to maximal join subgroups, but, in practice, there usually exist other collections of parabolic subgraphs that are preserved by quasi-isometries  (see Theorem~\ref{thm:RAAGCrossingQI}). Thus, our techniques may lead to other quasi-isometric invariants. This is, in our opinion, an interesting direction to investigate.

\medskip \noindent
For instance, the second-named author introduced in \cite{GPautacyl} the \emph{small crossing graph} of the quasi-median graph $\mathrm{QM}(\Gamma,\mathcal{G})$ associated to the graph product $\Gamma \mathcal{G}$ and noticed that, whenever $\Gamma$ does not have two vertices with the same link or the same star (which can always be assumed), the automorphism group $\mathrm{Aut}(\Gamma \mathcal{G})$ naturally acts on it. Is the small crossing graph also preserved by quasi-isometries?  Can we identify the homotopy type of its flag completion?

\medskip \noindent
More geometrically, given a quasi-median graph $X$, say a hyperplane is \emph{maximal} if none of its fibres is contained in the carrier of another hyperplane; and define the \emph{small crossing complex} $\mathrm{sCross}^\triangle(X)$ to be the simplicial complex whose vertices are the maximal hyperplanes of $X$ and whose simplices are given by pairwise transverse hyperplanes.

\begin{question}
Given a quasi-median graph $X$, what is the homotopy type of its small crossing complex $\mathrm{sCross}^\triangle(X)$?
\end{question}

\noindent
It is worth noticing that the proof of Proposition~\ref{prop:PuncturedQM} generalises, proving that $\mathrm{sCross}^\triangle(X)$ is homotopy equivalent to the complement in the prism-completion $X^\square$ of the edges contained in hyperplanes that are not maximal. Can this description be further simplified?

\paragraph{Right-angled Coxeter groups.} The characterisation of the homotopy type of coset intersection complexes provided by Theorem~\ref{thm:CosetInterComplexRelativeCont} only deals with graph products of infinite groups. It would be interesting to a have a similar result for graph products of finite groups, such as right-angled Coxeter groups. This would allow us to compare two right-angled Coxeter groups up to quasi-isometry, but also to compare right-angled Coxeter groups to right-angled Artin groups, which is a widely open problem. 

\begin{question}\label{question:GPfinite}
Let $\Gamma$ be a finite graph and $\mathcal{G}$ a collection of finite groups indexed by $\Gamma$. What is the homotopy type of the coset intersection complex
$$\mathcal{K} \left( \Gamma \mathcal{G}, \{ \langle \Lambda \rangle, \Lambda \subset \Gamma \text{ maximal large join} \} \right)?$$
\end{question}

\noindent
Here, we refer to a join $\Lambda_1 \ast \cdots \ast \Lambda_n$ as \emph{large} if each $\Lambda_i$ contains at least two non-adjacent vertices. Such coset intersection complexes also have a geometric interpretation, in the same vein as Lemma~\ref{lem:ForSkewer}. Given a quasi-median graph $X$ and a collection $\mathbb{G}$ of gated subgraphs, define the \emph{double skewering complex} $\mathrm{dSkew}(X,\mathbb{G})$ as the simplicial complex whose vertex-set is $\mathbb{G}$ and whose simplices are given by collections of subgraphs all crossed by two non-transverse hyperplanes. Then the proof of Lemma~\ref{lem:ForSkewer} generalises and shows that our coset intersection complexes can be described as double skewering complexes. However, identifying the homotopy type of a double skewering complex does not seem to be straightforward, and the problem remains to be investigated.

\medskip \noindent
Notice that a solution to Question~\ref{question:GPfinite} would already provide invariants allowing us to compare right-angled Coxeter groups with other right-angled Coxeter groups or with right-angled Artin groups up to commensurability. For a comparison up to quasi-isometry, another difficulty towards the generalisation of our results to graph products of finite groups is Theorem~\ref{thm:RAAGmaxJoin}. An extension of this quasi-isometric rigidity is plausible, though. In fact, we can ask much more generally:

\begin{question}
Let $G_1,G_2$ be two groups acting geometrically and specially on two (quasi-)median graphs $X_1,X_2$. Given a quasi-isometry $\eta : X_1 \to X_2$, does there exist a constant $C \geq 0$ such that $\eta$ sends every maximal product subgraph of $X_1$ at Hausdorff distance $\leq C$ from some maximal product subgraph of $X_2$?
\end{question}

\noindent
Theorem~\ref{thm:RAAGmaxJoin} gives a positive answer for median graphs of right-angled Artin groups and \cite{MR4365047} also yields a positive answer for median graphs of cubical dimension two.

\paragraph{Rigid homotopy equivalence.} Despite the fact that we consider coset intersection complexes up to homotopy, Corollary~\ref{cor:CosetInterComplexJoinQI} shows that we do not only have a homotopy equivalence between our complexes but a genuine isomorphism. Extracting metric (compared to homotopy) invariants from such an isomorphism would be interesting. In this direction:

\begin{problem}
Let $X$ be a quasi-median graph and $\mathbb{G}$ a collection of gated subgraphs. Find metric connections between $\mathrm{Cont}(X,\mathbb{G})$ and the simplicial graphs $\mathrm{sL}_\mathbb{G}(x)$, $x \in X$.
\end{problem}

\noindent
For instance, it would be interesting to compare up to quasi-isometry two right-angled Artin groups $A(\Gamma_1), A(\Gamma_2)$ where $\Gamma_1$ is ``made of'' cycles of length $5$ while $\Gamma_2$ is ``made of'' cycles of length $6$.

\medskip \noindent
As justified in the introduction, we focus on the homotopy types of our coset intersection complexes because they are much easier to identify than their metric types. But, of course, we lose a lot of information in the process. It would be interesting to identify a relevant intermediate notion of equivalence between isomorphism and homotopy equivalences, keeping more information than homotopy equivalence but easier to identify than isomorphism.

\begin{question}
Is there a ``more rigid'' notion of homotopy equivalence such that the rigid homotopy equivalence type of relative contact complexes of quasi-median graphs can be identified?
\end{question}

\noindent
For instance, let $\Gamma_1,\Gamma_2,\Gamma_3$ denote the three graphs below. Homotopically, $\bigvee_\mathbb{N} \Gamma_1^{\bowtie}$, $\bigvee_\mathbb{N} \Gamma_2^{\bowtie}$, and $\bigvee_\mathbb{N} \Gamma_3^{\bowtie}$ are identical. However, $A(\Gamma_1)$ is not quasi-isometric to $A(\Gamma_2)$ but it is quasi-sometric to $A(\Gamma_3)$. (In fact, $A(\Gamma_3)$ can be realised as a subgroup of index two in $A(\Gamma_1)$.) It would be interesting to distinguish $\Gamma_2^{\bowtie}$ and $\Gamma_3^{\bowtie}$ up to ``rigid homotopy.'' 

\begin{center}
\includegraphics[width=0.8\linewidth]{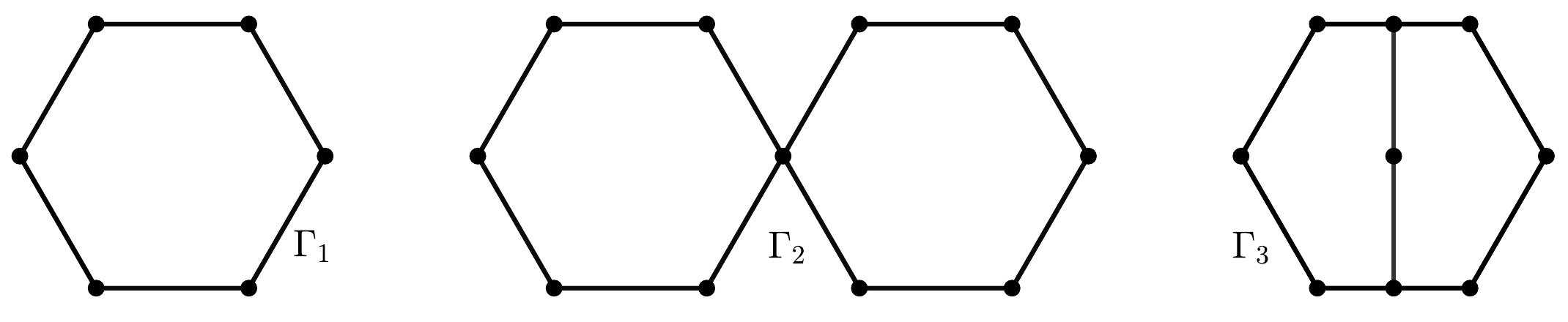}
\end{center}



\addcontentsline{toc}{section}{References}

\bibliographystyle{alpha}
{\footnotesize\bibliography{CrossingRAAG}}

\Addresses

%

\end{document}